\documentclass[smallextended]{svjour3}       % onecolumn (second format)
\smartqed  % flush right qed marks, e.g., at end of proof

\usepackage{graphicx}
\usepackage{amsmath,amssymb}
\usepackage{url,hyperref}
\usepackage[usenames,dvipsnames]{xcolor}
\usepackage{xfrac}
\usepackage{booktabs,multirow,paralist} % Horizontal rules in tables
\usepackage{algorithm}
\usepackage{algpseudocode}
\usepackage{pifont}
\usepackage[letterpaper, margin=1.5in]{geometry}
\usepackage{afterpage}
\usepackage{wrapfig}
\usepackage{makecell}

\newtheorem{assumption}{Assumption}%[section]

\newcommand{\Hspace}{{\mathbb{R}^p}}
\newcommand{\Id}{\mathbb{I}}
\newcommand{\R}{\mathbb{R}}

\newcommand{\set}[1]{\left\{#1\right\}}
\newcommand{\sets}[1]{\{#1\}}
\newcommand{\norm}[1]{\left\Vert#1\right\Vert}
\newcommand{\norms}[1]{\Vert#1\Vert}

\newcommand{\Eproof}{\hfill $\square$}

\newcommand{\dom}[1]{\mathrm{dom}(#1)}
\newcommand{\gra}[1]{\mathrm{gra}(#1)}
\newcommand{\iprod}[1]{\left\langle #1\right\rangle}
\newcommand{\iprods}[1]{\langle #1\rangle}
\newcommand{\Exp}[1]{\mathbb{E}\left[#1\right]}
\newcommand{\Exps}[2]{\mathbb{E}_{#1}\left[#2\right]}
\newcommand{\Expsk}[2]{\mathbb{E}_{#1}\big[#2\big]}
\newcommand{\Expsn}[2]{\mathbb{E}_{#1}[#2]}

\newcommand{\Expn}[1]{\mathbb{E}\big[#1\big]}

\newcommand{\Bc}{\mathcal{B}}
\newcommand{\Xc}{\mathcal{X}}

\newcommand{\Sc}{\mathcal{S}}

\newcommand{\Dc}{\mathcal{D}}
\newcommand{\Lc}{\mathcal{L}}
\newcommand{\Gc}{\mathcal{G}}

\newcommand{\Tc}{\mathcal{T}}
\newcommand{\Rc}{\mathcal{R}}

\newcommand{\Fc}{\mathcal{F}}

\newcommand{\BigO}[1]{\mathcal{O}\left(#1\right)}
\newcommand{\BigOs}[1]{\mathcal{O}\big(#1\big)}
\newcommand{\SmallO}[1]{o\left(#1\right)}
\newcommand{\SmallOs}[1]{o\big(#1\big)}

\newcommand{\zer}[1]{\mathrm{zer}(#1)}

\newcommand{\mbf}[1]{\mathbf{#1}}
\newcommand{\mcal}[1]{\mathcal{#1}}
\newcommand{\beforesec}{\vspace{-3.5ex}}
\newcommand{\aftersec}{\vspace{-2.25ex}}
\newcommand{\beforesubsec}{\vspace{-4ex}}
\newcommand{\aftersubsec}{\vspace{-2.5ex}}
\newcommand{\beforesubsubsec}{\vspace{-2.5ex}}
\newcommand{\aftersubsubsec}{\vspace{-2.5ex}}

\begin{document}

\title{
A Class of Accelerated Fixed-Point-Based Methods with Delayed Inexact Oracles and Its Applications
}

\titlerunning{
A Class of Accelerated Fixed-Point-Based Methods with Delayed Inexact Oracles
}        % if too long for running head

\author{Nghia Nguyen-Trung \and Quoc Tran-Dinh}
\authorrunning{N. Nguyen-Trung \and Q. Tran-Dinh}


\institute{Nghia Nguyen-Trung \and Q. Tran-Dinh \at
		Department of Statistics and Operations Research\\
		The University of North Carolina at Chapel Hill, 318 Hanes Hall, Chapel Hill, NC 27599-3260.\\ 
		Email: \url{nghiant@unc.edu, quoctd@email.unc.edu}.
}

\date{}%Version 4: The first manuscript  was posted on Arxiv on March 13, 2019}

\maketitle

\begin{abstract}
In this paper, we develop a novel accelerated fixed-point-based framework using delayed inexact oracles to approximate a fixed point of a nonexpansive operator (or equivalently, a root of a co-coercive operator), a central problem in scientific computing.
Our approach leverages both Nesterov's acceleration technique and the Krasnosel'ski\v{i}-Mann (KM) iteration, while accounting for delayed inexact oracles, a key mechanism in asynchronous algorithms.
We also introduce a unified approximate error condition for delayed inexact oracles, which can cover various practical scenarios. 
Under mild conditions and appropriate parameter updates, we establish both $\BigOs{1/k^2}$ non-asymptotic and $\SmallOs{1/k^2}$ asymptotic convergence rates in expectation for the squared norm of residual.
Our rate significantly improves the $\BigOs{1/k}$ rates in classical KM-type methods, including its asynchronous variants.
We also establish $\SmallOs{1/k^2}$ almost sure convergence rates and the almost sure convergence of iterates to a solution of the problem.
Within our framework, we instantiate three settings for the underlying operator: (i) a deterministic universal delayed oracle; (ii) a stochastic delayed oracle; and (iii) a finite-sum structure with asynchronous updates.
For each case, we instantiate our framework to obtain a concrete algorithmic variant for which our convergence results still apply, and whose iteration complexity depends linearly on the maximum delay.
Finally, we verify our algorithms and theoretical results through two numerical examples on both matrix game and shallow neural network training problems.
\end{abstract}


% Main paper.
%%%%%%%%%%%%%%%%%%%%%%%%%%%%%%%%%%%%%%%%%%%%%%%%%
%% 1. Introduction.
%%%%%%%%%%%%%%%%%%%%%%%%%%%%%%%%%%%%%%%%%%%%%%%%%
%%%%%%%%%%%%%%%%%%%%%%%%%%%%%%%%
%%%% 1. Introduction.
%%%%%%%%%%%%%%%%%%%%%%%%%%%%%%%%
\vspace{1ex}
\beforesec
\section{Introduction}\label{sec:intro}
\aftersec
%\beforesubsec
%\subsection{Problem statement}
%\aftersubsec
\noindent\textbf{$\mathrm{(a)}$~Problem statement.}
Let $F : \Hspace \to \Hspace$ be a nonexpansive  operator (i.e., $\norms{Fx - Fy} \leq \norms{x - y}$ for all $x, y \in \Hspace$).
We consider the following fixed-point problem:
\begin{equation}\label{eq:FP}
\textrm{Find $x^{\star}\in \Hspace$ such that:}~ x^{\star} = Fx^{\star}.
\tag{FP}
\end{equation}
This problem is fundamental in scientific computing and related fields such as optimization, machine learning, engineering, probability and statistics, and economics, see, e.g.,  \cite{agarwal2001fixed,Bonsall1962,Combettes2011a}.

It is well-known (e.g., \cite{Bauschke2011}) that if we define $G := \Id - F$, where $\Id$ is the identity mapping, then \eqref{eq:FP} can be reformulated equivalently into the following \textit{root-finding problem} of $G$:
\begin{equation}\label{eq:CE}
\textrm{Find $x^{\star} \in \Hspace$ such that:} \ Gx^{\star} = 0,
\tag{CE}
\end{equation}
where $G$ is $\frac{1}{2}$-co-coercive, i.e., $\iprods{Gx - Gy, x - y} \geq \frac{1}{2}\norms{Gx - Gy}^2$ for all $x, y \in \Hspace$.
Therefore, approximating a fixed-point $x^{\star}$ of \eqref{eq:FP} is equivalent to finding an approximate root of \eqref{eq:CE}.
For the sake of our presentation, we will treat the latter problem, i.e., \eqref{eq:CE}.
While we consider \eqref{eq:CE} in a finite-dimensional space $\R^p$, we believe that our results in this paper can be extended to a Hilbert space with a little effort. 

Throughout this paper, we require the following assumption on \eqref{eq:CE}.
\begin{assumption}\label{as:A1}
The equation \eqref{eq:CE} satisfies the following assumptions:
\begin{compactitem}
\item[$\mathrm{(i)}$] 
It has a solution, i.e., $\zer{\Phi} := \sets{x \in \Hspace \mid Gx^{\star} = 0} \neq\emptyset$.

\item[$\mathrm{(ii)}$]
There exist $\beta, \bar{\beta} \in [0, +\infty)$ and a nonnegative function $\Dc : \Hspace \times \Hspace \to \R_{+}$ satisfying $\Dc(x,x) = 0$ for all $x \in \Hspace$ such that
\begin{equation}\label{eq:G_cocoercivity}
\iprods{Gx - Gy, x - y}  \geq  \beta \norms{Gx - Gy}^2 + \bar{\beta} \Dc(x, y), \quad \forall x, y \in \Hspace.
\end{equation}
\end{compactitem}
\end{assumption}
While Assumption~\ref{as:A1}(i) makes sure that \eqref{eq:CE} is solvable and well-defined,  Assumption~\ref{as:A1}(ii) generalizes the $\beta$-co-coercivity of $G$ by adding the last term $\bar{\beta} \Dc(x, y)$, which will be specified later in our applications.
Specifically, in the finite-sum setting of \eqref{eq:CE} with $Gx = \frac{1}{n}\sum_{i=1}^nG_ix$, we will choose $\Dc(x, y) := \frac{1}{n}\sum_{i=1}^n\norms{G_ix - G_iy}^2$.
Note that, we allow either $\beta = 0$ or $\bar{\beta} = 0$ in \eqref{eq:G_cocoercivity}.
Under this flexibility, it is clear that if $G$ is $\beta$-co-coercive on $\Hspace$, then it automatically satisfies \eqref{eq:G_cocoercivity} with either $\bar{\beta} = 0$ or $\Dc(\cdot, \cdot) = 0$.
Since the co-coercivity of $G$ in \eqref{eq:CE} is equivalent to the nonexpansiveness of $F$ in \eqref{eq:FP}, Assumption~\ref{as:A1}(ii) is reasonable.

%\beforesubsec
%%\subsection{Our motivation}
%\aftersubsec
\vspace{0.5ex}
\noindent\textbf{$\mathrm{(b)}$~Motivation.}
In parallel and distributed  optimization methods, enforcing \textit{perfectly up-to-date oracle} evaluations such as [approximate] gradient or operator evaluations is often prohibitively expensive, and in many cases, infeasible \cite{Andrews2000,Bertsekas1989b,Schnitger2006}. 
Both exact and approximate synchronous oracle evaluations require global coordination across workers, memory systems, and data partitions, which causes significant idle time and communication overhead in heterogeneous computing environments. 
These limitations motivate the development of algorithms that tolerate delayed or asynchronous inexact oracle information rather than requiring workers to remain perfectly aligned to up-to-date oracle evaluations.
 
Hitherto, most advanced  asynchronous methods have been designed for minimization problems; see, for example, \cite{aviv2021asynchronous,Bertsekas1989b,cohen2021asynchronous,feyzmahdavian2023asynchronous,hannah2019a2bcd,Leblond2018a,leblond2017asaga,liu2015asynchronous,ManiaPanPapailiopoulosEtAl2015,Recht2011,tseng1991rate,Wei20131}. 
Such algorithms do not directly accommodate modern applications in robust and distributionally robust optimization, adversarial training, or reinforcement learning, which give rise to minimax or fixed-point formulations. 
Only a limited number of works address more general models such as \eqref{eq:CE} and their special cases, including minimax problems; see \cite{bertsekas1983distributed,chow2017cyclic,combettes2018asynchronous,davis2016smart,hannah2018unbounded,peng2016arock}. 
These approaches largely extend classical asynchronous optimization schemes to fixed-point problems or monotone inclusions.
However, these methods are often designed for a specific type of asynchronous mechanism, such as asynchronous coordinate updates \cite{bertsekas1983distributed,combettes2018asynchronous,hannah2018unbounded,peng2016arock} or incremental aggregation  \cite{davis2016smart}.  
In addition, these methods are non-accelerated, often require restrictive assumptions (e.g., quasi-strong monotonicity), achieve sub-optimal convergence rates (at most $\BigOs{1/k}$ or $\SmallOs{1/k}$ rates), or have inefficient complexity bounds (e.g., quadratic dependence on the maximum delay), see, e.g., \cite{davis2016smart,hannah2018unbounded,peng2016arock}. 
For instance, \cite{bertsekas1983distributed,combettes2018asynchronous} studied asynchronous fixed-point-type iterations for \eqref{eq:FP} or \eqref{eq:CE}, but they are non-accelerated and only asymptotic convergence was established without explicit iteration or oracle complexity bounds.
Recently, \cite{peng2016arock} proposed an asynchronous block-coordinate fixed-point iteration for \eqref{eq:CE} that attains asymptotic convergence properties under Assumption~\ref{as:A1}, or a linear convergence but requires a quasi-strong monotonicity of $G$, and its complexity could depend quadratically on the maximum delay. 
This dependence was later improved to linear in subsequent works, e.g., \cite{feyzmahdavian2023asynchronous}. 
Recently, \cite{davis2016smart} developed asynchronous aggregated schemes for solving \eqref{eq:CE} extended to the finite-sum setting and obtained similar convergence and complexity results using analogous block-coordinate updates and assumptions as in \cite{peng2016arock}. 

The aim of this paper is to develop a {\textit{unified accelerated fixed-point framework}} for \eqref{eq:CE} capable of handling delayed and inexact oracle evaluations and adaptable to a broad range of asynchronous update mechanisms without quasi-strong monotonicity. 
Our approach integrates Nesterov-type acceleration (different from previous works for \eqref{eq:CE}), Krasnosel'ski\v{i}-Mann iterations, and delayed inexact oracle information in a single scheme, while not using the ``quasi-strong monotonicity'' of $G$ as in  \cite{davis2016smart,peng2016arock}. 
As a result, our framework yields several improvements and new convergence properties compared to existing methods, even in special cases such as smooth convex  optimization and convex-concave minimax problems.

\vspace{0.5ex}
\noindent\textbf{$\mathrm{(c)}$~Our contributions.}
To this end, our contributions can be summarized as follows.
\begin{compactitem}
\item[$\mathrm{(i)}$] 
We apply Nesterov's acceleration techniques to the Krasnose{l}'ski\v{i}-Mann fixed-point iteration to derive a new \textbf{A}ccelerated \textbf{F}ixed-\textbf{P}oint-type framework (\ref{eq:iFKM4CE}) with delayed inexact oracles (deterministic or stochastic) to approximate solutions of \eqref{eq:CE}. 

\item[$\mathrm{(ii)}$] 
We introduce a generic \textit{error approximation criterion} to handle delayed inexact oracles of $G$.
Under this criterion, we can prove a $\BigOs{1/k^2}$ convergence rate of our \ref{eq:iFKM4CE} framework on the residual $\mathbb{E}[ \norms{Gy^k}^2 ]$, where $k$ is the iteration counter and $y^k$ is the last iterate.
We also establish a faster $\SmallO{1/k^2}$ convergence rate of $\norms{Gy^k}^2$ both in expectation and almost surely, and the almost sure convergence of iterates $\sets{y^k}$ to a solution of \eqref{eq:CE}.

\item[$\mathrm{(iii)}$] 
We apply our framework to derive an \ref{eq:iFKM4CE} variant with delay updates and achieve an iteration-complexity of $\BigOs{\frac{\tau}{\beta\epsilon}}$ to obtain an $\epsilon$-solution of \eqref{eq:CE}, where $\tau$ is the maximum delay and $\beta$ is the co-coercivity constant in Assumption~\ref{as:A1}.

\item[$\mathrm{(iv)}$] 
We also apply our method to design an asynchronous stochastic \ref{eq:iFKM4CE} variant for approximating a solution of \eqref{eq:CE} when $G$ is equipped with an unbiased stochastic oracle $\mbf{G}$ such that $Gx = \mathbb{E}[ \mbf{G}(x, \xi) ]$.
We prove that our method can achieve an $\epsilon$-solution after $\BigOs{\frac{\sigma}{\epsilon\sqrt{\tau}} + \frac{\tau}{\beta \epsilon}}$ oracle calls, where $\tau$ is the maximum delay and $\sigma^2$ is the variance of $\mbf{G}$.

\item[$\mathrm{(v)}$] 
As byproducts of our results, we also obtain two new variants of our \ref{eq:iFKM4CE} framework.
The first one is an incremental aggregated \ref{eq:iFKM4CE} scheme, while the second variant is a shuffling aggregated \ref{eq:iFKM4CE} algorithm.
Both variants have rigorous theoretical convergence rate guarantees as consequences of our unified convergence analysis. 
\end{compactitem}
\noindent\textbf{\textit{Algorithmic and analysis highlights.}}
First, in the absence of inexactness, our \ref{eq:iFKM4CE} scheme reduces to an exact deterministic version for solving \eqref{eq:CE}. 
Due to a different parameter update, this method is not identical to, but closely related to, Nesterov’s accelerated scheme in \cite{tran2022connection} and the fast KM method in \cite{bot2022bfast}. 
However, the convergence analyses in \cite{bot2022bfast,tran2022connection} are not well suited for developing inexact variants of \eqref{eq:CE}. 
The difficulty arises from bounding inner-product terms such as $\langle e^k, y^{k+1} - y^k\rangle$ and $\langle e^k, y^k - x^{\star}\rangle$, where $y^k$ is the iterate, $x^{\star} \in \zer G$, and $e^k = \widetilde{G}^k - Gy^k$ is the error between the inexact oracle $\widetilde{G}^k$ and the true value $Gy^k$. 
Processing these terms is non-trivial in convergence analysis.
This limitation motivates us to develop a new scheme whose analysis avoids such terms entirely.

%%%%
Second, although our method may initially appear similar to Nesterov’s accelerated scheme in \cite{lan2011primal,Nesterov2005c} for smooth convex minimization, the underlying analysis differs substantially. 
In convex optimization, convergence analysis typically uses the objective residual $f(x^k) - f(x^{\star})$ as part of the underlying Lyapunov function \cite{bansal2017potential}. 
For general operator problems such as \eqref{eq:CE}, no such objective exists, and the analysis instead relies on constructing an appropriate Lyapunov function -- a nontrivial task in this context.

%%%
Third, our method incorporates acceleration together with delayed inexact oracles, setting it apart from existing inexact KM-type schemes such as \cite{bravo2019rates,cortild2025krasnoselskii,liang2016convergence}. 
It is also different from recent non-accelerated paradigms in \cite{combettes2018asynchronous,davis2016smart,peng2016arock}, where the analysis often uses a standard metric such as $\norms{x^k - x^{\star}}^2$ to construct a Lyapunov function.
They often require stronger assumptions for non-asymptotic rates, such as quasi-strong monotonicity \cite{davis2016smart,peng2016arock}.

%%%
Fourth, in the stochastic setting, our approach also differs from recent variance-reduced fast KM methods \cite{tran2024accelerated,TranDinh2025a} due to the new algorithmic structure, convergence metrics, and the generic delayed error approximation condition introduced in Definition~\ref{de:error_bound_cond}. 

%%%
Finally, our framework is sufficiently general to accommodate additional variants beyond the three settings studied in Section~\ref{sec:app_of_iFKM}.
It can also be applied to solve more general classes of problems than \eqref{eq:CE}, such as monotone or co-hypomonotone inclusions, by applying appropriate reformulation techniques as in \cite{TranDinh2025a} (see Remark \ref{re:extension}).

%\todo{Detail later}

%\beforesubsec
%\subsection{Related work}
%\aftersubsec
\vspace{0.5ex}
\noindent\textbf{$\mathrm{(d)}$~Related work and comparison.}
Since \eqref{eq:FP} or its equivalent form \eqref{eq:CE} is classical and fundamental, its theory, numerical methods, and applications have been widely developed for many decades, including the following monographs \cite{Bauschke2011,Rockafellar1997,GoebelKirk1990}.
Here, we only discuss and compare the most relevant work to our setting and methods studied in this paper.

\vspace{0.5ex}
\noindent\textit{$\mathrm{(i)}$~Fixed-point methods and its accelerated  variants.}
Both the Banach-Picard (BP) iteration and the Krasnosel'ski\v{i}-Mann (KM) scheme are classical approaches for approximating a solution of \eqref{eq:FP} or, equivalently, \eqref{eq:CE}.
The BP iteration requires the operator $F$ to be contractive, whereas the KM scheme only assumes that $F$ is nonexpansive, and therefore applies to a broader class of algorithms --- particularly those arising in monotone operator theory such as proximal-point and splitting schemes; see, for example, \cite{Bauschke2011,Facchinei2003,minty1962monotone,ryu2016primer}.
While the BP method achieves a linear convergence rate due to the contractivity of $F$, the KM scheme guarantees at most $\BigOs{1/k}$ or $\SmallOs{1/k}$  rates in terms of the squared residual norm under mere nonexpansiveness, which is far from the $\BigOs{1/k^2}$-optimal rate \cite{partkryu2022}.

Boosting the convergence rates of classical fixed-point methods has become an attractive research topic in recent years.
Both Nesterov’s acceleration and Halpern’s fixed-point iteration have been extensively studied for solving \eqref{eq:FP} and \eqref{eq:CE}; see, for example, \cite{attouch2020convergence,attouch2019convergence,Beck2009,bot2023fast,chambolle2015convergence,diakonikolas2020halpern,kim2021accelerated,lieder2021convergence,mainge2021accelerated,partkryu2022,qi2021convergence,sabach2017first,tran2022connection,tran2021halpern,tran2024revisiting2,yoon2021accelerated}.
Depending on the assumptions and parameter update rules, these methods achieve convergence rates as fast as $\BigOs{1/k^2}$ or even $\SmallOs{1/k^2}$.
These rates are substantially improved over the non-accelerated counterparts by a factor of $1/k$ and often match the optimal rate up to a constant factor in many settings \cite{lieder2021convergence,park2022exact}.

Hitherto, however, most accelerated algorithms of this type are deterministic and rely on exact oracles.
Only a limited number of works address inexact oracles, with most focusing on stochastic oracles or randomized coordinate updates; see, e.g., \cite{cai2023variance,cai2022stochastic,tran2022accelerated,tran2024accelerated,TranDinh2025a}.
As observed in convex optimization \cite{Devolder2010}, accelerated methods tend to be more sensitive to oracle errors than their non-accelerated counterparts.
In addition, most convergence analysis of inexact methods relies on inexact oracles defined through the objective function. 
Unlike convex optimization, here the absence of an objective function in \eqref{eq:FP} makes the development of Nesterov-type inexact accelerated fixed-point algorithms particularly challenging.

\vspace{0.5ex}
\noindent\textit{$\mathrm{(ii)}$~Optimization algorithms with delayed or asynchronous updates.}
Asynchronous algorithms have been extensively developed  in the literature, primarily within the domain of convex optimization and numerical linear algebra.
The seminal monograph \cite{Bertsekas1989b} arguably laid the theoretical foundation for parallel and distributed computation, covering both synchronous and asynchronous algorithmic mechanisms.
The recent surge of interest in solving large-scale problems on parallel and distributed computing platforms has revived and greatly expanded research on asynchronous methods.
A major milestone in this renewed interest is \cite{Recht2011}, introducing HOGWILD!, which established near-optimal convergence guarantees for lock-free stochastic gradient descent.
Subsequent theoretical analyses have typically demonstrated convergence degradation proportional either to the maximum delay \cite{agarwal2011distrubuted,arjevani2020Dtight,feyzmahdavian2016asynchronous,lian2015asynchronous,ManiaPanPapailiopoulosEtAl2015,stich2020error} or to the average delay \cite{aviv2021asynchronous,cohen2021asynchronous}.
The work of \cite{feyzmahdavian2023asynchronous} refined this characterization by replacing the maximum delay with a threshold parameter and established a non-accelerated 
$\BigOs{1/k}$ best-iterate convergence rate in expected optimality gap.
The ASAGA algorithm \cite{leblond2017asaga} achieved linear convergence for finite-sum, smooth, and strongly convex problems, with explicit step-size and constant dependencies on the bounded overlap (delay) parameter, matching the serial SAGA performance up to delay-dependent factors and allowing near-linear speedups.
Similarly, A2BCD \cite{hannah2019a2bcd}, a Nesterov's accelerated block-coordinate method, attained linear convergence with optimal complexity under moderate delay conditions.
Nevertheless, extending this approach to accelerated fixed-point methods remains largely open.

\vspace{0.5ex}
\noindent\textit{$\mathrm{(iii)}$~Fixed-point methods with delayed or asynchronous updates.}
Unlike optimization, asynchronous algorithms for fixed-point problems such as \eqref{eq:FP} and generalized equations such as \eqref{eq:CE} remain comparatively underdeveloped.
While most early methods were presented in \cite{Bertsekas1989b}, convergence rates and complexity analysis similar to those in optimization remain largely elusive. 
Remarkably, the work \cite{peng2016arock} systematically studied a class of asynchronous coordinate schemes for \eqref{eq:CE} based on the KM iteration, called ARock,  allowing inconsistent reads and finite maximal delays, and establishing almost sure convergence (and linear rates under an additional quasi-strong monotonicity assumption) with delay-dependent step sizes.
Subsequent variants of ARock can be found, e.g., in \cite{hannah2018unbounded,heaton2019asynchronous}.
The work \cite{davis2016smart} extended ARock's idea and developed asynchronously incremental aggregated schemes for a finite-sum setting of \eqref{eq:CE} and obtained similar convergence results as \cite{peng2016arock} under analogous assumptions.
For monotone inclusions, \cite{combettes2018asynchronous} proposed an asynchronous block-iterative primal-dual splitting method with weak and strong asymptotic convergence guarantees.
These methods, however, are non-accelerated, highlighting the need for a general, non-asymptotic, and accelerated framework capable of accommodating asynchrony in both nonexpansive mappings and monotone equations or inclusions.

%%%%%%
\vspace{0.5ex}
\noindent\textbf{$\mathrm{(e)}$~Paper outline.}
The rest of this paper is organized as follows.
Section~\ref{sec:background} reviews the necessary background and mathematical tools used in the sequel.
Section~\ref{sec:iFKM_framework} develops our accelerated fixed-point framework with delayed inexact oracles and analyzes its convergence under a general error approximation condition.
Section~\ref{sec:app_of_iFKM} specifies our framework for three concrete computing models, and establishes their convergence rates accordingly.
Section~\ref{sec:num_experiments} presents two numerical examples to verify our methods and theoretical results.
For clarity of presentation, we defer some technical proofs to the appendix.

\beforesec
\section{Background and Mathematical Tools}\label{sec:background}
\aftersec
In this section, we recall the necessary concepts and mathematical tools used in the sequel.
We also briefly review Nesterov's accelerated gradient method for minimizing a smooth and convex function as a starting point of our development in this paper.

\beforesubsec
\subsection{\textbf{\textit{Basic concepts}}}\label{subsec:basic_concepts}
\aftersubsec
\textbf{$\mathrm{(a)}$~Basic notations.} 
We work with a finite-dimensional space $\R^p$ equipped with an inner product $\iprods{\cdot, \cdot}$, inducing the norm $\norms{\cdot}$. 
For a mapping $G$ (single-valued or multivalued mapping), we denote by $\dom{G} := \sets{ x \in \Hspace: Gx \neq \emptyset}$ the domain of $G$ and by $\gra{G} := \sets{(x, u) \in \Hspace \times\Hspace : u \in Gx}$ the graph of $G$.
For given functions $g(t)$ and $h(t)$, we say that $g(t) = \BigOs{h(t)}$ if there exists $M > 0$ and $t_0 \geq 0$ such that $g(t) \leq M h(t)$ for $t \geq t_0$, and we say that $g(t) = o(h(t))$ if $\lim_{t \to \infty} \frac{g(t)}{h(t)} = 0$.

\vspace{0.75ex}
\noindent\textbf{$\mathrm{(b)}$~Filtration and expectations.} 
We define by $\Fc_k$ the $\sigma$-algebra generated by all the randomness induced by $x^0, x^1, \dots, x^k$, up to the $k$-th iteration.
This $\sigma$-algebra generates a filtration $\sets{\Fc_k}$.
We denote by $\Expsn{k}{\cdot} = \Expn{\cdot \mid \Fc_k}$ the conditional expectation w.r.t. $\Fc_k$, by $\Expsn{\xi}{\cdot} = \Expsn{\xi\sim\mathbb{P}}{\cdot}$ the expectation w.r.t. $\xi$, and by $\Expn{\cdot}$ the total expectation.

\vspace{0.75ex}
\noindent\textbf{$\mathrm{(c)}$~Co-coercivity, Lipschitz continuity, and nonexpansiveness.} 
For a single-valued operator $G: \Hspace \to \Hspace$, we say that $G$ is $\beta$-co-coercive if there exists $\beta \geq 0$ such that $\iprods{Gx - Gy, x - y} \geq \beta\norms{Gx - Gy}^2$ for all $x, y \in \Hspace$;
we say that $G$ is $L$-Lipschitz continuous if there exists $L \geq 0$ such that $\norms{Gx - Gy} \leq L \norms{x - y}$ for all $x, y \in \dom{G}$.
If $L=1$, then $G$ is nonexpansive.
If $Gx = \nabla{f}(x)$, the gradient of a convex function $f$, then the $L$-Lipschitz continuity of $\nabla{f}$ is equivalent to the $\frac{1}{L}$-co-coercivity of $\nabla{f}$ due to the well-known Baillon--Haddad theorem \cite{Bauschke2011}.

\vspace{0.75ex}
\noindent\textbf{$\mathrm{(d)}$~Monotonicity and co-hypomonotonicity.}
For a possibly multivalued operator $G: \Hspace \rightrightarrows \Hspace$, we say that $G$ is $\rho$-co-hypomonotone if $\iprods{w_x - w_y, x - y} \geq -\rho\norms{w_x - w_y}^2$ for all $(x, w_x), (y, w_y) \in \gra{G}$ and $\rho \geq 0$.
If $\rho = 0$, then we say that $G$ is monotone. 
A co-hypomonotone operator is not necessarily monotone, see, e.g., \cite{TranDinh2025a} for concrete examples.

%\vspace{0.75ex}
%\noindent\textbf{$\mathrm{(e)}$~Resolvent operators.}
%Given a set-valued operator $T: \Hspace \rightrightarrows \Hspace$, the operator $J_{T}x := \sets{y \in \R^p: x \in y + Ty}$ is called the resolvent of $T$, denoted by $J_{T}x = (\Id + T)^{-1}x$, where $\Id$ is the identity operator.
%If $T = \partial{g}$, the subdifferential of a proper, closed, and convex function $f$, then $J_Tx = (\Id + \partial{g})^{-1}(x) =  \mathrm{prox}_g(x)$, the proximal operator of $g$.

\beforesubsec
\subsection{\textit{\textbf{Nesterov's accelerated gradient method and its variants}}}\label{subsec:Nesterov_method}
\aftersubsec
We recall here one  common variant of Nesterov's accelerated gradient method \cite{Nesterov2004} for minimizing a convex and $L$-smooth function $f$ from \cite{lan2011primal}.
This method  starts from an arbitrary $x^0 \in \R^p$, chooses $y^0 := x^0$,  and at iteration $k \geq 0$, updates:
\begin{equation}\label{eq:NAG}
\arraycolsep=0.2em
\left\{\begin{array}{lcl}
y^k &:= & \frac{t_k - s}{t_k}x^k + \frac{s}{t_k}z^k, \vspace{1ex}\\
x^{k+1} &:= & y^k - \eta\nabla{f}(y^k), \vspace{1ex}\\
z^{k+1} &:= & z^k + t_k(x^{k+1} - y^k),
\end{array}\right.
\tag{NAG}
\end{equation}
where $s > 0$ and $\eta \in (0, \frac{1}{L}]$ are given, and $t_k =  \BigOs{k}$ can be adaptively updated.
%This variant generates a triple $(x^k, z^k, y^k)$ by using one gradient $\nabla{f}(y^k)$ of $f$ evaluated at the intermediate point $y^k$.

In the unconstrained setting, \eqref{eq:NAG} can be expressed equivalently to several well-known variants of Nesterov’s accelerated gradient scheme, including those presented in \cite{Nesterov1983,Nesterov2004}.
With an appropriate choice of the parameters $s$ and $t_k$, the method attains both $\BigOs{1/k^2}$ and $\SmallOs{1/k^2}$ convergence rates for the objective residual $f(x^k) - f(x^{\star})$.
Moreover, its iterates $\sets{x^k}$ also converge to a minimizer $x^{\star}$ of $f$; see, for instance, \cite{attouch2016rate}.
Over the past two decades, research on this topic has evolved extensively, with significant contributions summarized in \cite{beck2017first,lan2020first,Nesterov2004}.
Furthermore, these ideas have been successfully extended to broader settings, including monotone inclusions and variational inequalities; see, e.g., \cite{attouch2020convergence,bot2023fast,kim2021accelerated,mainge2021accelerated,mainge2021fast,tran2022connection}.

%%%%%%%%%%%%%%%%%%%%%%%%%%%%%%%%%%%%%%%%%%%%%%%%%%%%%%%%%%
%%% Generic Inexact Accelerated Fixed-Point Methods.
%%%%%%%%%%%%%%%%%%%%%%%%%%%%%%%%%%%%%%%%%%%%%%%%%%%%%%%%%%
\beforesec
\section{Accelerated Fixed-Point Framework with Delayed Inexact Oracle}\label{sec:iFKM_framework}
\aftersec
In this section, we develop a \textbf{novel accelerated fixed-point iteration framework} with \textbf{\textit{delayed inexact oracle}} to approximate a solution of \eqref{eq:FP}, or equivalently, a root of \eqref{eq:CE}.
Our method is inspired by Nesterov's accelerated gradient algorithm described in \eqref{eq:NAG}.

\beforesubsec
\subsection{\textbf{\textit{The accelerated fixed-point method with delayed inexact oracle}}}\label{subsec:iFKM_scheme}
\aftersubsec
Our method for solving \eqref{eq:CE} is described as follows.
%Building up on Nesterov's accelerated methods for smooth and convex optimization \cite{Nesterov2004}, we propose the following generic scheme to approximate a solution of \eqref{eq:CE}.
\textit{
\begin{compactitem}[$\bullet$]
\item\textbf{Initialization:} 
Choose a constant $s > 1$.
Given an initial point $y^0 \in \Hspace$, we set $z^0 := y^0$.
\item \textbf{Main loop:}
At each iteration $k \geq 0$, we compute an appropriate \textbf{delayed inexact estimator} $\widetilde{G}^k$ of $Gy^k$ $($either deterministic or stochastic$)$, and then update
\begin{equation}\label{eq:iFKM4CE}
\arraycolsep=0.2em
\left\{\begin{array}{lcl}
x^{k+1} &:= & y^k - \eta_k \widetilde{G}^{k}, \vspace{1ex}\\
z^{k+1} &:= & z^k + \frac{\gamma_k}{s}(x^{k+1} - y^k), \vspace{1ex}\\
y^{k+1} & := & \frac{t_k - s}{t_k}x^{k+1} + \frac{s}{t_k}z^{k+1},
\end{array}\right.
\tag{AFP}
\end{equation}
where  $\eta_k > 0$, $\gamma_k \in [0, 1]$,  and $t_k > 1$ are given parameters, determined later.
\end{compactitem}
}
%%%
In the subsequent sections, we will specify how to construct the delayed inexact estimator $\widetilde{G}^k$ of $Gy^k$ as well as the choice of $s$ and the update rules of $t_k$, $\eta_k$, and $\gamma_k$.
First, we define
\begin{equation}\label{eq:oracle_error}
e^k := \widetilde{G}^k - Gy^k
\end{equation}
to be the error between the \textit{\textbf{delayed inexact estimator}} $\widetilde{G}^k$ and the true operator value $Gy^k$ at $y^k$.
Let us briefly discuss the relations between our scheme \eqref{eq:iFKM4CE} and other methods.
\begin{compactitem}[$\diamond$]
\item 
In the exact case, i.e., when $\widetilde{G}^k = Gy^k$, our method is much closer in spirit to Nesterov’s variant in \cite{allen2017linear,lan2011primal} than to those in \cite{Nesterov1983,Nesterov2004}.
A key distinction, however, lies in the choice of $\gamma_k$.
In convex optimization, as shown in \eqref{eq:NAG}, the parameter is taken as $\gamma_k = st_k = \BigOs{k}$ \cite{allen2017linear,lan2011primal}, whereas in our method $\gamma_k$ is constant, namely $\gamma_k = \gamma = \BigOs{1}$.
This difference leads to a more aggressive update of $z^k$ in convex optimization compared with our approach.
 
\item 
Alternatively, if $\widetilde{G}^k = Gy^k$ (i.e., the exact value $Gy^k$ is used), then our scheme reduces to a deterministic accelerated fixed-point method for solving \eqref{eq:CE}, which can be transformed into forms similar to those studied in \cite{attouch2020convergence,bot2022bfast,kim2021accelerated,mainge2021accelerated,mainge2021fast,tran2022connection}.
However, as mentioned earlier, extending the convergence analyses from these works to the inexact oracle variants is difficult compared with \eqref{eq:iFKM4CE}, because their proofs require bounding product terms such as $\iprods{e^k, y^k - y^{k-1}}$ and $\iprods{e^k, y^k - x^{\star}}$.
Controlling such terms becomes challenging unless they vanish; for example, when using unbiased estimators $\widetilde{G}^k$ satisfying $\Expsn{k}{\widetilde{G}^k} = Gy^k$.
 
\item 
For the inexact setting, we allow both deterministic and stochastic inexact oracles in our scheme \eqref{eq:iFKM4CE}.
The stochastic oracle is not required to be unbiased, which makes our method more flexible and broad enough to encompass a wide range of existing and new inexact oracles, particularly those involving delays.

\item 
In convex optimization, the underlying inexact oracles for accelerated methods are often defined directly (see, e.g., \cite{Devolder2010,dvurechensky2016stochastic}) or indirectly through the objective function (see, e.g., \cite{Schmidt2011}).
Such formulations make it possible to avoid directly bounding inner-product terms such as $\iprods{e^k, y^k - x^{\star}}$.
However, because no objective function underlies \eqref{eq:CE}, establishing convergence for inexact methods becomes substantially more difficult.
\end{compactitem}
%%%%
Note that our framework \eqref{eq:iFKM4CE} is single-loop.
In terms of per-iteration complexity, \eqref{eq:iFKM4CE} requires only one evaluation of $\widetilde{G}^k$ per iteration.
All additional operations consist solely of vector additions or scalar-vector multiplications, making their computational cost comparable to that of classical fixed-point schemes such as BP iteration or KM fixed-point schemes.
 
\beforesubsec
\subsection{\textit{\textbf{Key properties of \eqref{eq:iFKM4CE}}}}\label{subsec:fast_FP_method_key_estimate}
\aftersubsec
To analyze Nesterov's accelerated methods in convex optimization, we often use a Lyapunov function of the form $\Lc_k = a_k(f(x^k) - f(x^{\star})) + \frac{1}{2}\norms{z^k - x^{\star}}^2$ for some $a_k = \BigOs{k^2}$, see, e.g., \cite{bansal2017potential}.
For our scheme \eqref{eq:iFKM4CE}, we first introduce the following function:
\begin{equation}\label{eq:iFKM4CE_Lk_func}
\arraycolsep=0.2em
\begin{array}{lcl}
\Lc_k &:= & \frac{a_k}{2}\norms{Gy^k}^2 + s t_{k-1}\iprods{Gy^k, y^k - z^k} + \frac{s^2(s-1)}{2\eta_{k-1}\gamma_{k-1}}\norms{z^k - x^{\star}}^2, 
\end{array}
\end{equation}
where $a_k > 0$ is given, determined later, and other parameters are given in \eqref{eq:iFKM4CE}.

The function $\Lc_k$ is not yet a Lyapunov function for our method.
One needs to augment it with additional terms due to the delayed inexact oracle.
Next, we state the following two technical lemmas to characterize the properties of $\Lc_k$ in  \eqref{eq:iFKM4CE_Lk_func}.
%Lemma~\ref{le:iFKM4CE_key_estimate2} establishes a lower bound on the difference $\Lc_k - \Lc_{k+1}$ for $\Lc_k$, while Lemma~\ref{le:iFKM4CE_lower_bound_of_Lk} gives a lower lower bound of $\Lc_k$.
For the sake of presentation, the proof of these lemmas is deferred to Appendix~\ref{sec:apdx:iFKM4CE_framework}.

%%%% Lemma 2.2.
\begin{lemma}\label{le:iFKM4CE_key_estimate2}
	Suppose that Condition~\eqref{eq:G_cocoercivity} of Assumption~\ref{as:A1} holds for \eqref{eq:CE} and $\gamma \in [0, 1]$, $s > 1 + \gamma$, and $\eta > 0$ are given.
	Let $\sets{(x^k, y^k, z^k)}$ be generated by \eqref{eq:iFKM4CE} using
	\begin{equation}\label{eq:iFKM4CE_param1}
	\arraycolsep=0.2em
	\begin{array}{lcl}
	t_k = t_{k-1} + 1, \quad t_k > 2s, \quad \gamma_k := \gamma \in (0, 1], \quad \text{and} \quad \eta_k := \frac{\eta t_k}{2(t_k - s)}.
	\end{array}
	\end{equation}
	Then, for any $c_1 > 0$ and $c_2 > 0$, and $e^k$ defined by \eqref{eq:oracle_error}, $\Lc_k$ defined by \eqref{eq:iFKM4CE_Lk_func} satisfies
	\begin{equation}\label{eq:iFKM4CE_key_est2}
		\arraycolsep=0.2em
		\begin{array}{lcl}
			\Lc_k - \Lc_{k+1} & \geq & \frac{[2\beta - (1+c_1)\eta] t_k(t_k - s)}{2} \norms{Gy^{k+1} - Gy^k}^2  + \bar{\beta}t_k(t_k - s)\Dc(y^{k+1}, y^k)\vspace{1ex}\\
			&& + {~} \frac{\eta (s - \gamma - 1) t_k}{4}\norms{Gy^k}^2  + s(s-1) \iprods{ Gy^k, y^k - x^{\star} } \vspace{1ex}\\
			&& - {~} \frac{\eta}{2} \left(\frac{1}{c_1} + \frac{s-\gamma}{c_2}\right) t_k(t_k-s) \norms{e^k}^2.
		\end{array}
	\end{equation}
\end{lemma}

%%% Lemma 2.3.
\begin{lemma}\label{le:iFKM4CE_lower_bound_of_Lk}
	Let $\Lc_k$ be defined in \eqref{eq:iFKM4CE_Lk_func}.
	Then, under Assumption~\ref{as:A1}, if $s \geq 1 + 3\gamma$, then we have
	\begin{equation}\label{eq:iFKM4CE_iasyn_key_est2}
		\arraycolsep=0.2em
		\begin{array}{lcl}
			\Lc_k &\geq& \frac{\eta t_{k-1}^2}{8} \norms{Gy^k}^2+ \frac{s^2 (t_{k-1} - 3s)}{\eta t_{k-1}} \norms{z^k - x^{\star}}^2.
		\end{array}
	\end{equation}
\end{lemma}

%%%%%%%%%%%%%%%%%%%%%%%%%%%%%%%%%%%%%%%%%%%%%%%%
%%% 3.3. Convergence of iFKM under general error bound condition.
\beforesubsec
\subsection{\textbf{\textit{Convergence of \eqref{eq:iFKM4CE} under general error approximation condition}}}\label{subsec:iFKM_convergence1}
\aftersubsec
\noindent\textbf{$\mathrm{(a)}$~General error approximation condition.}
Without loss of generality, suppose that $\widetilde{G}^k$ is a stochastic approximation of $Gy^k$ in \eqref{eq:iFKM4CE} such that the error $e^k := \widetilde{G}^k - Gy^k$ defined by \eqref{eq:oracle_error} satisfies the following definition.
%%%
%%% Definition 1.
\begin{definition}[\textbf{\textit{Error approximation condition}}]\label{de:error_bound_cond}
\textit{
Suppose that $\sets{ y^k }$ is generated by \eqref{eq:iFKM4CE}, and $e^k$ is the error between the [possibly stochastic] approximation $\widetilde{G}^k$ and $Gy^k$ defined by \eqref{eq:oracle_error} adapted to the filtration $\sets{\Fc_k}$.
There exists a sequence of nonnegative random variables $\sets{\Delta_k}$, three constants $\kappa \in (0, 1]$, $\Theta \geq 0$, and $\hat{\Theta} \geq 0$, an integer delay $\tau_k \in [0, \tau]$, and a nonnegative sequence $\sets{\delta_k}$ such that, almost surely, the following conditions hold:
\begin{equation}\label{eq:iFKM_error_bound}
\arraycolsep=0.2em
\hspace{-2ex}
\left\{\begin{array}{lcl}
\Expsn{k}{ \norms{e^k}^2 } &\leq & \Expsn{k}{\Delta_k}, \vspace{0ex}\\
\Expsn{k}{ \Delta_k} & \leq & (1-\kappa)\Delta_{k-1} + {\!\!\!\!} {\displaystyle\sum_{l=[k-\tau_k + 1]_{+}}^k } \big[\Theta\norms{Gy^l - Gy^{l-1}}^2 + \hat{\Theta}\Dc(y^l, y^{l-1})\big] + \frac{\delta_k}{t_k (t_k - s)},
\end{array}\right.
\hspace{-2ex}
\vspace{-1ex}
\end{equation}
where $y^0 = y^{-1} = \cdots = y^{\tau_0-1}$, $t_k$, and $s$ is given in \eqref{eq:iFKM4CE}, and $\Dc$ is in \eqref{eq:G_cocoercivity} of Assumption~\ref{as:A1}.
}
\end{definition}
%%%
Note that, in \eqref{eq:iFKM_error_bound}, if $\beta = 0$ in Assumption~\ref{as:A1}, then $\Theta = 0$, and if $\bar{\beta} = 0$, then $\hat{\Theta} = 0$.
To be consistent, we use the convention $\frac{c}{0} = +\infty$ for any $c \geq 0$ when necessary.
Our approximation condition \eqref{eq:iFKM_error_bound} in Definition~\ref{de:error_bound_cond} is sufficiently general to cover various scenarios, including deterministic and stochastic (both unbiased and biased) errors, delayed updates, and asynchronous updates.
Let us further discuss a few aspects of Definition~\ref{de:error_bound_cond}.
\begin{compactitem}[$\bullet$]
\item First, to the best of our knowledge, the condition \eqref{eq:iFKM_error_bound} is the first to address delayed inexact oracles for fixed-point-type methods in a unified manner (deterministic and stochastic).
\item Second, in the special case where \eqref{eq:CE} corresponds to an optimization problem, most asynchronous algorithms rely on specific mechanisms (e.g., deterministic, stochastic, or incremental aggregated approaches).
In contrast, the framework captured by Definition~\ref{de:error_bound_cond} allows us to encompass a broad class of algorithms, as will be demonstrated in Section~\ref{sec:app_of_iFKM}.
\item Third, in the absence of delay, our framework includes \cite{TranDinh2025a} as a special case.
However, the scheme \eqref{eq:iFKM4CE} differs substantially from \cite{TranDinh2025a} because it uses $Gy^k$ rather than $Gx^k$, which leads to a significantly different analysis.
In this delay-free setting, our framework also differs from \cite{tran2024accelerated} since we do not impose the unbiasedness requirement in Definition~\ref{de:error_bound_cond}.
\end{compactitem}
\noindent\textbf{$\mathrm{(b)}$~Lyapunov function and its key bound.}
The following lemma introduces a new Lyapunov function and establishes a key recursive inequality to analyze convergence of our scheme \eqref{eq:iFKM4CE}, whose proof is deferred to Appendix~\ref{subsec:apdx:le:iFKM4CE_descent_property}.

%%% Lemma 3.4.
\begin{lemma}\label{le:iFKM4CE_descent_property}
	Suppose that Assumption~\ref{as:A1} holds for \eqref{eq:CE}, 
	and $\gamma \in [0,1]$, $s > 1 + \gamma$, $\tau \geq 1$, $\Lambda := 1 + s - \gamma$, and $\eta > 0$ are given.
	Let $\sets{(x^k, y^k, z^k)}$ be generated by \eqref{eq:iFKM4CE} using an estimator $\widetilde{G}^k$ satisfying Definition~\ref{de:error_bound_cond}, where $t_k$, $\gamma_k$, and $\eta_k$ are respectively updated by
	\begin{equation}\label{eq:iFKM4CE_le_params}
		\arraycolsep=0.2em
		\begin{array}{lcl}
			t_k := k + 3s + \tau, \quad \gamma_k := \gamma, \quad \text{and} \quad \eta_k := \frac{\eta t_k}{2(t_k - s)}.
		\end{array}
	\end{equation}
	Suppose further that $\kappa$ in \eqref{eq:iFKM_error_bound} satisfies $\alpha := \frac{2s + 1}{(s+1)^2}  < \kappa \leq 1$ and $0 \leq \tau_k \leq \tau$ for all $k \geq 0$.
	For $\Lc_k$ given in \eqref{eq:iFKM4CE_Lk_func} and $\theta := 1 - (1-\kappa)\alpha > 0$, consider the following \textbf{Lyapunov function}:
	\begin{equation}\label{eq:iFKM4CE_le_Vk}
		\hspace{-2ex}
		\arraycolsep=0.2em
		\begin{array}{lcl}
			V_k := \Lc_k + t_{k-1}(t_{k-1} {\!} - {\!} s)\left[\frac{ 2\beta - (1+\tau)\eta }{2} \norms{Gy^k {\!} - Gy^{k \!- \! 1}}^2 + \bar{\beta} \Dc(y^{k}, y^{k \!-\! 1}) +  \frac{ \Lambda\alpha(1-\kappa)\eta}{2\tau\theta} \Delta_{k-1} \right].
		\end{array}
		\hspace{-2ex}
	\end{equation}
	Then, we can show that almost surely 
	\begin{equation}\label{eq:iFKM4CE_le_key_est}
		\arraycolsep=0.2em
		\begin{array}{lcl}
			\Expsn{k}{ V_{k+1} } &\leq & V_k - U_k + \mu_k{\displaystyle\sum_{l=[k-\tau_k + 1]_{+} }^k} S_l - \nu_k  S_k  + \omega_k, 
		\end{array}
	\end{equation}
	where  
	\begin{equation}\label{eq:iFKM4CE_le_Uk}
		\hspace{-2ex}
		\arraycolsep=0.2em
		\left\{\begin{array}{lcl}	
			U_k &:=& \frac{ t_{k-1}(t_{k-1}-s)}{2} \left[ \frac{\Lambda(1 - \kappa)(s+1)\eta }{2\tau\theta s^2}\Delta_{k-1}  + \beta \norms{Gy^k - Gy^{k-1}}^2 + \bar{\beta} \Dc(y^k, y^{k-1}) \right] \vspace{1ex}\\
			&& + {~} \frac{\eta (s - \gamma - 1) t_k + 4\beta s(s-1)}{4} \norms{Gy^k}^2  , \vspace{1.5ex}\\
			
			S_k &:=& \Theta\norms{Gy^k - Gy^{k-1}}^2 + \hat{\Theta}\Dc(y^k, y^{k-1}), \vspace{1.5ex}\\
			
			\omega_k &:=& \frac{\Lambda \eta}{2\tau\theta} \delta_k,  \quad \mu_k := \frac{\Lambda \eta}{2\tau\theta} t_k(t_k - s), \quad\text{and} \quad  \nu_k := \frac{1}{2} t_{k-1} (t_{k-1} - s)\min\set{\frac{\beta  - (1+\tau)\eta}{\Theta}, \frac{\bar{\beta}}{\hat{\Theta}}}.
		\end{array}\right.
		\hspace{-5ex}
	\end{equation} 
\end{lemma}

\noindent\textbf{$\mathrm{(c)}$~Main result 1: Convergence rates and iteration-complexity.}
For given $s > 1$, $\tau \geq 1$, and $\gamma \in [0, 1]$, we define  $\Lambda := 1 + s - \gamma > 0$ and $\alpha := \frac{2s + 1}{(s+1)^2} \in (0, 1)$ as before.
Given $\kappa$, $\Theta$, and $\hat{\Theta}$ in Definition~\ref{de:error_bound_cond} such that $\alpha < \kappa \leq 1$, and $\beta$ and $\bar{\beta}$ in Assumption~\ref{as:A1}, let us denote $\theta := 1 - (1-\kappa)\alpha > 0$ and define the following quantity:
\begin{equation}\label{eq:iFKM4CE_eta}
	\arraycolsep=0.2em
	\bar{\eta} := \left\{ \begin{array}{ll}
		\frac{3\theta \bar{\beta}}{7 \Lambda \hat{\Theta} } \quad &\text{if } 0 < \bar{\beta} \leq \frac{7\Lambda \hat{\Theta} \beta}{7\Lambda \Theta + 3\theta(1+\tau)}, \vspace{1ex}\\
		
		\frac{3\theta\beta}{7\Lambda\Theta  + 3\theta(1+\tau)} \quad &\text{if } \bar{\beta} = 0 \ \text{ or } \ \frac{7\Lambda \hat{\Theta} \beta}{7\Lambda \Theta + 3\theta(1+\tau)} \leq \bar{\beta} \leq \frac{\hat{\Theta}}{\Theta}\beta.
	\end{array}\right.
\end{equation}
Then, we are ready to state and prove our first main result in this section.

%%% Theorem 3.1
\begin{theorem}\label{th:iFKM4CE_convergence_expectation}
	Suppose that Assumption~\ref{as:A1} holds for  \eqref{eq:CE}.
	Let $\sets{(x^k, y^k, z^k)}$ be generated by \eqref{eq:iFKM4CE} such that $\widetilde{G}^k$ satisfies Definition~\ref{de:error_bound_cond} with $\Theta \bar{\beta} \leq \hat{\Theta}\beta$, $\alpha := \frac{2s + 1}{(s+1)^2}  < \kappa \leq 1$ for a given $s > 1$, and a bounded delay $0 \leq \tau_k \leq \tau$ for a given $\tau \geq 1$.
	Suppose further that $\bar{\eta}$ is defined by \eqref{eq:iFKM4CE_eta}, $\gamma \in [0,1]$, $s \geq 1 + 3\gamma$, and $t_k$, $\gamma_k$, and $\eta_k$ are respectively updated as follows:
	\begin{equation}\label{eq:iFKM4CE_params}
		\arraycolsep=0.2em
		\begin{array}{lll}
			t_k := k + 3s + \tau,  \quad \gamma_k := \gamma, \quad \textrm{and}  \quad 
			\eta_k := \dfrac{\eta t_k}{2(t_k - s)}, \quad\text{where}\quad 0 < \eta\leq \bar{\eta}.
		\end{array}
	\end{equation}
	Then, the following statements hold.
	\begin{compactitem}
	%%% a. The Big-O rate
	\item[\textbf{$\mathrm{(i)}$~The $\BigOs{1/k^2}$ convergence rate in expectation.}]
	We have
	\begin{equation}\label{eq:iFKM4CE_th_BigO_rates}
		\arraycolsep=0.2em
		\begin{array}{lcl}
			\displaystyle \Expn{\norms{Gy^K}^2} &\leq& \displaystyle \frac{4}{\eta (K + 3s + \tau - 1)^2} \left(\Rc_0^2 + \frac{\Lambda\eta \Sc_K }{\theta\tau}\right),
		\end{array}
	\end{equation}
	where $\theta := 1 - (1-\kappa)\alpha$, \  $\Rc_0^2 := \frac{\eta (3s + \tau - 1)^2}{2} \norms{Gy^0}^2 + \frac{2s^3}{\eta \gamma}\norms{y^0 - x^{\star}}^2$, and $\Sc_K := \sum_{k=0}^K\delta_k$.
	%%% b. Summability bounds.
	\item[\textbf{$\mathrm{(ii)}$~Summability bounds in expectation.}] 
	If $\Sc_{\infty} := \sum_{k=0}^{\infty}\delta_k < +\infty$, then we have
	\begin{equation}\label{eq:iFKM4CE_th_summability_bounds}
		\arraycolsep=0.2em
		\begin{array}{lcl}
			\displaystyle \sum_{k=0}^\infty (k +  \tau) \Expn{\norms{x^{k+1} - x^k}^2}  < +\infty 
			\quad\textrm{and} \quad  
			\sum_{k=0}^\infty (k  + \tau) \Expn{\norms{Gy^k}^2}  < +\infty.
		\end{array}
	\end{equation}
	%%% c. The small-o rate.
	\item[\textbf{$\mathrm{(iii)}$~The $\SmallOs{1/k^2}$ convergence in expectation.}] 
	If $\Sc_{\infty} := \sum_{k=0}^{\infty}\delta_k < +\infty$, then we get
	\begin{equation}\label{eq:iFKM4CE_th_smallo_rates}
	\arraycolsep=0.2em
	\begin{array}{lcl}
		\displaystyle \lim_{k \to \infty} (k + \tau)^2 \Expn{\norms{x^{k+1} - x^k}^2} = 0 
		\quad\textrm{and}\quad
		\displaystyle \lim_{k \to \infty} (k + \tau)^2 \Expn{\norms{Gy^k}^2} &=& 0.
	\end{array}
\end{equation}	
	%%% d. Iteration-complexity.
	\item[\textbf{$\mathrm{(iv)}$~Iteration-complexity.}]
	For a given tolerance $\epsilon > 0$, if we choose $\delta_k := \frac{\sigma^2}{(k+r)^{1+\omega}}$ for some $r > 0$, $\omega > 0$, and $\sigma \geq 0$, then the total number of iterations $k$ to achieve $\Expn{\norms{Gy^k}^2} \leq \epsilon^2$ is at most 
	\begin{equation*}
	\arraycolsep=0.2em
	k := \left\{\begin{array}{ll}
	\BigO{\frac{\sigma}{\epsilon\sqrt{\tau}} + \frac{\tau}{\epsilon} + \frac{\hat{\Theta}}{\bar{\beta}\epsilon}} & \text{if $0 < \bar{\beta} \leq \frac{7\Lambda \hat{\Theta} \beta}{7\Lambda \Theta + 3\theta(1+\tau)}$}, \vspace{1ex}\\
	\BigO{\frac{\sigma}{\epsilon\sqrt{\tau}} + \frac{\tau}{\epsilon} + \frac{\Theta + \tau}{\beta\epsilon}} & \text{if $\bar{\beta} = 0$ or $\frac{7\Lambda \hat{\Theta} \beta}{7\Lambda \Theta + 3\theta(1+\tau)} \leq \bar{\beta} \leq \frac{\hat{\Theta}}{\Theta}\beta$}.
	\end{array}\right.
	\end{equation*}
\end{compactitem}	
\end{theorem}

%%% Proof of Theorem 3.1
\begin{proof}
First, for $\nu_k$ and $\mu_k$ in \eqref{eq:iFKM4CE_le_Uk}, and $\theta := 1 - (1-\kappa)\alpha > 0$, we need to guarantee $\sum_{l=0}^{\tau-1} \mu_{k+l} \leq \nu_k$ for \eqref{eq:iFKM4CE_le_key_est}  to apply Lemma~\ref{le:Asyn_lemma_A2} for our subsequent proofs. 
Indeed, we have
	\begin{equation*}
		\arraycolsep=0.2em
		\begin{array}{lcl}
			\sum_{l=0}^{\tau-1} \mu_{k+l}  & = &  \frac{\Lambda\eta}{2\tau \theta } \sum_{l=0}^{\tau-1} (t_{k-1} + l + 1)(t_{k-1} - s + l + 1)  \vspace{1ex}\\
%			&=& \frac{\Lambda\eta}{2\tau\theta } \sum_{l=0}^{\tau-1} \left[ t_{k-1}(t_{k-1} - s) + (2t_{k-1} - s)(l+1) + (l+1)^2 \right]  \vspace{1ex}\\
			&=& \frac{\Lambda\eta}{2\tau\theta } \left[ \tau t_{k-1}(t_{k-1} - s) + \frac{\tau(\tau+1)}{2}(2t_{k-1} - s) + \frac{\tau(\tau+1)(2\tau+1)}{6} \right]  \vspace{1ex}\\
			&\stackrel{\tiny\textcircled{1}}{\leq}& \frac{\Lambda\eta}{2\tau\theta } \left[ \tau t_{k-1}(t_{k-1} - s) + \tau t_{k-1}(t_{k-1} - s) + \frac{1}{3}\tau t_{k-1}(t_{k-1} - s) \right]  \vspace{1ex}\\
			&=& \frac{ 7\Lambda\eta}{6\theta } t_{k-1}(t_{k-1} - s),
		\end{array}
	\end{equation*}
	where {\tiny\textcircled{1}} holds due to the choice $t_{k-1} \geq 2s + \tau$, $s > 1$ and $\tau \geq 1$. 

	Now, given this bound, we consider the following cases:
	\begin{compactitem}[$\bullet$]
		\item 
		If $\bar{\beta} = 0$, then $\nu_k$ reduces to $\nu_k := \frac{\beta - (1+\tau)\eta}{2\Theta}t_{k-1}(t_{k-1}-s)$.
		To guarantee $\sum_{l=0}^{\tau-1} \mu_{k+l} \leq \nu_k$, we require $\frac{7\Lambda\eta}{6\theta} \leq \frac{\beta - (1+\tau)\eta}{2\Theta}$, leading to $0 \leq \eta \leq \bar{\eta} := \frac{3\theta\beta}{7\Lambda\Theta  + 3\theta(1+\tau)}$.
		
		\item 
		If $0 < \bar{\beta} \leq \frac{\hat{\Theta}}{\Theta}\beta$, then we consider two sub-cases:
		\begin{compactitem}[$\diamond$]
			\item If $0 \leq \eta \leq \frac{\Theta}{1+\tau} (\frac{\beta}{\Theta} - \frac{\bar{\beta}}{\hat{\Theta}})$, then $\min\set{ \frac{\beta - (1+\tau)\eta}{\Theta}, \ \frac{\bar{\beta}}{\hat{\Theta}} } = \frac{\bar{\beta}}{\hat{\Theta}}$.
			To guarantee $\sum_{l=0}^{\tau-1} \mu_{k+l} \leq \nu_k$, we require $\frac{7\Lambda\eta}{3\theta} \leq \frac{\bar{\beta}}{\hat{\Theta}}$, which is equivalent to $\eta \leq \bar{\eta} :=   \frac{3\theta \bar{\beta}}{7\Lambda  \hat{\Theta}}$.
			\begin{compactitem}[$\triangleright$]
				\item If $0 < \bar{\beta} \leq \frac{7\Lambda \hat{\Theta} \beta}{7\Lambda \Theta + 3\theta (1+\tau)}$, then we get $0 < \eta \leq  \bar{\eta} :=  \frac{3\theta \bar{\beta}}{7 \Lambda \hat{\Theta}}$.
				\item If $\frac{7\Lambda \hat{\Theta} \beta}{7\Lambda \Theta + 3\theta (1+\tau)} \leq \bar{\beta} \leq \frac{\hat{\Theta}}{\Theta}\beta$, then we get $0 < \eta \leq \frac{\Theta}{1+\tau} \left(\frac{\beta}{\Theta} - \frac{\bar{\beta}}{\hat{\Theta}}\right)$.
			\end{compactitem}
			
			\item If $\eta \geq \frac{\Theta}{1+\tau} \left(\frac{\beta}{\Theta} - \frac{\bar{\beta}}{\hat{\Theta}}\right) \geq 0$, then $\min\set{ \frac{\beta - (1+\tau) \eta}{\Theta}, \ \frac{\bar{\beta}}{\hat{\Theta}} } = \frac{\beta - (1+\tau)\eta}{\Theta}$.
			To guarantee $\sum_{l=0}^{\tau-1} \mu_{k+l} \leq \nu_k$, we require $\frac{7\Lambda \eta}{3\theta } \leq \frac{\beta - (1+\tau)\eta}{\Theta}$, which is equivalent to $\eta \leq \bar{\eta} :=  \frac{3\theta \beta}{7\Lambda \Theta + 3\theta (1+\tau)}$.
			This condition holds if $\frac{7\Lambda \hat{\Theta}}{7\Lambda \Theta + 3\theta (1+\tau)} \beta \leq \bar{\beta} \leq \frac{\hat{\Theta}}{\Theta}\beta$, and thus, we finally get $\frac{\Theta}{1+\tau} \left(\frac{\beta}{\Theta} - \frac{\bar{\beta}}{\hat{\Theta}}\right) \leq \eta \leq \bar{\eta} :=  \frac{3\theta \beta}{7\Lambda \Theta + 3\theta (1+\tau)}$.
		\end{compactitem}
	\end{compactitem}
	Combining the above arguments, it confirms the choice of $\eta$ in \eqref{eq:iFKM4CE_params}, with $\bar{\eta}$ given by \eqref{eq:iFKM4CE_eta}.
	
	Next, under the condition $\sum_{l=0}^{\tau-1} \mu_{k+l} \leq \nu_k$, we are ready to prove the remaining results.
	
%	\vspace{1ex}
	\noindent\textbf{$\mathrm{(i)}$ The $\BigOs{1/k}$ convergence rate.}
	Applying Lemma~\ref{le:Asyn_lemma_A2}(i) to \eqref{eq:iFKM4CE_le_key_est} in Lemma~\ref{le:iFKM4CE_descent_property}, we have
	\begin{equation*}
		\arraycolsep=0.2em
		\begin{array}{lcl}
			\Expn{V_{K+1}} &\leq& \Expn{V_0} + \frac{\Lambda\eta}{2\tau\theta } \sum_{k=0}^{K}\delta_k, \vspace{1ex}\\
			\sum_{k=0}^K \frac{\eta(s - \gamma - 1)t_k + 4\beta s(s-1) }{4}\Expn{\norms{Gy^k}^2} &\leq& \Expn{V_0} + \frac{\Lambda\eta}{2\tau\theta } \sum_{k=0}^{K}\delta_k, \vspace{1ex}\\
			\sum_{k=0}^K \frac{\Lambda\eta (1 - \kappa)(s+1)t_{k-1}(t_{k-1}-s)}{4\tau s^2\theta }\Expn{\Delta_{k-1}} &\leq& \Expn{V_0} + \frac{\Lambda\eta}{2\tau\theta } \sum_{k=0}^{K}\delta_k, \vspace{1ex}\\
			\sum_{k=0}^K \frac{\beta t_{k-1}(t_{k-1} - s)}{2}  \Expn{\norms{Gy^k - Gy^{k-1}}^2} &\leq& \Expn{V_0} + \frac{\Lambda\eta}{2\tau\theta } \sum_{k=0}^{K}\delta_k, \vspace{1ex}\\
			\sum_{k=0}^K \frac{\bar{\beta}t_{k-1}(t_{k-1} - s)}{2}  \Expn{\Dc(y^k, y^{k-1})} &\leq& \Expn{V_0} + \frac{\Lambda\eta}{2\tau\theta } \sum_{k=0}^{K}\delta_k.
		\end{array}
	\end{equation*}
	Let us denote by $\Sc_K := \sum_{k=0}^K \delta_k$. 
	Since $z^0 = y^{-1} := y^0$, $\Delta_{-1} := 0$, $t_{-1} = 3s + \tau - 1$, and $\eta_{-1} = \frac{\eta t_{-1}}{2(t_{-1} - s)} = \frac{\eta (3s + \tau - 1)}{2(2s + \tau - 1)}$, $a_0 = \eta_{-1} t_{-1} (t_{-1} - s) = \frac{\eta (3s + \tau - 1)^2}{2}$, we have
	\begin{equation*}
	\arraycolsep=0.2em
	\begin{array}{lcl}
		\Expn{V_0} &=& V_0 = \Lc_0 = \frac{a_0}{2}\norms{Gy^0}^2 + \frac{s^2(s-1)}{2\eta_{-1}\gamma_{-1}}\norms{y^0 - x^{\star}}^2 \vspace{1ex} \\
		&=& \frac{\eta (3s + \tau - 1)^2}{4} \norms{Gy^0}^2 + \frac{s^2(s-1)(2s + \tau - 1)}{\eta\gamma (3s + \tau - 1)}\norms{y^0 - x^{\star}}^2 \vspace{1ex} \\
		&\stackrel{\tiny\textcircled{1}}{\leq}& \frac{\eta (3s + \tau - 1)^2}{4} \norms{Gy^0}^2 + \frac{s^3}{\eta \gamma}\norms{y^0 - x^{\star}}^2
		 := \frac{\Rc_0^2}{2},
	\end{array}
	\end{equation*}
	where {\tiny\textcircled{1}} holds due to 
	$\frac{s^2(s-1)(2s + \tau - 1)}{\eta\gamma (3s + \tau - 1)}
	\leq \frac{s^3}{\eta \gamma}$.
	Moreover, since $t_{k-1} = k + 3s + \tau - 1 \geq 3s$, \eqref{eq:iFKM4CE_iasyn_key_est2} implies that 
	$\norms{Gy^k}^2 \leq \frac{8\Lc_k}{\eta t_{k-1}^2} \leq \frac{8V_k}{\eta t_{k-1}^2}$.
	Using the last two expressions and $t_k = k + 3s + \tau$, we can easily prove the following bounds:
	\begin{equation*}%\label{eq:iFKM4CE_th_proof1}
	\arraycolsep=0.2em
%	\hspace{-2ex}
	\begin{array}{lcl}
		\Expn{\norms{Gy^K}^2} &\leq& \frac{4\theta \tau \Rc_0^2 + 4\Lambda\eta \Sc_{K-1}}{\eta\theta\tau (K + 3s + \tau - 1)^2}, \vspace{1ex}\\
		\sum_{k=0}^K (k + 3s + \tau)\Expn{\norms{Gy^k}^2} &\leq& \frac{2\theta\tau\Rc_0^2 + 2\Lambda\eta \Sc_K }{\theta\tau\eta(s - \gamma - 1)}, \vspace{1ex}\\
		\sum_{k=0}^K (k + 3s + \tau - 1)(k + 2s + \tau - 1) \Expn{\Delta_{k-1}} &\leq& \frac{2s^2 \left(\tau\theta \Rc_0^2 + \Lambda\eta \Sc_K \right)}{\Lambda\eta (1 - \kappa)(s+1)} , \vspace{1ex}\\
		\sum_{k=0}^K (k + 3s + \tau - 1)(k + 2s + \tau - 1)  \Expn{\norms{Gy^k - Gy^{k-1}}^2} &\leq& \frac{\tau\theta\Rc_0^2 + \Lambda\eta \Sc_K}{\beta \tau\theta}, \vspace{1ex}\\
		\sum_{k=0}^K (k + 3s + \tau - 1)(k + 2s + \tau - 1)  \Expn{\Dc(y^k, y^{k-1})} &\leq& \frac{\tau\theta \Rc_0^2 +  \Lambda\eta \Sc_K }{\bar{\beta} \tau\theta }.
	\end{array}
%	\hspace{-2ex}
	\end{equation*}
	The first line of these results leads to \eqref{eq:iFKM4CE_th_BigO_rates}.

	In addition, if we assume that $\Sc_k \leq \Sc_{\infty} = \sum_{k=0}^\infty \delta_k < +\infty$, then by passing through limits as $K \to \infty$ in the last four inequalities, we obtain 
	\begin{equation}\label{eq:iFKM4CE_th_proof1.5}
	\arraycolsep=0.2em
	\hspace{-2ex}
	\begin{array}{lcl}
		\sum_{k=0}^{\infty} (k + 3s + \tau)\Expn{\norms{Gy^k}^2} &\leq& \frac{2 \tau\theta \Rc_0^2 +  2\Lambda\eta \Sc_{\infty} }{ \tau\theta \eta(s - \gamma - 1)} < +\infty, \vspace{1ex}\\
		\sum_{k=0}^{\infty} (k + 3s + \tau)(k + 2s + \tau) \Expn{\Delta_{k}} &\leq& \frac{2s^2\left(\theta \tau\Rc_0^2 + \Lambda\eta \Sc_{\infty} \right)}{\Lambda\eta (1 - \kappa)(s+1)} < +\infty , \vspace{1ex}\\
		\sum_{k=0}^{\infty} (k + 3s + \tau )(k + 2s + \tau)  \Expn{\norms{Gy^{k+1} - Gy^k}^2} &\leq& \frac{\tau\theta \Rc_0^2 + \Lambda\eta \Sc_{\infty} }{\beta \tau\theta} < +\infty, \vspace{1ex}\\
		\sum_{k=0}^{\infty} (k + 3s + \tau)(k + 2s + \tau)  \Expn{\Dc(y^{k+1}, y^k)} &\leq& \frac{ \tau\theta \Rc_0^2 +  \Lambda\eta \Sc_{\infty} }{\bar{\beta} \tau\theta } < +\infty.
	\end{array}
	\hspace{-1ex}
	\end{equation}
	
	%%% (b)~Summability results.
	\vspace{0ex}
	\noindent\textbf{$\mathrm{(ii)}$~Summability bounds.} We divide the proof of this part into three steps.
	
	%%% Step 1.
	\vspace{1ex}
	\noindent\textit{\underline{Step 1}. Prove the summability result for $\Expn{\norms{e^k}^2}$ and $\Expn{\norms{\widetilde{G}^k}^2}$.}
	First, taking the total expectation of the first line of \eqref{eq:iFKM_error_bound}, and using the third line of \eqref{eq:iFKM4CE_th_proof1.5}, we obtain
	\begin{equation}\label{eq:iFKM4CE_th_ek_sum}
		\arraycolsep=0.2em
		\begin{array}{lcl}
			\displaystyle \sum_{k=0}^\infty (k + 3s + \tau)(k + 2s + \tau) \Expn{\norms{e^k}^2} &<& +\infty.
		\end{array}
	\end{equation}
	Moreover, since $e^k := \widetilde{G}^k - Gy^k$, by Young's inequality, we have $\norms{\widetilde{G}^k}^2 \leq 2\norms{Gy^k}^2 + 2\norms{e^k}^2$.
	Next, using the last two relations and the second line of \eqref{eq:iFKM4CE_th_proof1.5}, we obtain
	\begin{equation}\label{eq:iFKM4CE_th_Gtilde_sum}
	\arraycolsep=0.2em
	\begin{array}{lcl}
		\displaystyle \sum_{k=0}^\infty (k + 3s + \tau) \Expn{\norms{\widetilde{G}^k}^2} &<&  +\infty.
	\end{array}
	\end{equation}
	
	%%% Step 2.
	\vspace{1ex}
	\noindent{\textit{\underline{Step 2}. Prove the summability result for $\Expn{\norms{x^{k+1} - x^k}^2}$.}}
	From \eqref{eq:iFKM4CE}, we can derive that
	\begin{equation*}
		\arraycolsep=0.2em
		\begin{array}{lcl}
			t_{k-1}(x^{k+1} - x^k + \eta_k \widetilde{G}^k) &=& s(z^k - x^k), \vspace{1ex}\\
			
			(t_{k-2} - s)(x^k - x^{k-1} + \eta_{k-1} \widetilde{G}^{k-1}) &=& s(z^{k-1} - x^k - \eta_{k-1} \widetilde{G}^{k-1}) \vspace{1ex}\\
			&=& s \big[z^k - x^k - (1 - \frac{\gamma}{s})\eta_{k-1} \widetilde{G}^{k-1} \big].
		\end{array}
	\end{equation*}
	Combining these two expressions, we obtain
	\begin{equation}\label{eq:iFKM4CE_th_proof2}
		\arraycolsep=0.2em
		\begin{array}{lcl}
			t_{k-1}(x^{k+1} - x^k + \eta_k \widetilde{G}^k) &=& (t_{k-2} - s)(x^k - x^{k-1} + \eta_{k-1} \widetilde{G}^{k-1}) + (s - \gamma)\eta_{k-1} \widetilde{G}^{k-1}.
		\end{array}
	\end{equation}
	Denoting $v^k := x^{k+1} - x^k + \eta_k \widetilde{G}^k$.
	Then, the last expression can be rewritten as
	\begin{equation*}
		\arraycolsep=0.2em
		\begin{array}{lcl}
			t_{k-1}v^k &=& (1 - \frac{s}{t_{k-2}})t_{k-2} v^{k-1} + \frac{s(s - \gamma)\eta_{k-1}t_{k-2}}{st_{k-2}} \widetilde{G}^{k-1}.
		\end{array}
	\end{equation*}
	By the convexity of $\norms{\cdot}^2$ and $\frac{s}{t_{k-2}} \in (0, 1]$, since $\eta_k = \frac{\eta t_k}{2(t_k - s)} \leq \eta$, we have
	\begin{equation}\label{eq:iFKM4CE_th_proof2.5}
		\arraycolsep=0.2em
		\begin{array}{lcl}
			t_{k-1}^2 \norms{v^k}^2 &\leq& (1 - \frac{s}{t_{k-2}}) t_{k-2}^2 \norms{v^{k-1}}^2 + \frac{(s - \gamma)^2 \eta_{k-1}^2 t_{k-2}}{s} \norms{\widetilde{G}^{k-1}}^2 \vspace{1ex}\\
			&\leq& t_{k-2}^2 \norms{v^{k-1}}^2 - s t_{k-2}\norms{v^{k-1}}^2 + \frac{(s - \gamma)^2 \eta^2}{s} t_{k-2} \norms{\widetilde{G}^{k-1}}^2.
		\end{array}
	\end{equation}
	Since $\sum_{k=0}^\infty t_{k-1} \Expn{\norms{\widetilde{G}^k}^2} < +\infty$ due to \eqref{eq:iFKM4CE_th_Gtilde_sum}, taking the total expectation of the last inequality and then applying Lemma~\ref{lem:tech_convergence1}, we can show that
	\begin{equation}\label{eq:iFKM4CE_th_proof3.0}
		\arraycolsep=0.2em
		\begin{array}{lcl}
			\lim_{k \to \infty} t_{k-1}^2 \Expn{\norms{v^k}^2} \text{ exists \quad and \quad } \sum_{k=0}^\infty t_{k-1} \Expn{\norms{v^k}^2} < +\infty.
		\end{array}
	\end{equation}
	Using Lemma~\ref{lem:tech_convergence2}, the last two relations imply that
	\begin{equation}\label{eq:iFKM4CE_th_proof3}
		\arraycolsep=0.2em
		\begin{array}{lcl}
			\lim_{k \to \infty} t_{k-1}^2 \Expn{\norms{v^k}^2} = 0.
		\end{array}
	\end{equation}
	Furthermore, from Young's inequality and $\eta_k \leq \eta$, we have 
	\begin{equation*}
		\arraycolsep=0.2em
		\begin{array}{lcl}
			\norms{x^{k+1} - x^k}^2 \leq 2 \norms{x^{k+1} - x^k + \eta_k \widetilde{G}^k}^2 + 2 \eta_k^2 \norms{\widetilde{G}^k}^2 \leq 2\norms{v^k}^2 + 2\eta^2 \norms{\widetilde{G}^k}^2.
		\end{array}
	\end{equation*}	
	Using this relation, \eqref{eq:iFKM4CE_th_Gtilde_sum}, and \eqref{eq:iFKM4CE_th_proof3.0}, we can prove that
	\begin{equation*}%\label{eq:iFKM4CE_th_proof4}
		\arraycolsep=0.2em
		\begin{array}{lcl}
			\sum_{k=0}^\infty t_{k-1} \Expn{\norms{x^{k+1} - x^k}^2}  < +\infty.
		\end{array}
	\end{equation*}
	This proves the first line of \eqref{eq:iFKM4CE_th_summability_bounds} due to $t_{k-1} = k + 3s + \tau - 1 \geq k + \tau$.
	
	%%%% Step 3.
	\vspace{1ex}
	\noindent{\textit{\underline{Step 3}. Prove the summability result for $\Expn{\norms{Gy^k}^2}$.}}
	From \eqref{eq:iFKM4CE_th_proof2}, we can show that
	\begin{equation*}
		\arraycolsep=0.2em
		\begin{array}{lcl}
			t_{k-1}[x^{k+1} - x^k + \eta_k (\widetilde{G}^k - Gy^k)] &=& (t_{k-2} - s)[x^k - x^{k-1} + \eta_{k-1} (\widetilde{G}^{k-1} - Gy^{k-1})] \vspace{1ex}\\
			&&  + {~} (s - \gamma)\eta_{k-1} \widetilde{G}^{k-1} - t_{k-1}\eta_k Gy^k + (t_{k-2} - s) \eta_{k-1} Gy^{k-1}.
		\end{array}
	\end{equation*}
	Let us define $w^k := x^{k+1} - x^k + \eta_k e^k$. % for $e^k := \widetilde{G}^k - Gy^k$. 
	Then, the last expression can be rewritten as
	\begin{equation*}
		\arraycolsep=0.2em
		\begin{array}{lcl}
			t_{k-1}w^k  &=& (t_{k-2} - s)w^{k-1} +  (s - \gamma)\eta_{k-1} \widetilde{G}^{k-1} - t_{k-1}\eta_k Gy^k + (t_{k-2} - s) \eta_{k-1} Gy^{k-1},% \vspace{1ex}\\
%			&=& (1 - \frac{s}{t_{k-2}}) t_{k-2} \left[ w^{k-1} - \frac{t_{k-1} \eta_k}{t_{k-2} - s} (Gy^k - Gy^{k-1}) \right] \vspace{1ex}\\
%			&& + {~} \frac{s}{t_{k-2}} \cdot \frac{t_{k-2}}{s} \left[ \big( (t_{k-2} - \gamma) \eta_{k-1} - t_{k-1} \eta_k \big) Gy^{k-1} + (s - \gamma) \eta_{k-1} e^{k-1} \right].
		\end{array}
	\end{equation*}
	leading to
	\begin{equation*}
	\arraycolsep=0.2em
	\begin{array}{lcl}
		t_{k-1}[w^k + \eta_k (Gy^k - Gy^{k-1}) ]  &=& \frac{s}{t_{k-2}} \cdot \frac{t_{k-2}}{s} \left[ \big( (t_{k-2} - \gamma) \eta_{k-1} - t_{k-1} \eta_k \big) Gy^{k-1} + (s - \gamma) \eta_{k-1} e^{k-1} \right] \vspace{1ex}\\
		&& + {~} (1 - \frac{s}{t_{k-2}}) t_{k-2} w^{k-1}.
	\end{array}
	\end{equation*}	
	Using the convexity of $\norms{\cdot}^2$, $x^{k+1} - x^k = y^k - y^{k-1} - (\eta_k \widetilde{G}^k - \eta_{k-1}\widetilde{G}^{k-1})$ from the first line of \eqref{eq:iFKM4CE}, Young's inequality, Assumption~\ref{as:A1}(ii), and $\eta_k \leq \eta$, we can show that
	\begin{equation}\label{eq:iFKM4CE_th_proof5}
		\arraycolsep=0.2em
		\hspace{-2ex}
		\begin{array}{lcl}
			t_{k-1}^2 \norms{w^k}^2 &\leq& t_{k-2} (t_{k-2} - s)  \norms{w^{k-1}}^2
			- 2t_{k-1}^2\eta_k \iprods{Gy^k - Gy^{k-1}, x^{k+1} - x^k} \vspace{1ex}\\
			&& - {~} 2t_{k-1}^2\eta_k \iprods{Gy^k - Gy^{k-1}, e^k} - t_{k-1}^2\eta_k^2 \norms{Gy^k - Gy^{k-1}}^2
			 \vspace{1ex}\\
			&& + {~} \frac{t_{k-2}}{s} \norms{\big( (t_{k-2} - \gamma) \eta_{k-1} - t_{k-1} \eta_k \big) Gy^{k-1} + (s - \gamma) \eta_{k-1} e^{k-1}}^2 \vspace{1ex}\\
			%%
%			&\leq& t_{k-2} (t_{k-2} - s) \norms{w^{k-1}}^2 - 2t_{k-1}^2\eta_k \iprods{Gy^k - Gy^{k-1}, y^k - y^{k-1}} \vspace{1ex}\\
%			&& - {~} 2t_{k-1}^2\eta_k \iprods{Gy^k - Gy^{k-1}, e^k} - t_{k-1}^2\eta_k^2 \norms{Gy^k - Gy^{k-1}}^2 \vspace{1ex}\\
%			&& + {~} 2t_{k-1}^2\eta_k \iprods{Gy^k - Gy^{k-1}, \eta_k \widetilde{G}^k - \eta_{k-1}\widetilde{G}^{k-1}} \vspace{1ex}\\
%			&& + {~} \frac{t_{k-2} [(t_{k-2} - \gamma) \eta_{k-1} - t_{k-1} \eta_k]^2}{s}  \norms{Gy^{k-1}}^2 + \frac{t_{k-2}(s - \gamma)^2 \eta_{k-1}^2}{s} \norms{e^{k-1}}^2 \vspace{1ex}\\
%			&\leq& t_{k-2} (t_{k-2} - s) \norms{w^{k-1}}^2 - 2\beta t_{k-1}^2\eta_k \norms{Gy^k - Gy^{k-1}}^2 - 2\bar{\beta} t_{k-1}^2\eta_k \Dc(y^k, y^{k-1}) \vspace{1ex}\\
%			&& + {~} t_{k-1}^2\eta_k \norms{Gy^k - Gy^{k-1}}^2 + t_{k-1}^2\eta_k \norms{e^k}^2 - t_{k-1}^2\eta_k^2 \norms{Gy^k - Gy^{k-1}}^2 \vspace{1ex}\\
%			&& + {~} t_{k-1}^2\eta_k \norms{Gy^k - Gy^{k-1}}^2 + t_{k-1}^2\eta_k  \norms{\eta_k \widetilde{G}^k - \eta_{k-1}\widetilde{G}^{k-1}}^2 \vspace{1ex}\\
%			&& + {~} \frac{t_{k-2} [(t_{k-2} - \gamma) \eta_{k-1} - t_{k-1} \eta_k]^2}{s}  \norms{Gy^{k-1}}^2 + \frac{t_{k-2}(s - \gamma)^2 \eta_{k-1}^2}{s} \norms{e^{k-1}}^2 \vspace{1ex}\\
			&\leq& t_{k-2} (t_{k-2} - s) \norms{w^{k-1}}^2 - t_{k-1}^2\eta_k (\eta_k + 2\beta) \norms{Gy^k - Gy^{k-1}}^2 \vspace{1ex}\\
			&& - {~} 2\bar{\beta} t_{k-1}^2\eta_k \Dc(y^k, y^{k-1}) + 2\eta t_{k-1}^2 \norms{Gy^k - Gy^{k-1}}^2 + \eta t_{k-1}^2 \norms{e^k}^2 \vspace{1ex}\\
			&& + {~} t_{k-1}^2 \eta_k \norms{\eta_k \widetilde{G}^k - \eta_{k-1}\widetilde{G}^{k-1}}^2 + \frac{t_{k-2} [(t_{k-2} - \gamma) \eta_{k-1} - t_{k-1} \eta_k]^2}{s}  \norms{Gy^{k-1}}^2  \vspace{1ex}\\
			&& + {~} \frac{(s - \gamma)^2 \eta^2 t_{k-2}}{s} \norms{e^{k-1}}^2.
		\end{array} 
		\hspace{-4ex}
	\end{equation}
	Since $\eta_k \widetilde{G}^k - \eta_{k-1}\widetilde{G}^{k-1} 
%	= (\eta_k Gy^k - \eta_{k-1}Gy^{k-1}) + (\eta_k e^k - \eta_{k-1}e^{k-1}) 
	= \eta_k (Gy^k - Gy^{k-1}) + (\eta_k - \eta_{k-1})Gy^{k-1} + \eta_k e^k - \eta_{k-1}e^{k-1}$, 
	where $\eta_k - \eta_{k-1} = -\frac{s\eta}{2(t_k - s)(t_k - s - 1)}$, using Young's inequality, $\eta_k \leq \eta$, and $\frac{s^2 \eta^2 t_{k-1}^2 \eta_k}{4(t_k - s)^2(t_k - s - 1)^2} \leq \frac{s^2 \eta^3}{(t_{k-1} - s)^2}$, we can show that
	\begin{equation*}
	\arraycolsep=0.2em
	\begin{array}{lcl}
		t_{k-1}^2 \eta_k \norms{\eta_k \widetilde{G}^k - \eta_{k-1}\widetilde{G}^{k-1}}^2 
		&\leq& 4t_{k-1}^2 \eta_k^3 \norms{Gy^k - Gy^{k-1}}^2 + \frac{4s^2 \eta^2 t_{k-1}^2 \eta_k}{4(t_k - s)^2(t_k - s - 1)^2}\norms{Gy^{k-1}}^2 \vspace{1ex}\\
		&& + {~} 4t_{k-1}^2\eta_k^3 \norms{e^k}^2 + 4t_{k-1}^2 \eta_k \eta_{k-1}^2 \norms{e^{k-1}}^2 \vspace{1ex}\\
		&\leq& 4\eta^3 t_{k-1}^2 \norms{Gy^k - Gy^{k-1}}^2 + \frac{4s^2 \eta^3}{(t_{k-1} - s)^2} \norms{Gy^{k-1}}^2 \vspace{1ex}\\
		&& + {~} 4\eta^3 t_{k-1}^2 \norms{e^k}^2 + 4\eta^3 t_{k-1}^2 \norms{e^{k-1}}^2.
	\end{array}
	\end{equation*}
	Moreover, we can also evaluate 
	\begin{equation*}
		\arraycolsep=0.2em
		\begin{array}{lcl}
			\left|(t_{k-2} - \gamma) \eta_{k-1} - t_{k-1} \eta_k\right| &=& \frac{\eta}{2} \cdot \frac{t_k - 1}{t_k - s - 1} \cdot \frac{(\gamma + 1)t_k - s(\gamma + 2)}{t_k - s} \leq \frac{\eta}{2} \cdot \frac{3}{2} \cdot 2 = \frac{3\eta}{2}.
		\end{array}
	\end{equation*}
	Substituting the last two relations into \eqref{eq:iFKM4CE_th_proof5}, we obtain
	\begin{equation}\label{eq:iFKM4CE_th_proof6}
		\arraycolsep=0.2em
%		\hspace{-2ex}
		\begin{array}{lcl}
			t_{k-1}^2 \norms{w^k}^2 
			&\leq& t_{k-2}^2 \norms{w^{k-1}}^2 - st_{k-2} \norms{w^{k-1}}^2 - t_{k-1}^2 \eta_k (\eta_k + 2\beta) \norms{Gy^k - Gy^{k-1}}^2 \vspace{1ex}\\
			&& - {~} 2\bar{\beta} t_{k-1}^2\eta_k \Dc(y^k, y^{k-1}) +  2\eta (1+2\eta^2) t_{k-1}^2 \norms{Gy^k - Gy^{k-1}}^2 \vspace{1ex}\\
			&& + {~}  \left[\frac{9\eta^2 t_{k-2}}{4s} + \frac{4s^2\eta^3}{(t_{k-1} - s)^2} \right]  \norms{Gy^{k-1}}^2  +  \eta(1+4\eta^2) t_{k-1}^2 \norms{e^k}^2 \vspace{1ex}\\
			&& + {~}  \left[\frac{(s - \gamma)^2 \eta^2 t_{k-2}}{s} + 4\eta^3 t_{k-1}^2 \right] \norms{e^{k-1}}^2.
		\end{array} 
%		\hspace{-2ex}
	\end{equation}
	Since the last four terms of \eqref{eq:iFKM4CE_th_proof6} are summable due to \eqref{eq:iFKM4CE_th_BigO_rates} and \eqref{eq:iFKM4CE_th_ek_sum}, applying Lemma~\ref{lem:tech_convergence1} to \eqref{eq:iFKM4CE_th_proof6} after taking the total expectation, we can show that
	\begin{equation}\label{eq:iFKM4CE_th_proof7.1}
	\arraycolsep=0.2em
	\begin{array}{lcl}
		\lim_{k \to \infty} t_{k-1}^2 \Expn{\norms{w^k}^2} \text{ exists \quad and \quad } \sum_{k=0}^\infty t_{k-1} \Expn{\norms{w^k}^2} < +\infty.
	\end{array}
	\end{equation}	
	Using Lemma~\ref{lem:tech_convergence2}, the last two relations imply that
	\begin{equation}\label{eq:iFKM4CE_th_proof7.2}
		\arraycolsep=0.2em
		\begin{array}{lcl}
			\lim_{k \to \infty} t_{k-1}^2 \Expn{\norms{w^k}^2} = 0.
		\end{array}
	\end{equation}	
	Moreover, since $\eta_k \geq \frac{\eta}{2}$, by Young's inequality, we have
	\begin{equation}\label{eq:iFKM4CE_th_proof7.3}
		\arraycolsep=0.2em
		\begin{array}{lcl}
			\frac{\eta^2}{4} \norms{Gy^k}^2 &\leq& \eta_k^2 \norms{Gy^k}^2 \leq 2\norms{x^{k+1} - x^k + \eta_k \widetilde{G}^k}^2 + 2\norms{x^{k+1} - x^k + \eta_k (\widetilde{G}^k - Gy^k)}^2 \vspace{1ex}\\
			&=& 2\norms{v^k}^2 + 2\norms{w^k}^2.
		\end{array}
	\end{equation}	
	Using this fact, \eqref{eq:iFKM4CE_th_proof3.0}, and \eqref{eq:iFKM4CE_th_proof7.1}, we can prove the second line of \eqref{eq:iFKM4CE_th_summability_bounds}.
	
	%%%% (c) Small-o convergence rates.
	\vspace{0.5ex}
	\noindent\textbf{$\mathrm{(iii)}$}~\textbf{The $\SmallOs{1/k}$ convergence rates.}
	By Young's inequality and $\eta_k \leq \eta$, we have
	\begin{equation}\label{eq:iFKM4CE_th_proof7.4}
	\arraycolsep=0.2em
	\begin{array}{lcl}
		\norms{x^{k+1} - x^k}^2 \leq 2\norms{x^{k+1} - x^k + \eta_k e^k}^2 + 2\eta_k^2 \norms{e^k}^2 \leq 2\norms{w^k}^2 + 2\eta^2 \norms{e^k}^2.
	\end{array}
	\end{equation}	
	Combining this fact, \eqref{eq:iFKM4CE_th_proof7.2}, and \eqref{eq:iFKM4CE_th_ek_sum}, we get the first line of \eqref{eq:iFKM4CE_th_smallo_rates}.
	Finally, using \eqref{eq:iFKM4CE_th_proof7.3}, \eqref{eq:iFKM4CE_th_proof3} and \eqref{eq:iFKM4CE_th_proof7.2}, we obtain the second line of \eqref{eq:iFKM4CE_th_smallo_rates}.
	
	\vspace{0.5ex}
	\noindent\textbf{$\mathrm{(iv)}$~Iteration-complexity.}
	If $\delta_k := \frac{\sigma^2}{(k+r)^{1+\nu}}$ for some $r, \nu > 0$ and a given $\sigma > 0$, we have $\Sc_k \leq \Sc_\infty = \sum_{k=0}^\infty \frac{\sigma^2}{(k+r)^{1+\nu}} = \Gamma\sigma^2 < +\infty$, where $\Gamma := \sum_{k=0}^\infty \frac{1}{(k+r)^{1+\nu}} < +\infty$.
	Then, from \eqref{eq:iFKM4CE_th_BigO_rates}, we have
	\begin{equation*}
		\arraycolsep=0.2em
		\begin{array}{lcl}
			\Expn{\norms{Gy^k}^2} &\leq& \frac{4}{\eta (k + 3s + \tau - 1)^2} \left(\Rc_0^2 + \frac{\Lambda\eta \Gamma \sigma^2}{\theta \tau}\right) \vspace{1ex}\\
			&=& \frac{2(3s + \tau - 1)^2 \norms{Gy^0}^2}{(k + 3s + \tau - 1)^2} 
			+ \frac{8s^3\norms{y^0 - x^{\star}}^2 }{\gamma\eta^2 (k + 3s + \tau - 1)^2}
			+ \frac{4\Lambda\Gamma \sigma^2}{ \theta  \tau (k + 3s + \tau - 1)^2}.
		\end{array}
	\end{equation*}
	Thus, for a given tolerance $\epsilon > 0$, to guarantee $\Expn{\norms{Gy^k}^2} \leq \epsilon^2$, we impose 
	\begin{equation}\label{eq:iFKM4CE_th_proof8}
		\arraycolsep=0.2em
		\hspace{-1ex}
		\begin{array}{lcl}
			\frac{4\Lambda\Gamma \sigma^2}{ \theta  \tau (k + 3s + \tau - 1)^2} \leq \frac{\epsilon^2}{3}, \qquad 
			\frac{2(3s + \tau - 1)^2 \norms{Gy^0}^2}{(k + 3s + \tau - 1)^2} \leq \frac{\epsilon^2}{3},  \qquad  \text{and}  \qquad  
			\frac{8s^3\norms{y^0 - x^{\star}}^2 }{\gamma\eta^2 (k + 3s + \tau - 1)^2} \leq \frac{\epsilon^2}{3}.
		\end{array}
		\hspace{-1ex}
	\end{equation}
	The first two conditions hold if 
	$k \geq \Big\lfloor \frac{2\sigma \sqrt{3\Lambda \Gamma}}{\epsilon \sqrt{\theta  \tau}} \Big\rfloor$ 
	and $k \geq \Big\lfloor \frac{\sqrt{6}(3s+\tau-1)\norms{Gy^0} }{\epsilon}  \Big\rfloor$, respectively.
	
	\noindent\textit{$\diamond$}
	If $0 < \bar{\beta} \leq \frac{7\Lambda \hat{\Theta} \beta}{7\Lambda \Theta + 3\theta (1+\tau)}$, then we choose $\eta := \frac{3\theta \bar{\beta}}{7 \hat{\Theta} \Lambda}$, and the third condition in \eqref{eq:iFKM4CE_th_proof8} holds if $k \geq \Big\lfloor \frac{14\sqrt{6} s^{3/2} \hat{\Theta} \Lambda \norms{y^0 - x^{\star}}}{3\theta  \bar{\beta} \gamma^{1/2} \epsilon }  \Big\rfloor$.
	Combining all the above results, we need at most
	\begin{equation*}
		\arraycolsep=0.2em
		\begin{array}{lcl}
			k &:=& \max\set{ \Big\lfloor \frac{2\sigma \sqrt{3\Lambda \Gamma}}{\epsilon \sqrt{\theta  \tau}} \Big\rfloor,
			\Big\lfloor \frac{\sqrt{6}(3s+\tau-1)\norms{Gy^0} }{\epsilon}  \Big\rfloor,
			\Big\lfloor \frac{14\sqrt{6} s^{3/2} \hat{\Theta} \Lambda \norms{y^0 - x^{\star}}}{3\theta  \bar{\beta} \gamma^{1/2} \epsilon }  \Big\rfloor
			} %\vspace{1ex}\\
			= \BigO{\frac{\sigma}{\epsilon\sqrt{\tau}} + \frac{\tau}{\epsilon} + \frac{\hat{\Theta}}{\bar{\beta}\epsilon}}
		\end{array}
	\end{equation*}
	iterations to achieve $\Expn{\norms{Gy^k}^2} \leq \epsilon^2$.

	\noindent\textit{$\diamond$}
	If $\bar{\beta} = 0$ or $\frac{7\Lambda \hat{\Theta} \beta}{7\Lambda \Theta + 3\theta (1+\tau)} \leq \bar{\beta} \leq \frac{\hat{\Theta}}{\Theta}\beta$, then we choose $\eta := \frac{3\theta \beta}{7\Lambda \Theta + 3\theta (1+\tau)}$.
	In this case, if 
	$k \geq \Big\lfloor \frac{2\sqrt{6}s^{3/2} [7\Lambda \Theta + 3\theta (1+\tau)]\norms{y^0 - x^{\star}}}{3 \theta  \beta \gamma^{1/2} \epsilon }  \Big\rfloor$, then the third condition in \eqref{eq:iFKM4CE_th_proof8} holds.
	Combining all the above results, we need at most
	\begin{equation*}
		\arraycolsep=0.2em
		\begin{array}{lcl}
			k &:=& \max\set{ \Big\lfloor \frac{2\sigma \sqrt{3\Lambda \Gamma}}{\epsilon \sqrt{\theta  \tau}} \Big\rfloor,
			\Big\lfloor \frac{\sqrt{6}(3s+\tau-1)\norms{Gy^0} }{\epsilon}  \Big\rfloor,
			\Big\lfloor \frac{2\sqrt{6}s^{3/2} [7\Lambda \Theta + 3\theta (1+\tau)]\norms{y^0 - x^{\star}}}{3 \theta  \beta \gamma^{1/2} \epsilon }  \Big\rfloor
			} \vspace{1ex}\\
			&=& \BigO{\frac{\sigma}{\epsilon\sqrt{\tau}} + \frac{\tau}{\epsilon} + \frac{\Theta + \tau}{\beta\epsilon}}
		\end{array}
	\end{equation*}
	iterations to achieve $\Expn{\norms{Gy^k}^2} \leq \epsilon^2$ in expectation.
\Eproof	
\end{proof}
%%% End of Proof

%%% Remarks
\begin{remark}\label{re:convergence_on_Gx}$($\textbf{Convergence of $\Exp{\norms{Gx^k}^2}$}$)$
Theorem~\ref{th:iFKM4CE_convergence_expectation} establishes convergence on $\Exp{\norms{Gy^k}^2}$.
This criterion is different from standard Nesterov's accelerated methods in convex optimization, where convergence is certified on the sequence $\sets{x^k}$ for the gradient norm.
It is similar to the Ravine method studied in \cite{attouch2022ravine}.
Nevertheless, by Young's inequality and Assumption~\ref{as:A1}, we have $\norms{Gx^k}^2 \leq 2\norms{Gx^k}^2 + 2\norms{Gy^k - Gx^k}^2 \leq 2\norms{Gx^k}^2 + \frac{2}{\beta^2}\norms{y^k - x^k}^2$.
Moreover, one can show that $\sum_{k=0}^{\infty}\Exp{\norms{y^k - x^k}^2 } < +\infty$.
As a result, it is possible to derive convergence rate results for $\Exp{\norms{Gx^k}^2}$ analogous to those for $\Exp{\norms{Gy^k}^2}$ in Theorem~\ref{th:iFKM4CE_convergence_expectation}.
For brevity, we omit the details of this proof.
\end{remark}

\begin{remark}\label{re:choice_of_params}$($\textbf{Choices of parameters}$)$
	We note that the parameter ranges in Theorem~\ref{th:iFKM4CE_convergence_expectation}, such as  $s \geq 1 + 3\gamma$ originated from Lemma~\ref{le:iFKM4CE_lower_bound_of_Lk}, are chosen to make the proof tractable and are not necessarily optimized.
	When implementing the algorithm, one may choose  values outside these strict ranges (e.g., we can choose $s > 1$) to achieve a better empirical performance, but we omit these heuristics here to avoid complicating the theoretical analysis.
\end{remark}
 
\begin{remark}\label{re:dependence_of_tau}$($\textbf{Linear dependence of the iteration-complexity on $\tau$}$)$.
Our iteration-complexity depends \textbf{linearly} on the maximum delay $\tau$, and either $\hat{\Theta}$, or $\Theta + \tau$.
In Section~\ref{sec:app_of_iFKM}, where we specialize our framework for concrete settings, we establish that either $\Theta = \BigOs{\tau}$ or $\hat{\Theta} = \BigOs{\tau}$.
Consequently, our iteration-complexity can be simplified to either
\vspace{-0.5ex}
\begin{equation*}
	\BigO{\frac{\sigma}{\epsilon \sqrt{\tau}} + \frac{\tau}{\beta \epsilon}} \quad \text{or} \quad \BigO{\frac{\sigma}{\epsilon \sqrt{\tau}} + \frac{\tau}{\bar{\beta} \epsilon}}.
\vspace{-1ex}	
\end{equation*}
This demonstrates a \textbf{linear dependence} on the maximum delay $\tau$, which is known to be the optimal dependence in asynchronous and delayed methods for optimization \cite{arjevani2020tight}.
Since $\beta$ and $\bar{\beta}$ are co-coercivity constants, they are of the order $\BigOs{1/L}$, where $L$ denotes the Lipschitz constant of the operator $G$ (or, in particular, the smoothness constant of the objective function in optimization).
Therefore, our  complexity match the linear dependence on $\tau$ commonly found in optimization literature as well (e.g., \cite{arjevani2020tight,stich2020error} for SGD with delayed updates), which typically yields a complexity of $\widetilde{\mcal{O}}\left(\frac{\sigma^2}{\epsilon^2} + \frac{L\tau}{\epsilon}\right)$ for $L$-smooth functions.
%%
%%Nevertheless, unlike the aforementioned results which focus on optimization and are non-accelerated, our method addresses fixed-point and root-finding problems while leveraging Nesterov’s acceleration strategy. To the best of our knowledge, this is the first unified result for root-finding problems that simultaneously achieves optimal dependence on the maximum delay $\tau$ and incorporates acceleration techniques.
\end{remark}

\noindent\textbf{$\mathrm{(d)}$~Main result 2: Almost sure convergence.}
Our second main result is the almost sure convergence of our framework \eqref{eq:iFKM4CE} under the general error approximation condition~\eqref{eq:iFKM_error_bound}.
%%% Theorem 3.2: Almost sure convergence
\begin{theorem}\label{th:iFKM4CE_convergence_almost_sure}
	Under the same conditions and settings as in Theorem \ref{th:iFKM4CE_convergence_expectation}, if additionally $S_{\infty} := \sum_{k=0}^{\infty} \delta_k < +\infty$, then the following summability bounds hold almost surely:
	\begin{equation}\label{eq:iFKM4CE_as_summability}
		\arraycolsep=0.2em
		\begin{array}{lcl}
			\sum_{k=0}^\infty (k +  \tau) \norms{Gy^k}^2 < +\infty 
			\qquad \textrm{and} \qquad 
			\sum_{k=0}^\infty (k+ \tau) \norms{x^{k+1} - x^k}^2 < +\infty.
		\end{array}
	\end{equation}
	Moreover, we also have the following almost sure limits $($i.e., $\SmallOs{1/k^2}$-almost sure rates$)$:
	\begin{equation}\label{eq:iFKM4CE_as_limits}
		\arraycolsep=0.2em
		\begin{array}{lcl}
			\lim_{k\to\infty} (k + \tau)^2 \norms{Gy^k}^2 = 0 
			\qquad \textrm{and} \qquad 
			\lim_{k\to\infty} (k + \tau)^2 \norms{x^{k+1} - x^k}^2 = 0.
		\end{array}
	\end{equation}
	Finally, the iterate sequences $\sets{x^k}$, $\sets{y^k}$, and $\sets{z^k}$ converge almost surely to a $\zer{G}$-valued random variable $x^{\star}$, a solution to \eqref{eq:CE}.
\end{theorem}

%%% Proof of Theorem 3.2
\begin{proof} 
	We divide the proof of this theorem into three main steps as follows.
	
	\noindent\textbf{\textit{Step 1: Prove $\lim_{k\to \infty} t_k^2\norms{Gy^k}^2 = 0$ and $\lim_{k \to \infty} t_k^2 \norms{x^{k+1} - x^k}^2 = 0$ almost surely.}}
	Apply Lemma~\ref{le:Asyn_lemma_A2}(ii) to \eqref{eq:iFKM4CE_le_key_est}, we can show that almost surely,
	\begin{equation}\label{eq:th32_proof1}
	\arraycolsep=0.2em
	\begin{array}{lcl}
		\lim_{k \to \infty} V_k \quad \text{exists}, \vspace{1ex}\\
		\sum_{k=0}^\infty t_k \norms{Gy^k}^2 &<& +\infty, \vspace{1ex}\\
		\sum_{k=0}^\infty t_k(t_k - s) \Delta_k &<& +\infty, \vspace{1ex}\\
		\sum_{k=0}^\infty t_k(t_k - s) \norms{Gy^{k+1} - Gy^k}^2 &<& +\infty, \vspace{1ex}\\
		\sum_{k=0}^\infty t_k(t_k - s) \Dc(y^{k+1}, y^k) &<& +\infty.
	\end{array}
	\end{equation}	
	Since $t_k = k + 3s + \tau > k + \tau$ due to \eqref{eq:iFKM4CE_params}, the second line of \eqref{eq:th32_proof1} proves the first almost sure summability bound of \eqref{eq:iFKM4CE_as_summability}.
	
	Given the last two summability bounds, we can use Lemma~\ref{le:appendix_convergence_series} to show that 
	\begin{equation}\label{eq:th32_proof2}
	\arraycolsep=0.2em
	\begin{array}{lcl}
		\sum_{k=0}^\infty t_k(t_k - s) \sum_{l = [k - \tau_k + 1]_+}^k \norms{Gy^{l} - Gy^{l-1}}^2 &<& +\infty, \vspace{1ex}\\
		\sum_{k=0}^\infty t_k(t_k - s) \sum_{l = [k - \tau_k + 1]_+}^k \Dc(y^l, y^{l-1}) &<& +\infty.
	\end{array}
	\end{equation}	
	Now, from \eqref{eq:iFKM_error_bound}, we have
	\begin{equation*}
	\arraycolsep=0.2em
	\begin{array}{lcl}
		\Expsn{k}{\norms{e^k}^2} &\leq& \Expsn{k}{ \Delta_k}  \leq  (1-\kappa)\Delta_{k-1} + {{\!\!\!\!\!\!\!}\displaystyle\sum_{l=[k-\tau_k+1]_{+}}^k }{\!\!\!\!\!\!\!} \big[\Theta\norms{Gy^l - Gy^{l-1}}^2 + \hat{\Theta}\Dc(y^l, y^{l-1})\big]  + \frac{\delta_k}{t_k (t_k - s)}.
	\end{array}
	\end{equation*}		 
	Multiplying both sides of this inequality by $t_k (t_k - s)$, we can derive that
	\begin{equation*}
		\arraycolsep=0.2em
		\begin{array}{lcl}
			t_k (t_k - s)\Expsn{k}{\norms{e^k}^2} &\leq& t_{k-1} (t_{k-1} - s) \norms{e^{k-1}}^2 - t_{k-1} (t_{k-1} - s) \norms{e^{k-1}}^2 \vspace{1ex}\\
			&& + {~} t_k (t_k - s)\sum_{l=[k-\tau_k+1]_{+}}^k\big[\Theta\norms{Gy^l - Gy^{l-1}}^2 + \hat{\Theta}\Dc(y^l, y^{l-1})\big] \vspace{1ex}\\
			&& + {~} (1-\kappa)t_k (t_k - s)\Delta_{k-1} + \delta_k.
		\end{array}
	\end{equation*}	
	Using \eqref{eq:th32_proof1}, \eqref{eq:th32_proof2}, and $\sum_{k=0}^\infty \delta_k < +\infty$ from the assumption, the last two lines of this inequality are summable almost surely.
	Thus, applying Lemma~\ref{lem:supermartingale}, we conclude that 
	\begin{equation}\label{eq:th32_proof3}
		\arraycolsep=0.2em
		\begin{array}{lcl}
			&\sum_{k=0}^\infty t_k (t_k - s)\norms{e^k}^2 < +\infty \quad  \text{almost surely}, \vspace{1ex}\\
			&\lim_{k\to \infty}t_k (t_k - s)\norms{e^k}^2 = 0 \quad  \text{almost surely.}
		\end{array}
	\end{equation}
	
	Next, since $\norms{\widetilde{G}^k}^2 \leq 2\norms{Gy^k}^2 + 2\norms{e^k}^2$ by Young's inequality, using the first line of \eqref{eq:th32_proof1} and \eqref{eq:th32_proof3}, we can show that
	\begin{equation}\label{eq:th32_proof4}
		\arraycolsep=0.2em
		\begin{array}{lcl}
			\sum_{k=0}^\infty t_k\norms{\widetilde{G}^k}^2 < +\infty \quad \text{almost surely}.
		\end{array}
	\end{equation}
	For $v^k := x^{k+1} - x^k + \eta_k \widetilde{G}^k$ as in Theorem \ref{th:iFKM4CE_convergence_expectation}, taking the conditional expectation $\Expsk{k}{\cdot}$ of \eqref{eq:iFKM4CE_th_proof2.5}, we obtain
	\begin{equation*}
		\arraycolsep=0.2em
		\begin{array}{lcl}
			t_{k-1}^2 \Expsk{k}{\norms{v^k}^2} &\leq&  t_{k-2}^2 \norms{v^{k-1}}^2 - s t_{k-2}\norms{v^{k-1}}^2 + \frac{(s - \gamma)^2 \eta^2}{s} t_{k-2} \norms{\widetilde{G}^{k-1}}^2.
		\end{array}
	\end{equation*}
	Since $\sum_{k=0}^\infty t_{k-1} \norms{\widetilde{G}^k}^2 < +\infty$ almost surely by \eqref{eq:th32_proof4}, applying Lemma~\ref{lem:supermartingale} to the last inequality, and then using Lemma~\ref{lem:tech_convergence2}, we can show that
	\begin{equation}\label{eq:th32_proof5}
		\arraycolsep=0.2em
		\begin{array}{lcl}
			\sum_{k=0}^\infty t_k \norms{v^k}^2 < +\infty \quad \text{and} \quad \lim_{k \to \infty} t_k^2 \norms{v^k}^2 = 0 \quad \text{almost surely}.
		\end{array}
	\end{equation}
	Similarly, for $w^k := x^{k+1} - x^k + \eta_k e^k$ as in Theorem \ref{th:iFKM4CE_convergence_expectation}, using similar arguments but applying to \eqref{eq:iFKM4CE_th_proof6}, we can prove that
	\begin{equation}\label{eq:th32_proof6}
		\arraycolsep=0.2em
		\begin{array}{lcl}
			\sum_{k=0}^\infty t_k \norms{w^k}^2 < +\infty \quad \text{and} \quad \lim_{k \to \infty} t_k^2 \norms{w^k}^2 = 0 \quad \text{almost surely}.
		\end{array}
	\end{equation}
	Combining \eqref{eq:iFKM4CE_th_proof7.4}, \eqref{eq:th32_proof6}, and \eqref{eq:th32_proof3}, we can show that $\sum_{k=0}^\infty t_k \norms{x^{k+1} - x^k}^2 = 0$ almost surely.
	Since $t_k = k + 3s + \tau$ due to \eqref{eq:iFKM4CE_params}, the last expression proves the second almost-sure summability bound in \eqref{eq:iFKM4CE_as_summability}.
	
	Combining \eqref{eq:iFKM4CE_th_proof7.3}, the two limits in \eqref{eq:th32_proof5} and \eqref{eq:th32_proof6}, we can show that 
	\begin{equation*} 
		\arraycolsep=0.2em
		\begin{array}{lcl}
			\lim_{k \to \infty} t_k^2 \norms{Gy^k}^2 = 0 \quad \text{almost surely}.
		\end{array}
	\end{equation*}
	This is exactly the first result of \eqref{eq:iFKM4CE_as_limits}. 
	Similarly, combining \eqref{eq:iFKM4CE_th_proof7.4}, \eqref{eq:th32_proof6}, and \eqref{eq:th32_proof3}, we can show that $\lim_{k \to \infty} t_k^2  \norms{x^{k+1} - x^k}^2 = 0$ almost surely, which is the second result of \eqref{eq:iFKM4CE_as_limits}.
	
	\vspace{1ex}
	\noindent\textbf{\textit{Step 2: Prove the existence of $\lim_{k\to \infty} \norms{z^k - x^{\star}}^2$.}}
	From \eqref{eq:iFKM4CE}, we have $s(z^k - x^k) = t_{k-1}(x^{k+1} - x^k + \eta_k\widetilde{G}^k) = t_{k-1} v^k$.
	Since $\lim t_k^2 \norms{v^k}^2 = 0$ almost surely, we obtain that $\lim_{k \to \infty} \norms{z^k - x^k}^2 = 0$ almost surely. 
	Moreover, from the third line of \eqref{eq:iFKM4CE}, we can show that $y^k - z^k = (1 - \frac{s}{t_{k-1}})(x^k - z^k)$, thus
	\begin{equation*} 
		\arraycolsep=0.2em
		\begin{array}{lcl}
			\norms{y^k - z^k}^2 &=& (1 - \frac{s}{t_{k-1}})^2 \norms{x^k - z^k}^2 \leq \norms{x^k - z^k}^2 \to 0 \quad \text{as $k \to \infty$}, \quad \text{almost surely}.
		\end{array}
	\end{equation*}
	Using Cauchy-Schwarz inequality, the last limit and the first line of \eqref{eq:iFKM4CE_as_limits}, we can show that almost surely
	\begin{equation*} 
		\arraycolsep=0.2em
		\begin{array}{lcl}
			|st_{k-1}\iprods{Gy^k, y^k - z^k}|^2 \leq s^2 t_{k-1}^2 \norms{Gy^k}^2 \norms{y^k - z^k}^2 \to 0 \quad \text{as $k \to \infty$}.
		\end{array}
	\end{equation*}
	Hence, we get $\lim_{k\to\infty} s t_{k-1} \iprods{Gy^k, y^k - z^k} = 0$ almost surely.
	
	Now, from \eqref{eq:iFKM4CE_Lk_func} and \eqref{eq:iFKM4CE_le_Vk}, we can write
	\begin{equation}\label{eq:th32_proof9}
		\arraycolsep=0.2em
		\begin{array}{lcl}
			V_k &=& \frac{a_k}{2}\norms{Gy^k}^2 + s t_{k-1}\iprods{Gy^k, y^k - z^k} + \frac{s^2(s-1)}{2\gamma\eta_{k-1}}\norms{z^k - x^{\star}}^2 \vspace{1ex}\\
			&& + {~} \frac{(2\beta - \eta)t_{k-1}(t_{k-1}-s)}{2} \norms{Gy^k - Gy^{k-1}}^2 + \bar{\beta} t_{k-1} (t_{k-1}-s) \Dc(y^k, y^{k-1}) \vspace{1ex}\\
			&& + {~} \frac{\Lambda\eta\alpha(1-\kappa)}{2[1 - (1-\kappa)\alpha ]} t_{k-1}(t_{k-1}-s)\Delta_{k-1}.
		\end{array}
	\end{equation}
	Collecting all necessary almost sure limits we have proven so far, we have, almost surely,
	\begin{equation*}
		\arraycolsep=0.2em
		\begin{array}{lcl}
			\lim_{k \to \infty} V_k \quad \text{exists},\vspace{1ex}\\
			\lim_{k \to \infty} \frac{a_k}{2}\norms{Gy^k}^2 = \lim_{k \to \infty} t_k^2 \norms{Gy^k}^2 = 0 \quad \text{(due to $a_k = \BigOs{t_k^2}$)}, \vspace{1ex}\\
			\lim_{k \to \infty} s t_{k-1}\iprods{Gy^k, y^k - z^k} = 0, \vspace{1ex}\\
			
			\lim_{k \to \infty} t_{k-1}(t_{k-1}-s) \norms{Gy^k - Gy^{k-1}}^2 = 0, \vspace{1ex}\\
			
			\lim_{k \to \infty} t_{k-1} (t_{k-1}-s) \Dc(y^k, y^{k-1}) = 0, \vspace{1ex}\\
			\lim_{k \to \infty} t_{k-1}(t_{k-1}-s)\Delta_{k-1} = 0.
		\end{array}
	\end{equation*}
	Combining these limits and \eqref{eq:th32_proof9}, and noting that $\eta_{k-1} = \BigOs{1}$, we conclude that $\lim_{k \to \infty} \norms{z^k - x^{\star}}^2$ exists almost surely.
	Moreover, since $\lim_{k \to \infty} \norms{z^k - x^k}^2 = \lim_{k \to \infty} \norms{y^k - z^k}^2 = 0$ almost surely, by triangle inequality, we also obtain that $\lim_{k \to \infty} \norms{x^k - x^{\star}}^2$ and $\lim_{k \to \infty} \norms{y^k - x^{\star}}^2$ exist almost surely.
	
	\vspace{1ex}
	\noindent\textbf{\textit{Step 3: Almost sure convergence of the iterates.}}
	Since $\lim_{k \to \infty}\norms{y^k - x^{\star}}^2$ exists almost surely for any solution $x^{\star} \in \zer{G}$, $\lim_{k \to \infty} \norms{Gy^k} = 0$ almost surely by the first result of \eqref{eq:iFKM4CE_as_limits}, and $G$ is continuous, applying Lemma~\ref{lem:Opial_NE}, we conclude that $\sets{y^k}$ converges almost surely to a random variable $x^{\star} \in \zer{G}$.
	Finally, since $\lim_{k\to \infty} \norms{y^k - z^k} = \lim_{k \to \infty}\norms{z^k - x^k} = 0$ almost surely, using triangle inequality, we can also prove that $\sets{x^k}$ and $\sets{z^k}$  almost surely converge to $x^{\star} \in \zer{G}$.
\Eproof
\end{proof}
%%% End of Proof

\noindent\textbf{Discussion.}
Our almost sure convergence results in Theorem~\ref{th:iFKM4CE_convergence_almost_sure} appear to be the first for asynchronous accelerated methods, even for convex optimization.
Here, we establish summability bounds, $\SmallOs{1/k^2}$ convergence rates, and the convergence of iterates.
Moreover, our framework covers a wide range of algorithms in different settings as discussed in Section~\ref{sec:app_of_iFKM}.

%%% Extensions
\begin{remark}\label{re:extension}[\textbf{Extension to generalized equations}]
We can apply our scheme \eqref{eq:iFKM4CE} to solve the following \textit{\textbf{generalized equation}} (known as inclusion):
\begin{equation}\label{eq:GE}
\textrm{Find $x^{\star} \in \R^p$ such that:}~ 0 \in Gx^{\star} + Tx^{\star},
\tag{GE}
\end{equation}
where $G$ is defined as in \eqref{eq:CE}, while $T: \R^p \rightrightarrows \R^p$ is a possibly multivalued mapping that is $\rho$-co-hypomonotone. 
We sketch the main steps of this application as follows. 
\begin{compactitem}[$\bullet$]
\item First, we equivalently reformulate \eqref{eq:GE} into \eqref{eq:CE} using either a forward-backward splitting (FBS) operator or a backward-forward splitting (BFS) operator as done in \cite{TranDinh2025a}.
\item Next, we develop a similar result as Corollary 1 in \cite{TranDinh2025a} to verify Assumption~\ref{as:A1}.
\item Then, we apply \eqref{eq:iFKM4CE} and its convergence theory to the resulting FBS or BFS equation.
\item Finally, we convert the convergence results of this equation to the original problem \eqref{eq:GE}.  
\end{compactitem}
More generally, our framework \eqref{eq:iFKM4CE} can be applied to solve any problem that can be reformulated into the fixed-point problem \eqref{eq:FP} of a nonexpansive operator $F$. 
\end{remark}
%%%%%%%%%%%%%%%%%%%%%%%%%%%%%%%%%%%%%%%%%%%%%%%%%%%%%%%%%%%
%%% 4. Fast KM Method with Delay Updates
%%%%%%%%%%%%%%%%%%%%%%%%%%%%%%%%%%%%%%%%%%%%%%%%%%%%%%%%%%%
\beforesec
\section{Applications to AFP Methods with Delay and Asynchronicity}\label{sec:app_of_iFKM}
\aftersec
Section~\ref{sec:iFKM_framework} develops our delayed inexact \ref{eq:iFKM4CE} framework under the error approximation condition \eqref{eq:iFKM_error_bound} in Definition~\ref{de:error_bound_cond}.
Now, we  apply our framework \eqref{eq:iFKM4CE} to derive three different algorithms with delayed and asynchronous updates for solving \eqref{eq:CE}, or equivalently,  \eqref{eq:FP}.

\beforesubsec
\subsection{\textbf{\textit{Accelerated Fixed-Point Method with Delayed Oracles}}}\label{subsec:iFKM4CE_delay_universal}
\aftersubsec
In this setting of \eqref{eq:CE}, we do not assume any structure on $G$ and treat it as a universal operator satisfying Assumption~\ref{as:A1}(ii) with $\bar{\beta} = 0$ and $\beta > 0$, i.e., $G$ is $\beta$-co-coercive. 
This setting is equivalent to finding a fixed-point of a nonexpansive operator $F$ in \eqref{eq:FP}.

\beforesubsubsec
\subsubsection{The computing mechanism}\label{subsec:setting1_mechanism}
\aftersubsubsec
We consider a standard \textbf{server-worker system} outlined in Figure~\ref{fig:consistent_architecture} with one central \emph{server} and multiple \emph{workers}. 
The server maintains (i) the decision variables in a \texttt{V}-memory block and (ii) the most recent operator value in a \texttt{G}-memory block. 
All workers have read access to the shared data and can therefore evaluate the operator (though their speeds may differ).

\begin{figure}[hpt!]
\vspace{-0ex}
	\begin{center}
	\includegraphics[scale=0.7]{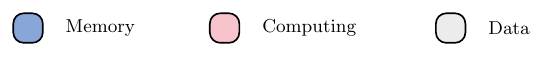}\\
	\includegraphics[scale=0.6]{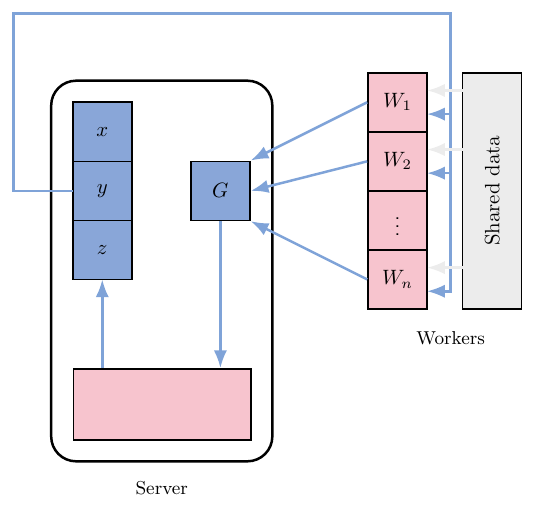} \hspace{1cm}
	\includegraphics[scale=0.6]{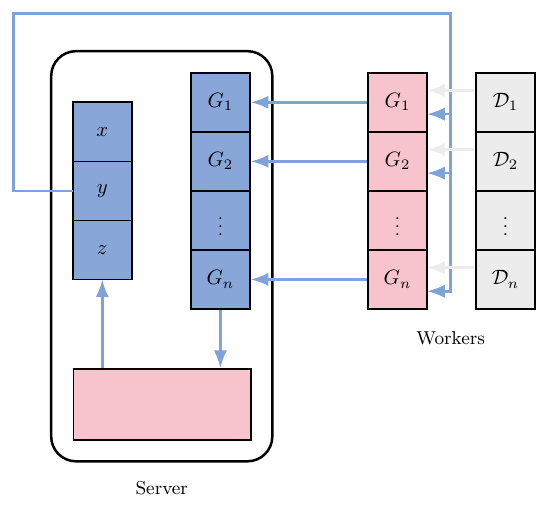}
	\vspace{-1ex}
	\caption{The server-worker architecture with: shared data for solving the universal setting of \eqref{eq:CE} (left) and distributed data for solving the finite-sum setting of \eqref{eq:CE} (right)}
	\label{fig:consistent_architecture}
	\end{center}
\vspace{-6ex}	
\end{figure}

%\vspace{0.5ex}
\noindent$\mathrm{(a)}$~\textbf{\textit{Protocol.}} 
The server initializes \texttt{V}-memory with $y^0$ and broadcasts the current iterate $y^k$ whenever updated. % it is updated. 
Upon receiving $y^k$, a worker computes $Gy^k$ using the shared data and, when finished, attempts to commit the result to \texttt{G}-memory. 
A successful commit triggers the server to form the next iterates $x^{k+1},y^{k+1},z^{k+1}$ and atomically write them to \texttt{V}-memory.

\vspace{0.5ex}
\noindent$\mathrm{(b)}$~\textbf{\textit{Atomic assumption.}} All reads/writes to \texttt{G}- and \texttt{V}-memory are \emph{atomic}: a memory region being written is locked (mutual exclusion), so every reader observes either the old value or the fully written new value -- never a mixture (no torn reads).
In practice, such atomicity can be enforced using some technical or systems techniques, such as a memory lock or double-/dual-memory (double buffering) with an atomic swaps approach  \cite{peng2016arock}.

\vspace{0.5ex}
\noindent$\mathrm{(c)}$~\textbf{\textit{Staleness and estimator.}} 
Because workers are heterogeneous, the value present in the \texttt{G}-memory at the $k$-th server update (i.e., $k$-th iteration) may be a delayed evaluation $Gy^{k-\tau_k}$ with a delay $\tau_k \geq 0$. 
Hence, the estimator used by the server is
\begin{equation}\label{eq:delayed_oracle_of_G}
\arraycolsep=0.2em
\begin{array}{lcl}
	\widetilde{G}^k = Gy^{k-\tau_k},
\end{array}
\end{equation}
which is one of the previously computed operator values in $\sets{Gy^0, \cdots, Gy^k}$. 
This is called a \emph{consistent asynchrony} model: the operator value is one of the historical computed values.

\vspace{0.5ex}
\noindent$\mathrm{(d)}$~\textbf{\textit{Bounded-delay option.}}
To prevent excessively stale information, the server can enforce a threshold $\tau>0$ by discarding any worker result with $\tau_k>\tau$. 
This ensures a uniform bound on delays throughout the procedure, i.e., $0 \leq \tau_k \leq \tau$ for all $k\geq 0$. 

%%%% 4.1. Convergence of FKM with delayed oracle
\beforesubsubsec
\subsubsection{Convergence of \ref{eq:iFKM4CE} with universal delayed oracle}\label{subsec:FKM_with_delayed_oracle}
\aftersubsubsec
Now, we specify Theorems~\ref{th:iFKM4CE_convergence_expectation} and \ref{th:iFKM4CE_convergence_almost_sure} to the setting in Subsection~\ref{subsec:setting1_mechanism} to obtain the following corollary to guarantee convergence of \eqref{eq:iFKM4CE} with delayed oracle \eqref{eq:delayed_oracle_of_G}.

%%% Corollary 4.1.
\begin{corollary}\label{co:iFKM4CE_delayed_convergence_expectation}
	Suppose that Assumption~\ref{as:A1} holds for 	\eqref{eq:CE}.
	Let $\sets{(x^k, y^k, z^k)}$ be generated by \eqref{eq:iFKM4CE} using a delayed oracle $\widetilde{G}^k = Gy^{k-\tau_k}$ as in \eqref{eq:delayed_oracle_of_G} with uniformly bounded delays $0 \leq \tau_k \leq \tau$ for some $\tau \geq 1$, where we choose $s$ and $\eta$, and update $t_k$, $\gamma_k$, and $\eta_k$  as follows:
	\begin{equation*} 
		\arraycolsep=0.2em
		\hspace{-2ex}
		\begin{array}{lcl}
			s \geq 1 + 3\gamma, \quad  0 < \eta \leq \frac{3\beta}{3 + (7\Lambda+3)\tau}, \quad t_k := k + 3s + \tau,  \quad \gamma_k := \gamma \in [0,1],  \quad \eta_k := \frac{\eta t_k}{2(t_k - s)}.
		\end{array}
		\hspace{-2ex}
	\end{equation*}
	Then, the following statements hold.
	\begin{compactitem}
	\item[\textbf{$\mathrm{(i)}$~Convergence in expectation.}] 
	For $\Rc_0^2 := \frac{\eta (3s+\tau-1)^2}{2} \norms{Gy^0}^2 + \frac{2s^3}{\eta \gamma}\norms{y^0 - x^{\star}}^2$, we get
	\begin{equation*} 
		\arraycolsep=0.2em
		\begin{array}{lcl}
			\Expn{\norms{Gy^k}^2} &\leq& \displaystyle \frac{4\Rc_0^2}{\eta (k + 3s + \tau - 1)^2} \quad \text{and} \quad
			\lim_{k \to \infty} (k + \tau)^2 \Expn{\norms{Gy^k}^2} = 0.
		\end{array}
	\end{equation*}
	
	\item[\textbf{$\mathrm{(ii)}$~Almost sure convergence.}]
	The following results hold almost surely:
	\begin{equation*} 
			\sum_{k=0}^\infty (k+\tau) \norms{Gy^k}^2 < +\infty \qquad \text{and} \qquad \lim_{k \to \infty} (k+\tau)^2 \norms{Gy^k}^2 = 0.
	\end{equation*}
	Moreover, the sequences $\sets{x^k}$, $\sets{y^k}$, and $\sets{z^k}$ converge almost surely to a $\zer{G}$-valued random variable $x^{\star}$, a solution to \eqref{eq:CE}.
		
	\item[\textbf{$\mathrm{(iii)}$~Iteration-complexity.}] 
	For any given tolerance $\epsilon > 0$,  we require at most $\BigO{\frac{\tau}{\epsilon} + \frac{\tau}{\beta\epsilon}}$ iterations to achieve $\Expn{\norms{Gy^k}^2} \leq \epsilon^2$.
	\end{compactitem}
\end{corollary}

%%% Proof of Lemma 3.1.
\begin{proof}
Since our update is consistent as in \eqref{eq:delayed_oracle_of_G}, by \eqref{eq:oracle_error}, we have $e^k := \widetilde{G}^k - Gy^k = Gy^{k - \tau_k} - Gy^k$.
Using the convexity of $\norms{\cdot}^2$ and $0 \leq \tau_k \leq \tau$, we have 
\begin{equation*} 
	\arraycolsep=0.2em
	\begin{array}{lcl}
		\norms{e^k}^2 & = & \norms{Gy^{k-\tau_k} - Gy^k}^2 = \norms{\sum_{l = [k-\tau_k + 1]_{+}}^k (Gy^l - Gy^{l-1})}^2 \vspace{1ex}\\
		&=& \tau_k^2 \norms{\sum_{l = [k-\tau_k + 1]_{+}}^k \frac{1}{\tau_k}(Gy^l - Gy^{l-1})}^2 \vspace{1ex}\\
		& \leq & \tau_k \sum_{l = [k-\tau_k + 1]_{+}}^k \norms{Gy^l - Gy^{l-1}}^2 \vspace{1ex}\\
		& \leq & \tau \sum_{l = [k-\tau + 1]_{+}}^k\norms{Gy^{l} - Gy^{l-1} }^2.
	\end{array}
\end{equation*}
Using this inequality, one can easily verify that $\widetilde{G}^k = Gy^{k-\tau_k}$ satisfies Definition~\ref{de:error_bound_cond} with $\kappa = 1$, $\Theta = \tau$, $\hat{\Theta} = 0$, and $\delta_k = 0$.
Under this choice of $\widetilde{G}^k$, the stepsize $\eta$ in Theorem~\ref{th:iFKM4CE_convergence_expectation} reduces to $0 < \eta \leq \frac{3\beta}{7 \Lambda \tau + 3(1+\tau)} = \frac{3\beta}{3 + (7\Lambda + 3)\tau}$ as stated in Corollary~\ref{co:iFKM4CE_delayed_convergence_expectation}. 
Moreover, the statements in parts (i) and (ii) are direct consequences of Theorems~\ref{th:iFKM4CE_convergence_expectation} and \ref{th:iFKM4CE_convergence_almost_sure}, respectively.
Finally, for a given tolerance $\epsilon > 0$, the number of iterations $k$ to achieve $\Expn{\norms{Gy^k}^2} \leq \epsilon^2$ in Theorem~\ref{th:iFKM4CE_convergence_expectation} reduces to $\BigO{\frac{\tau}{\epsilon} + \frac{\tau}{\beta\epsilon}}$ as stated in Corollary~\ref{co:iFKM4CE_delayed_convergence_expectation}.
\Eproof
\end{proof}
%%% End of Proof.

%%%%%%%%%%%%%%%%%%%%%%%%%%%%%%%%%%%%%%%%%%%%%%%%%%%%%%%%%%%%%%%%%%%%%%%%%%%%
\beforesubsec
\subsection{\textbf{\textit{Asynchronous Stochastic Accelerated Fixed-Point Method}}}\label{subsec:FKM_delay_stochastic}
\aftersubsec
In this subsection, we consider the setting when $G$ in \eqref{eq:CE} is inaccessible but can be queried only through an unbiased stochastic oracle $\mbf{G}(\cdot, \xi)$ with bounded variance defined on a given probability space.
More specifically, we impose the following additional assumption on $G$:
\begin{assumption}\label{as:unbiased_oracle}
	$G$ in \eqref{eq:CE} is equipped with an unbiased stochastic oracle $\mbf{G}(\cdot, \xi)$ such that
	\begin{equation}\label{eq:unbiased_oracle}
		Gx = \Expsn{\xi}{\mbf{G}(x, \xi)}, \quad \forall x \in \dom{G}.
	\end{equation}
	Moreover, there exists $\sigma \geq 0$ such that $\Expsn{\xi}{\norms{\mbf{G}(x, \xi) - Gx}^2} \leq \sigma^2$ for all $x \in \dom{G}$.
\end{assumption}
This assumption is standard in stochastic approximation methods, see, e.g., \cite{Nemirovski2009}, and it has been extended to different forms.
Here, we simply use a simple version as in \cite{Nemirovski2009}.

\beforesubsubsec
\subsubsection{The computing mechanism}\label{subsubsec:ASFKM_mechanism}
\aftersubsubsec
The scheme studied in this subsection is the same as \eqref{eq:iFKM4CE}, and the mechanism is very similar to the mechanism studied in the previous subsection.
However, unlike Subsection~\ref{subsec:iFKM4CE_delay_universal}, due to the inaccessibility to the full operator $G$, workers only return a stochastic estimator of $Gy^k$ for any $k \geq 0$ rather than the true value.
This leads to the choice $\widetilde{G}^k$ used in this variant of \eqref{eq:iFKM4CE} as an \textit{``delayed stochastic estimator''}, i.e., a stochastic estimator of an obsolete operator value $Gy^{k - \tau_k}$ for some $\tau_k \in \sets{0, 1, \cdots, \tau}$. 

\vspace{0.5ex}
\noindent\textbf{$\mathrm{(a)}$~\textit{Delayed stochastic estimator.}}
We use the same update scheme as in~\eqref{eq:iFKM4CE}, and the execution flow mirrors the previous subsection, Subsection~\ref{subsec:iFKM4CE_delay_universal}. 
The key difference is that workers cannot evaluate the full operator $G$ exactly; instead, upon receiving an iterate $y^{k-\tau_k}$, they return a \emph{stochastic estimation} of $Gy^{k-\tau_k}$. 
Consequently, at the $k$-th iteration, \eqref{eq:iFKM4CE} employs a \emph{delayed stochastic estimator} $\widetilde{G}^k$ of $Gy^{k-\tau_k}$.
Compared to Subsection~\ref{subsec:iFKM4CE_delay_universal}, we now have two sources of approximation errors that need to be addressed:
\begin{compactitem}[$\bullet$]
	\item \textit{Stochastic approximation error:} The error between the stochastic estimator $\widetilde{G}^k$ and $Gy^{k - \tau_k}$. 
	\item \textit{Asynchronous error:} The error comes from the obsolescence between $Gy^{k - \tau_k}$ and $Gy^k$.
\end{compactitem}
\vspace{0.5ex}
\noindent\textbf{$\mathrm{(b)}$~\textit{Approximation condition for $\widetilde{G}^k$.}}
Given an unbiased stochastic oracle $\mbf{G}(\cdot, \xi)$ of $G$, we assume that $\widetilde{G}^k$ is constructed by querying the oracles $\mbf{G}(y^{k-\tau_k}, \xi)$ of $Gy^{k - \tau_k}$ adapted to the filtration $\sets{\Fc_k}$ and there exists a nonnegative sequence $\sets{\hat{\delta}_k}$ such that
\begin{equation}\label{eq:error_bound}
	\arraycolsep=0.2em
	\begin{array}{lcl}
	\Expsn{\xi}{\norms{\widetilde{G}^k - Gy^{k - \tau_k}}^2} \leq \frac{\hat{\delta}_k}{t_k (t_k - s)},
	\end{array}
\end{equation}
where $t_k$ and $s$ are from \eqref{eq:iFKM4CE}.
One way to choose $\hat{\delta}_k$ is $\hat{\delta}_k := \frac{\sigma^2}{(k+r)^{1 +\omega}}$ for some $r > 0$, $\omega \geq 0$, and $\sigma$ given in Assumption~\ref{as:unbiased_oracle}.

\vspace{0.5ex}
\noindent\textbf{\textit{Example: Mini-batch estimator.}}
We can construct $\widetilde{G}^k$ using adaptive mini-batches with increasing batch size.
For example, given an i.i.d. mini-batch $\Bc_k$ of size $b_k$, we construct
\begin{equation}\label{eq:increasingmb_est}
	\arraycolsep=0.2em
	\begin{array}{lcl}
		\widetilde{G}^k := \frac{1}{b_k} \sum_{\xi_k \in \Bc_k} \mbf{G}(y^{k - \tau_k}, \xi_k).
	\end{array}
\end{equation}
By Jensen's inequality, it is well-known that $\Expsn{\Bc_k}{\norms{\widetilde{G}^k - Gy^{k-\tau_k}}^2} \leq \frac{\sigma^2}{b_k}$.
If we choose  $\hat{\delta}_k := \frac{\sigma^2}{(k+r)^{1 +\omega}}$, then to guarantee \eqref{eq:error_bound}, we need to impose $\frac{\sigma^2}{b_k} \leq \frac{\sigma^2}{(k+r)^{1+\omega} t_k (t_k - s)}$, leading to an increasing batch-size update $b_k \geq (k+r)^{1+\omega} t_k (t_k - s)$.
%Looking at Figure~\ref{fig:consistent_architecture}, t
This mechanism is almost the same as the universal delayed oracle mechanism, except that the workers return a mini-batch estimator $\widetilde{G}^k := \frac{1}{b_k} \sum_{\xi_k \in \Bc_k} \mbf{G}(y^{k - \tau_k}, \xi_k)$ to the server, rather than the true value $Gy^{k - \tau_k}$, which is not accessible in this setting.

The condition \eqref{eq:error_bound} can hold for a wide range of  stochastic variance-reduced estimators, including unbiased and biased instances such as common estimators: SVRG \cite{SVRG}, SAGA \cite{Defazio2014}, SARAH \cite{nguyen2017sarah}, and Hybrid-SGD \cite{Tran-Dinh2019a}.
However, we omit their construction for brevity.
We only consider the mini-batch estimator \eqref{eq:increasingmb_est} as a simple instance of our class.

\beforesubsubsec
\subsubsection{Convergence of asynchronous stochastic AFP method}\label{subsubsec:ASFKM_convergence}
\aftersubsubsec
We specify Theorems~\ref{th:iFKM4CE_convergence_expectation} and \ref{th:iFKM4CE_convergence_almost_sure} to this asynchronous stochastic AFP variant and obtain the following convergence results.

%%% Corollary 4.2
\begin{corollary}\label{co:iFKM4CE_delay_stochastic_convergence_expectation}
	Suppose that Assumption~\ref{as:A1} holds with $\bar{\beta} = 0$ and Assumption~\ref{as:unbiased_oracle} also holds for \eqref{eq:CE}.
	Let $\sets{(x^k, y^k, z^k)}$ be generated by \eqref{eq:iFKM4CE} using a delayed stochastic estimator $\widetilde{G}^k$ satisfying \eqref{eq:error_bound} with uniformly bounded delays, i.e., there exists $\tau \geq 1$ such that $0 \leq \tau_k \leq \tau$ for all $k \geq 0$. 
	Suppose that we choose  $s$ and $\eta$, and update $t_k$, $\gamma_k$, and $\eta_k$ as follows:
	\begin{equation*} 
		\arraycolsep=0.2em
		\hspace{-2ex}
		\begin{array}{lcl}
			s \geq 1+3\gamma, \quad  0 < \eta \leq \frac{3\beta}{3 + (7\Lambda + 3)\tau}, \quad \gamma_k := \gamma \in [0,1], \quad  t_k := k + 3s + \tau,  \quad \eta_k := \frac{\eta t_k}{2(t_k - s)}.
		\end{array}
		\hspace{-2ex}
	\end{equation*}
	Suppose further that $\Sc_{\infty} := \sum_{k=0}^\infty \delta_k < +\infty$.
	Then, the following statements hold.
	\begin{compactitem}
	\item[\textbf{$\mathrm{(i)}$~Convergence in expectation.}] 
	For $\Rc_0^2 := \frac{\eta (3s+\tau-1)^2}{2} \norms{Gy^0}^2 + \frac{2s^3}{\eta \gamma}\norms{y^0 - x^{\star}}^2$, we have
	\begin{equation*} 
		\arraycolsep=0.2em
		\begin{array}{lcl}
			\displaystyle \Expn{\norms{Gy^k}^2} &\leq& \displaystyle \frac{4(\Rc_0^2 + \Lambda \eta \tau^{-1}\Sc_{\infty})}{\eta (k + 5s + \tau - 1)^2} \quad \text{and} \quad
			\lim_{k \to \infty} (k + \tau)^2 \Expn{\norms{Gy^k}^2} = 0.
		\end{array}
	\end{equation*}
	
	\item[\textbf{$\mathrm{(ii)}$~Almost sure convergence.}] 
	The following results hold almost surely:
	\begin{equation*} 
			\sum_{k=0}^\infty (k + \tau) \norms{Gy^k}^2 < +\infty \qquad \text{and} \qquad \lim_{k \to \infty} (k + \tau)^2 \norms{Gy^k}^2 = 0.
	\end{equation*}
	Moreover, the iterate sequences $\sets{x^k}$, $\sets{y^k}$, and $\sets{z^k}$ converge almost surely to a $\zer{G}$-valued random variable $x^{\star}$, a solution to \eqref{eq:CE}.
	
	\item[\textbf{$\mathrm{(iii)}$~Iteration-complexity.}]
	For a given $\epsilon > 0$,  we requires at most $k := \BigO{\frac{\sigma}{\epsilon\sqrt{\tau}} + \frac{\tau}{\epsilon} + \frac{\tau}{\beta\epsilon}}$ iterations to achieve $\Expn{\norms{Gy^k}^2} \leq \epsilon^2$.
	\end{compactitem}
\end{corollary}

%%% Proof of Lemma 3.1.
\begin{proof}
	We decompose the error term $e^k := \widetilde{G}^k - Gy^k$ from \eqref{eq:oracle_error} into $e^k = \tilde{e}^k + \hat{e}^k$, where $\tilde{e}^k := \widetilde{G}^k - Gy^{k - \tau_k}$ is the 	stochastic approximation error and $\hat{e}^k := Gy^{k - \tau_k} - Gy^k$ is the asynchronous error.
	Then, using the convexity of $\norms{\cdot}^2$ and $0 \leq \tau_k \leq \tau$, we can derive that
	\begin{equation*}
		\arraycolsep=0.2em
		\begin{array}{lcl}
			\norms{\hat{e}^k}^2 & = & \norms{Gy^{k-\tau_k} - Gy^k}^2  
			\leq  \tau \sum_{l = [k-\tau + 1]_{+}}^k\norms{Gy^{l} - Gy^{l-1} }^2.
		\end{array}
	\end{equation*}
	Since $e^k = \tilde{e}^k + \hat{e}^k$, applying Young's inequality, \eqref{eq:error_bound}, and the last relation, we have
	\begin{equation*}
		\arraycolsep=0.2em
		\begin{array}{lcl}
			\Expsn{k}{\norms{e^k}^2} & \leq & 2\Expsn{k}{\norms{\tilde{e}^k}^2} + 2\Expsn{k}{\norms{\hat{e}^k}^2} \leq  \frac{2\hat{\delta}_k}{t_k(t_k - s)} + 2\tau \sum_{l = [k-\tau + 1]_{+}}^k \norms{Gy^{l} - Gy^{l-1} }^2.
		\end{array}
	\end{equation*}
	This inequality shows that $\widetilde{G}^k$ constructed by the delayed stochastic estimator \eqref{eq:error_bound} satisfies Definition~\ref{de:error_bound_cond} with $\kappa = 1$, $\Theta = \tau$, $\hat{\Theta} = 0$, and $\delta_k = 2\hat{\delta}_k$.
	
	Using this estimator, the stepsize $\eta$ in Theorem~\ref{th:iFKM4CE_convergence_expectation} reduces to $0 < \eta \leq \frac{3\beta}{7 \Lambda \tau + 3(1+\tau)} = \frac{3\beta}{3 + (7\Lambda + 3)\tau}$ as stated in Corollary~\ref{co:iFKM4CE_delay_stochastic_convergence_expectation}. 
	Moreover, the statements in parts (i) and (ii) are direct consequences of Theorems~\ref{th:iFKM4CE_convergence_expectation} and \ref{th:iFKM4CE_convergence_almost_sure}, respectively.
	Finally, for a given tolerance $\epsilon > 0$, if we choose $\hat{\delta}_k := \frac{\sigma^2}{(k+r)^{1+\omega}}$, leading to $\delta_k := \frac{2\sigma^2}{(k+r)^{1+\omega}}$, the number of iterations to achieve $\Expn{\norms{Gy^k}^2} \leq \epsilon^2$ in Theorem~\ref{th:iFKM4CE_convergence_expectation} becomes $ k = \BigO{\frac{\sigma}{\epsilon\sqrt{\tau}} + \frac{\tau}{\epsilon} + \frac{\tau}{\beta\epsilon}}$ as stated in Corollary~\ref{co:iFKM4CE_delay_stochastic_convergence_expectation}.
\Eproof
\end{proof}
%%% End of Proof.

%%%%%%%%%%%%%%%%%%%%%%%%%%%%%%%%%%%%%%%%%%%%%%%%%%%%%%%%%%%%%%%%%%
%%%% 4.3. Finite-sum setting of \eqref{eq:CE}: Inconsistent Asynchronous Fast KM method.
\beforesubsec
\subsection{\textbf{\textit{Finite-sum setting of \eqref{eq:CE}: Inconsistent Asynchronous AFP method}}}\label{subsec:iFKM4CE_iasyn}
\aftersubsec
Finally, we consider the finite-sum structure $Gx := \frac{1}{n} \sum_{i=1}^n G_ix$ of $G$ in \eqref{eq:CE} with $\Dc (x,y) := \frac{1}{n} \sum_{i=1}^n \norms{G_ix - G_iy}^2$. 
More specifically, we assume that the following assumption holds.

%%% Assumption A.2.
\begin{assumption}\label{as:finite_sum}
The operator $G$ in \eqref{eq:CE} satisfies the following conditions.
\begin{compactitem}
	\item[$\mathrm{(i)}$] 
	 $G$ has a finite-sum structure, i.e., $Gx = \frac{1}{n}\sum_{i=1}^n G_i x$ for $n \geq 1$.
	
	\item[$\mathrm{(ii)}$] 
	There exist $\beta, \bar{\beta} \in [0, +\infty)$ such that
	\begin{equation}\label{eq:G_cocoercivity2}
		\arraycolsep=0.2em
		\begin{array}{lcl}
			\iprods{Gx - Gy, x - y} & \geq & \beta \norms{Gx - Gy}^2 + \frac{\bar{\beta}}{n} \sum_{i=1}^n \norms{G_ix - G_iy}^2, \quad \forall x, y \in \dom{G}.
		\end{array}
	\end{equation}
\end{compactitem}
\end{assumption}
For (i), the finite-sum structure is ubiquitous in statistical learning and machine learning applications.
It is often obtained by sample average approximation.
For (ii), we can see that $G$ satisfying \eqref{eq:G_cocoercivity2} is $\beta$-co-coercive.
However, the reverse statement is not true in general.
Thus, this assumption is stronger than the $\beta$-co-coercivity of $G$. 
If $\beta = 0$ and $\bar{\beta} > 0$, then $G$ is said to be $\bar{\beta}$-average co-coercive.
This assumption has been previously used in variance-reduced fixed-point-based algorithms, including \cite{cai2023variance,tran2024accelerated,TranDinh2025a}, and is common in convex optimization.

\beforesubsubsec
\subsubsection{The inconsistently asynchronous oracle}
\aftersubsubsec
Before proceeding, let us recall the computing mechanism from Subsection~\ref{subsec:iFKM4CE_delay_universal}. 
When $G$ has a finite-sum structure, $\widetilde{G}^k$ used in Subsection~\ref{subsec:iFKM4CE_delay_universal} was \emph{consistently asynchronous} as
\begin{equation*}
\arraycolsep=0.2em
\begin{array}{lcl}
	\widetilde{G}^k &:=& Gy^{k - \tau_k} = \frac{1}{n}\sum_{i=1}^n G_i y^{k-\tau_k},
\end{array}
\end{equation*}
which requires all components $G_iy^{k-\tau_k}$ at the same delayed iterate $y^{k-\tau_k}$ and therefore coincides with an actual historical operator value $Gy^l$ for some $k-\tau+ 1 \leq l \leq k$.

In contrast to the above consistency, this subsection allows \emph{inconsistently asynchronous updates}, i.e., each component $G_i$ may be evaluated at a (possibly) different stale iterate within a bounded delay window. 
Specifically, each component $G_iy^{k-\tau_{i,k}}$ was evaluated at each delayed iterate $y^{k-\tau_{i,k}}$ such that $\tau_{i,k} \in \sets{0, 1, \cdots, \tau}$.
Then, we set
\begin{equation}\label{eq:inconsistent_update}
\arraycolsep=0.2em
\begin{array}{lcl}
\widetilde{G}^k & := & \frac{1}{n}\sum_{i=1}^n G_iy^{k-\tau_{i,k}}.
\end{array}
\end{equation}
In this case,  $\widetilde{G}^k$ is possibly not equal to $Gy^l$ for any single $l \in \sets{k-\tau + 1, \cdots, k} $, and hence, ``inconsistent''. 
Nevertheless, it is \emph{atomically consistent}, i.e.,  every term $G_iy^{k-\tau_{i,k}}$ is a historical value drawn from $\sets{G_iy^{k-\tau + 1}, \cdots, G_iy^k}$ for $i \in [n]$.
Hitherto, we have not specified any computing mechanism yielding the oracle \eqref{eq:inconsistent_update} yet, but we will give a detailed description of two representative mechanisms later: \textit{incremental and shuffling strategies}.

%%%%% 4.3.2. Convergence of the inconsistently asynchronous AFP method
\beforesubsubsec
\subsubsection{Convergence of the inconsistently asynchronous AFP method}\label{subsubsec:iAFKM_for_finite_sum_convergence}
\aftersubsubsec
For given $\beta$ and $\bar{\beta}$ in Assumption~\ref{as:finite_sum}, $\tau \geq 1$, and $\Lambda := 1 + s - \gamma > 0$, let us define
\begin{equation}\label{eq:iFKM4CE_iasyn_eta}
\arraycolsep=0.2em
\bar{\eta} := \left\{ \begin{array}{ll}
	\frac{3\bar{\beta}}{7 \Lambda \tau} \quad &\text{if } 0 < \bar{\beta} \leq \frac{7 \Lambda \tau \beta}{3(1+\tau)}, \vspace{1ex}\\
	\frac{\beta}{1 + \tau} &\text{if } \bar{\beta} \geq \frac{7 \Lambda \tau \beta}{3(1+\tau)}.
\end{array}\right.
\end{equation}
Then, we have the following convergence results for this variant of \eqref{eq:iFKM4CE}.

%%% Corollary 4.3
\begin{corollary}\label{co:iFKM4CE_iasyn}
	Suppose that Assumptions~\ref{as:A1}$($a$)$ and \ref{as:finite_sum} holds for \eqref{eq:CE}.
	Let $\sets{(x^k, y^k, z^k)}$ be generated by \eqref{eq:iFKM4CE} using the inconsistent estimator $\widetilde{G}^k$ in \eqref{eq:inconsistent_update} such that the delays $0 \leq \tau_{i,k} \leq \tau$ for $\tau \geq 1$ and $i \in [n]$, where we choose $s$ and $\eta$, and update $t_k$, $\gamma_k$, and $\eta_k$ as 
	\begin{equation*}%\label{eq:iFKM4CE_iasyn_params}
				\arraycolsep=0.2em
		\hspace{-2ex}
		\begin{array}{lcl}
			s \geq 1+3\gamma, \quad  0 \leq \eta \leq \bar{\eta}, \quad \gamma_k := \gamma \in (0,1], \quad  t_k := k + 3s + \tau,  \quad \eta_k := \frac{\eta t_k}{2(t_k - s)}.
		\end{array}
		\hspace{-2ex}
		\end{equation*}
		Then, the following statements hold.
		\begin{compactitem}
		\item[\textbf{$\mathrm{(i)}$~Convergence in expectation.}]
		For $\Rc_0^2 := \frac{\eta (3s+\tau-1)^2}{2} \norms{Gy^0}^2 + \frac{2s^3}{\eta \gamma}\norms{y^0 - x^{\star}}^2$, we have
		\begin{equation*} 
		\arraycolsep=0.2em
		\begin{array}{lcl}
			\displaystyle \Expn{\norms{Gy^k}^2} &\leq& \displaystyle \frac{4\Rc_0^2}{\eta (k + 3s + \tau - 1)^2} \quad \text{and} \quad
			\lim_{k \to \infty} (k + \tau)^2 \Expn{\norms{Gy^k}^2} = 0.
		\end{array}
		\end{equation*}
		\item[\textbf{$\mathrm{(ii)}$~Almost sure convergence.}]
		The following results hold almost surely:
		\begin{equation*}
		\arraycolsep=0.2em
		\begin{array}{lcl}
			\displaystyle \sum_{k=0}^\infty (k + \tau)  \norms{Gy^k}^2 < + \infty \qquad \text{and} \qquad \lim_{k \to \infty} (k + \tau)^2 \norms{Gy^k}^2 = 0.
		\end{array}
		\end{equation*}
		Moreover, the sequences $\sets{x^k}$, $\sets{y^k}$, and $\sets{z^k}$ converge almost surely to a $\zer{G}$-valued random variable $x^{\star}$, a solution to \eqref{eq:CE}.
		
	\item[\textbf{$\mathrm{(iii)}$~Iteration-complexity.}] 
		For a given tolerance $\epsilon > 0$, the total number of iterations $k$ to achieve $\Expn{\norms{Gy^k}^2} \leq \epsilon^2$ is at most
	\begin{equation*}
		\arraycolsep=0.2em
		k := \left\{\begin{array}{ll}
			\BigO{\frac{\tau}{\epsilon} + \frac{\tau}{\bar{\beta}\epsilon}} \quad &\text{if } 0 < \bar{\beta} \leq \frac{7 \Lambda \tau \beta}{3(1+\tau)}, \vspace{1ex}\\
			
			\BigO{\frac{\tau}{\epsilon} + \frac{\tau}{\beta\epsilon}} \quad &\text{if $\bar{\beta} \geq \frac{7 \Lambda \tau \beta}{3(1+\tau)}$}.
		\end{array}\right.
	\end{equation*}
	\end{compactitem}
\end{corollary}

%%% Proof of Lemma 4.2.
\begin{proof}
	The source of randomness in inconsistent updates comes from which node has made an update at each iteration.
	Let $i_k$ be the node that sends an update to the server at the $k$-th iteration.
	Then, we can view $i_k$ as a random variable on $\sets{1, \cdots, n}$.
	Let $\Fc_k := \sigma(i_1, i_2, \cdots, i_{k-1})$ be the $\sigma$-field generated by all the randomness before the $k$-th iteration.
	In addition, let $\Expsn{k}{\cdot} = \Expn{\cdot \mid \Fc_k}$ be the conditional expectation conditioned on $\Fc_k$.
	
	From \eqref{eq:inconsistent_update}, by Jensen's inequality, the convexity of $\norms{\cdot}^2$, and $0 \leq \tau_{i,k} \leq \tau$, we get
	\begin{equation*} 
	\arraycolsep=0.2em
	\begin{array}{lcl}
	\Expsn{k}{\norms{e^k}^2} &= & \Expsn{k}{\norms{\widetilde{G}^k - Gy^k}^2}  =  \Expsn{k}{\norms{\frac{1}{n} \sum_{i=1}^n (G_i y^{k - \tau_{i,k}} - G_i y^k)}^2} \vspace{1ex}\\
	&\leq& \frac{1}{n}\sum_{i=1}^n \Expsn{k}{\norms{G_i y^{k - \tau_{i,k}} - G_i y^k}^2} \vspace{1ex}\\
	&=& \frac{1}{n}\sum_{i=1}^n \Expsn{k}{\norms{\sum_{l=k-\tau_{i,k} + 1}^k (G_i y^l - G_i y^{l-1})}^2} \vspace{1ex}\\
	&=& \frac{1}{n}\sum_{i=1}^n \Expsn{k}{\tau_{i,k}^2 \norms{\sum_{l=k-\tau_{i,k} + 1}^k \frac{1}{\tau_{i,k}} (G_i y^l - G_i y^{l-1})}^2} \vspace{1ex}\\
	&\leq& \frac{1}{n}\sum_{i=1}^n \Expsn{k}{\tau_{i,k} \sum_{l=k-\tau_{i,k} + 1}^k\norms{ G_i y^l - G_i y^{l-1}}^2} \vspace{1ex}\\
	&\leq& \frac{\tau}{n}\sum_{l=k-\tau + 1}^k \sum_{i=1}^n \Expsn{k}{\norms{ G_i y^l - G_i y^{l-1}}^2}.
	\end{array}
	\end{equation*}
	This expression shows that $\widetilde{G}^k$ computed by the inconsistent finite-sum update \eqref{eq:inconsistent_update} satisfies Definition~\ref{de:error_bound_cond} with $\kappa = 1$, $\Theta = 0$, $\hat{\Theta} = \tau$, and $\delta_k := 0$.

	Using this estimator $\widetilde{G}^k$, the stepsize $\eta$ in Theorem~\ref{th:iFKM4CE_convergence_expectation} reduces to $0 < \eta \leq \frac{3\bar{\beta}}{7 \Lambda \tau}$ if $0 < \bar{\beta} \leq \frac{7 \Lambda \tau \beta}{3(1+\tau)}$, and $0 \leq \eta \leq \frac{\beta}{1+\tau}$ if $\bar{\beta} \geq \frac{7 \Lambda \tau \beta}{3(1+\tau)}$ as stated in Corollary~\ref{co:iFKM4CE_iasyn}.  
	Moreover, the statements in Parts (i) and (ii) are direct consequences of Theorems~\ref{th:iFKM4CE_convergence_expectation} and \ref{th:iFKM4CE_convergence_almost_sure}, respectively.
	Finally, for a given tolerance $\epsilon > 0$, the total number of iterations to achieve $\Expn{\norms{Gy^k}^2} \leq \epsilon^2$ in Theorem~\ref{th:iFKM4CE_convergence_expectation} reduces to $\BigO{\frac{\tau}{\epsilon} + \frac{\tau}{\bar{\beta}\epsilon}}$ if $0 < \bar{\beta} \leq \frac{7 \Lambda \tau \beta}{3(1+\tau)}$, and $\BigO{\frac{\tau}{\epsilon} + \frac{\tau}{\beta\epsilon}}$ if $\bar{\beta} \geq \frac{7\Lambda \tau \beta}{3(1+\tau)}$ as stated in Corollary~\ref{co:iFKM4CE_iasyn}.
\Eproof
\end{proof}
%%% End of Proof.

%%% 4.3. Two concrete mechanisms: Incremental and Shuffling Aggregated. 
\beforesubsubsec
\subsubsection{A representative mechanism in decentralized computing systems}\label{subsubsec:two_concrete_exams}
\aftersubsubsec
While the inconsistently asynchronous updates \eqref{eq:inconsistent_update} may cover a wide range of computing mechanisms, we specify here a representative one that is usually studied in literature.

%%%%% Read-write mechanism description
\vspace{0.5ex}
\noindent\textbf{$\mathrm{(a)}$~\textit{General mechanism}.}
We again consider a server-worker system with one central server and $n$ single workers similar to Subsection~\ref{subsec:iFKM4CE_delay_universal}, but now each worker has their own data and cannot be shared with other workers, and the server has $n$ separate memory slots to maintain the most recent values of each component $G_i$ in the \texttt{G}-memory.
The $i$-th worker is now responsible for only computing $G_i$, $i = 1, \cdots, n$ (if it is not the case, then we can divide $G_i$ into groups such that each group is handled by one worker).
The server stores the most recent values of the components $G_i y^{k - \tau_{i,k}}$ and is responsible to updating the iterates $x^k$, $y^k$, and $z^k$ in the \texttt{V}-memory.
See Figure~\ref{fig:consistent_architecture} (right) for an illustration of the system architecture.

\vspace{0.5ex}
\noindent$\diamond$~\textbf{Initialization ($k = 0$) --- Synchronous step.}
At the initial iteration $k=0$, the server initializes an iterate $y^0$ in the \texttt{V}-memory and broadcasts this $y^0$ to all nodes.
\begin{compactitem}[$\bullet$]
	\item \textit{Worker $i$} ($i\in [n]$): It first evaluates $G_i y^0$, then sends $G_i y^0$ to the central server.
	\item \textit{Server}: Waits until gathering all values $G_i y^0$ from the $n$ nodes, then computes $Gy^0 := \frac{1}{n} \sum_{i=1}^n G_iy^0$ and sets $\widetilde{G}^0 := Gy^0$.
	Next, the server uses $\widetilde{G}^0$ to update the iterates $x^1$, $z^1$, $y^1$, and then broadcasts $y^1$ to all nodes as the initial iterate for the asynchronous phase. 
\end{compactitem}
\vspace{0.5ex}	
\noindent$\diamond$~\textbf{Main loop ($k \geq 1$) --- Asynchronous phase.} 
Let $\tau_{i,k}$ be the delay of the iterate at which $G_i$ is evaluated at the $k$-th iteration.
Let $I_k \subseteq [n]$ be the set of workers that make an update at the $k$-th iteration.
\begin{compactitem}[$\bullet$]
	\item \textit{Worker $i \in I_k$}: 
	Whenever a worker $i \in I_k$ finishes computing $G_{i} y^{k - \tau_{i, k}}$, it returns this operator value to the central server.
	\item \textit{Central server}: 
	Upon receiving $G_{i} y^{k - \tau_{{i},k}}$ from all workers $i \in I_k$, the central server updates the estimator $\widetilde{G}^k$ as:
	\begin{equation*}%\label{eq:Gtilde3}
		\begin{array}{lcl}
			\widetilde{G}^k := \widetilde{G}^{k-1} +  {\displaystyle\frac{1}{n} \sum_{i \in I_k}} \Big(\underbrace{G_{i} y^{k-\tau_{{i}, k}}}_{\substack{\text{add the newly}\\\text{updated values}}} - ~ \underbrace{G_{i} y^{(k-1)-\tau_{{i}, k-1}}}_{\substack{\text{discard values from}\\ \text{the last update}}} \Big).
		\end{array}
	\end{equation*}
	Simultaneously, we also replace the current value in the $G_i$ memory slot by $G_{i} y^{k - \tau_{i, k}}$ for all $i \in I_k$.
	Then, $\widetilde{G}^k$ is used in \eqref{eq:iFKM4CE} to update the iterates $x^{k+1}$, $z^{k+1}$, and $y^{k+1}$.
	After that, the server broadcasts this $y^{k+1}$ to all idle workers in the system.
\end{compactitem}
Note that we still assume the atomicity for all read-write processes in this mechanism, as well as the boundedness of the maximum delay, i.e., $0 \leq \tau_{i,k} \leq \tau$ for all $k\geq 0$ and $i \in [n]$.

\vspace{0.5ex}
\noindent\textbf{$\mathrm{(b)}$~\textit{Two concrete examples:  Incremental and Shuffling Aggregated.}}
Now, we specify two concrete mechanisms of the general decentralized computing system often studied in the literature: \emph{Incremental Aggregated} and \emph{Shuffling Aggregated}.
Both mechanisms maintain a component buffer $\hat{G}^k = [\hat{G}_1^k, \cdots, \hat{G}_n^k]$ and, at each iteration, update only \emph{one} component before forming the averaged estimator $\widetilde{G}^k := \frac{1}{n}\sum_{i=1}^n \hat G_i^k$ of $G$ (i.e., $|I_k| = 1$). 
These two mechanisms only differ in the \emph{order} in which indices in $\sets{1,\cdots, n}$ are visited.
\begin{compactitem}[$\bullet$]
	\item \textbf{Incremental (or cyclic) strategy.} Use a fixed order $\pi_0=[1,2, \cdots,n]$; at iteration $k$ pick $i=\pi_0((k \bmod n)+1)$, set $\hat G^k_i \leftarrow G_i y^k$, and leave the remaining $\hat G_j^k$ unchanged. 
	\item \textbf{Shuffling strategy.} Partition iterations into epochs of length $n$; at the start of each epoch ($k \bmod n=0$), draw a new permutation $\pi$ of $\{1,\cdots,n\}$ (deterministic or random), and during that epoch update components in the order prescribed by $\pi$. 
\end{compactitem}
Both schemes compute only a single $G_i$ per iteration while leveraging historical entries in $\hat{G}^k$, thus saving a huge amount of computation, especially when $n$ is large.
When applying these two mechanisms to \eqref{eq:iFKM4CE}, we consequently obtain Algorithm \ref{alg:shufflingFKM}.

\begin{algorithm}[hpt!]
	\caption{(Incremental/Shuffling Aggregated Accelerated Fixed-Point Method)}\label{alg:shufflingFKM}
	%\normalsize
	\begin{algorithmic}[1]
		\itemsep=-0.0em
		\State{\bfseries Input:} 
		An initial point $y^0 \in \Hspace$, and choose \texttt{method} $\in$ \sets{\texttt{incremental}, \texttt{shuffling}}.
		\State{\bfseries Initialization:} 
		Compute $\hat{G}^0 := [G_1 y^0, G_2 y^0, \cdots, G_n y^0]$.     
		\Statex If \texttt{method} = \texttt{incremental}, set $\tau := n$. If \texttt{method} = \texttt{shuffling}, set $\tau := 2n$.
		\Statex Set $z^0 := y^0$ and set $\pi_0 := [1, 2, \cdots, n]$.
		\State\hspace{0ex}{\bfseries For $k := 0,\cdots, k_{\max}$ do} 
		\State\hspace{2.5ex}{\bfseries If \texttt{method} $=$ \texttt{incremental} then} 
		\State\hspace{5ex}Set $\pi := \pi_0$;  
		\State\hspace{2.5ex}{\bfseries Else If \texttt{method} $=$ \texttt{shuffling} then}  
		\State\hspace{5ex}{\bfseries If $k \mod n = 0$ then}  
		\State\hspace{7.5ex}Sample a permutation $\pi$ of $\sets{1, 2, \cdots, n}$ (either deterministic or random)
		\State\hspace{5ex}{\bfseries End If} 
		\State\hspace{2.5ex}{\bfseries End If} 
		\State\hspace{2.5ex}Choose an active index $i := \pi ((k \mod n)+1)$   
		\State\hspace{2.5ex}Update the memory $\hat{G}^k_i \leftarrow G_i y^k$ and compute $\widetilde{G}^k = \frac{1}{n}\sum_{j=1}^n \hat{G}^k_j$ 
		\State\hspace{2.5ex}Update $t_k := k + 3s + \tau$, $\gamma_k := \gamma$, and $\eta_k := \frac{\eta t_k}{2(t_k - s)}$ 
		\State\hspace{2.5ex}Update 
		\begin{equation*}
		\arraycolsep=0.2em
		\left\{\begin{array}{lcl}
			x^{k+1} & :=  & y^k - \eta_k \widetilde{G}^{k} \vspace{1ex}\\
		  	z^{k+1} & :=  & z^k + \frac{\gamma_k}{s}({x}^{k+1} - {y}^{k}) \vspace{1ex}\\
 			y^{k+1} & :=  & \frac{s}{t_k}{z}^{k+1} + \frac{t_k - s}{t_k}{x}^{k+1}.
		\end{array}\right.
		\end{equation*}
		\State\hspace{0ex}{\bfseries End For}
	\end{algorithmic}
\end{algorithm}

%Figure~\ref{fig:memory_shuffling} illustrates how the memory slots for $G_i$'s are updated while Algorithm~\ref{alg:shufflingFKM} is proceeding.
Note that, the update of $x^{k+1}$ in both cases of Algorithm~\ref{alg:shufflingFKM} can be cast into the form 
\begin{equation*}
	\arraycolsep=0.2em
	\begin{array}{lcl}
		x^{k+1} & := & y^k - \eta_k \widetilde{G}^k,
	\end{array}
\end{equation*}
where the estimator $\widetilde{G}^k$ is constructed as follows.
\begin{compactitem}[$\diamond$]
	\item \textbf{Incremental aggregated strategy}: We form
	\begin{equation}\label{eq:incremental_update}
		\arraycolsep=0.2em
		\begin{array}{lcl}
			\widetilde{G}^k &:=& \frac{1}{n}\sum_{i=1}^n G_i y^{k - i},
		\end{array}
	\end{equation}	
	\item \textbf{Shuffling aggregated strategy}: We construct
	\begin{equation}\label{eq:shuffling_update}
		\arraycolsep=0.2em
		\begin{array}{lcl}
			\widetilde{G}^k &:=& \frac{1}{n}\sum_{i=1}^n G_{\pi(i)}y^{k - \pi(i)}.
		\end{array}
	\end{equation}
\end{compactitem}
Clearly, both \eqref{eq:incremental_update} and \eqref{eq:shuffling_update} show that the estimators $\widetilde{G}^k$ constructed by these two mechanisms are indeed special cases of the inconsistent updates \eqref{eq:inconsistent_update}.
Moreover, we also observe that in \eqref{eq:incremental_update}, the maximum delay $\tau$ is always $n$, while in  \eqref{eq:shuffling_update}, the maximum delay is at most $2n$.
Hence, we can apply Corollary~\ref{co:iFKM4CE_iasyn} to these concrete mechanisms with either $\tau := n$ or $\tau := 2n$ to obtain convergence results.
We omit stating this result here to avoid repetition. 

%%%%%%%%%%%%%%%%%%%%%%%%%%%%%%%%%%%%%%%%%%%%%%%%%%%%%%%
%%% 5. Numerical Experiments
%%%%%%%%%%%%%%%%%%%%%%%%%%%%%%%%%%%%%%%%%%%%%%%%%%%%%%%
\beforesec
\section{Numerical Experiments}\label{sec:num_experiments}
\aftersec
In this section, we present two numerical examples to illustrate the proposed algorithms and compare them with baseline methods.
All the algorithms are implemented in Python and run on UNC’s Longleaf HPC cluster with NVIDIA A100 (40 GB) and L40 (48 GB) GPUs.

\beforesubsec
\subsection{\textbf{Bilinear matrix game: Policeman vs. Burglar problem}}\label{subsec:experiment1}
\aftersubsec
The goal of this example is to test different variants of our framework \eqref{eq:iFKM4CE} specified in Section~\ref{sec:app_of_iFKM}, including both deterministic and stochastic variants.

%%%%
\vspace{0.5ex}
\noindent\textbf{$\mathrm{(a)}$~\textit{Mathematical model.}}
Consider an $m \times m$ urban grid comprising $m^2$ houses, where the $j$-th house possesses wealth $w_j$, $j = 1, \cdots, m^2$.
Each night, the Burglar selects a house $j$ to target, while the Policeman chooses a location proximate to house $k$ to station.
Upon initiation of the burglary, the Policeman is instantaneously informed of the incident's location, and the probability of apprehension is $\mbf{p}_c := \exp\sets{-\theta \mathrm{dist}(j,k)}$, where $\mathrm{dist}(j,k)$ denotes the distance between house $j$ and $k$.
However, neither the Burglar nor the Policeman is aware of the true wealth $w_j$; instead, they only have an estimate from a set of $n$ observations $\hat{w}_j^{(i)}$ for $i = 1, \cdots, n$.
The Burglar's objective is to maximize the expected reward $\mcal{R}_j := \frac{1}{n}\sum_{i=1}^n \hat{w}^{(i)}_{j}\bigl(1-\exp\sets{-\theta \mathrm{dist}(j,k)}\bigr)$,
whereas the Policeman seeks to minimize $\mcal{R}_j$.

The Policeman vs. Burglar problem outlined above admits a zero-sum game formulation.
This model was investigated in \cite{nemirovski2013mini}, and we introduce a minor modification to obtain a stochastic variant.
In line with \cite{nemirovski2013mini}, we specify the payoff matrix $\mbf{L}$ as follows:
\begin{equation*}
	\begin{array}{lcl}
		\mbf{L} := \frac{1}{n}\sum_{i=1}^n \mbf{L}^{(i)}, \quad \mbf{L}^{(i)}_{j,k} = {\hat{w}^{(i)}_{j}} \left(1 - \exp\{-\theta\, \mathrm{dist}(j, k)\}\right), \quad 1 \le j, k \le m^2.
	\end{array}
\end{equation*}
Then, we can formulate the above problem into the following classical bilinear matrix game: 
\begin{equation}\label{eq:bilinear_matrix_game}
	\min_{v \in \Delta_{p_1}}\max_{w \in \Delta_{p_2}}\Big\{ \Lc(v, w) := \iprods{\mbf{L}v, w} \Big\},
\end{equation}
where $v$ and $w$ represent mixed strategies of the Policeman and Burglar,  and $\Delta_{p_1}$ and $\Delta_{p_2}$ are the standard simplexes in $\R^{p_1}$ and $\R^{p_2}$, respectively with $p_1=p_2=m^2$.

Let us define $x := [v;w] \in \R^p$ with $p := p_1+p_2$, $G_i x := \mbf{G}_ix =  [\mbf{L}^{(i)\top}w; -\mbf{L}^{(i)}v]$ for $i \in [n]$, $Gx := \frac{1}{n}\sum_{i=1}^n G_ix$, and $Tx := [\partial \delta_{\Delta_{p_1}}(v); \partial \delta_{\Delta_{p_2}}(w)]$, where $\delta_{\Xc}(\cdot)$ is the indicator of $\Xc$.
Then, we can write the optimality condition of \eqref{eq:bilinear_matrix_game} as a generalized equation $0 \in Gx + Tx$.
%Since $\mbf{G} := \frac{1}{n}\sum_{i=1}^n \mbf{G}_i$ is a skew-symmetric matrix, $G$ is monotone and $L$-Lipschitz continuous.
It is well-known that this problem is equivalent to solving $\Gc_{\lambda} u = 0$, where $\Gc_{\lambda} u := G(J_{\lambda T}u) + \frac{1}{\lambda} (I - J_{\lambda T})u$ is the backward-forward splitting operator.
Then, $0 \in Gx^{\star} + Tx^{\star}$ if and only if $\Gc_{\lambda}u^\star = 0$ with $x^{\star} := J_{\lambda T} u^\star$.
The equation $\Gc_{\lambda}u^{\star} = 0$ is a special case of \eqref{eq:CE}.

\vspace{0.5ex}
\noindent\textbf{$\mathrm{(b)}$~\textit{Input data.}}
Following \cite{nemirovski2013mini}, we set $\mathrm{dist}(j,k) := |j-k|$ and fix $\theta := 0.8$.
We then draw the nominal wealth $w_j$ from a standard normal distribution and replace it by $|w_j|$ to enforce nonnegativity.
Noisy observations are generated as $w^{(i)}_j := |w_j + \epsilon^{(i)}_j|$, where  $\epsilon^{(i)}_j$ is a normal random variable with zero mean and variance $\sigma^2 = 0.05$.
Finally, we produce two sets of experiments, each comprising $5$ problem instances:
\begin{compactitem}[$\bullet$]
	\item \textbf{Experiment 1.} $m = 10$ and $n = 1000$, corresponding to $p_1=100$ and $p = 2p_1 = 200$.
	\item \textbf{Experiment 2.} $m = 15$ and $n = 2000$, corresponding to $p_1=225$ and $p = 2p_1 = 450$.
\end{compactitem}

\vspace{0.5ex}
\noindent\textbf{$\mathrm{(c)}$~\textit{Algorithms and parameters.}}
We implement five variants of our proposed method \eqref{eq:iFKM4CE} specified in Section~\ref{sec:app_of_iFKM} as follows.
\begin{compactitem}[$\diamond$]
	\item \texttt{AFP-D}: This is our \eqref{eq:iFKM4CE} using the delayed estimate \eqref{eq:delayed_oracle_of_G};
	\item \texttt{AFP-SD}: It is \eqref{eq:iFKM4CE} using  the stochastic mini-batch delayed estimate \eqref{eq:increasingmb_est};
	\item \texttt{AFP-IA}: It is \eqref{eq:iFKM4CE} using the incremental aggregated estimate \eqref{eq:incremental_update};
	\item \texttt{AFP-SA}: It is \eqref{eq:iFKM4CE} using the shuffling aggregated estimate \eqref{eq:shuffling_update}.
	\item \texttt{AFP-RA-$m$}: It is \eqref{eq:iFKM4CE} using a random aggregated estimator with $m$ components $\Gc_{\lambda,i}$ computed per iteration.
\end{compactitem}
For each variant, we set $\eta = \frac{1}{1+\tau}$ and $\eta = \frac{0.75}{1+\tau}$ in \textbf{Experiment 1} and \textbf{Experiment 2}, respectively, and choose $s=1.1$ and $\gamma = 1$ (other values still work).
The increasing mini-batch size for \texttt{AFP-SD} is chosen as $b_k = \max\sets{5, \min\sets{r(k+1)^3, n}}$, where $r>0$ is a small scale factor used to enforce that $b_k$ is increasing but not too fast as using $b_k = \BigOs{k^3}$.
The initial point is chosen as $x^0 := \big[ \frac{1}{p_1} \cdot \texttt{ones($p_1$)}; \frac{1}{p_2} \cdot \texttt{ones($p_2$)} \big]$ for all variants.

\vspace{0.5ex}
\noindent\textbf{$\mathrm{(d)}$~\textit{Numerical results.}}
First, we test the effect of the maximum delay $\tau$ on the performance of \eqref{eq:iFKM4CE}. 
To this end, we let $\tau$ vary from $\tau := 0$ (meaning no delays) to $\tau := 500$, and run the \texttt{AFP-D} variants with several values of $\tau$ in that range.
The results are presented in Figure~\ref{fig:matrix_games_tau_varies}, where \texttt{Number of full passes} is the number of times computing full $n$ components $\Gc_{\lambda,i}$ of the finite sum (which is also equal to the number of iterations since we treat the operator universally without assuming any specific  structure).

\begin{figure}[hpt!]
	\vspace{-3ex}
	\centering
	\includegraphics[width=\textwidth]{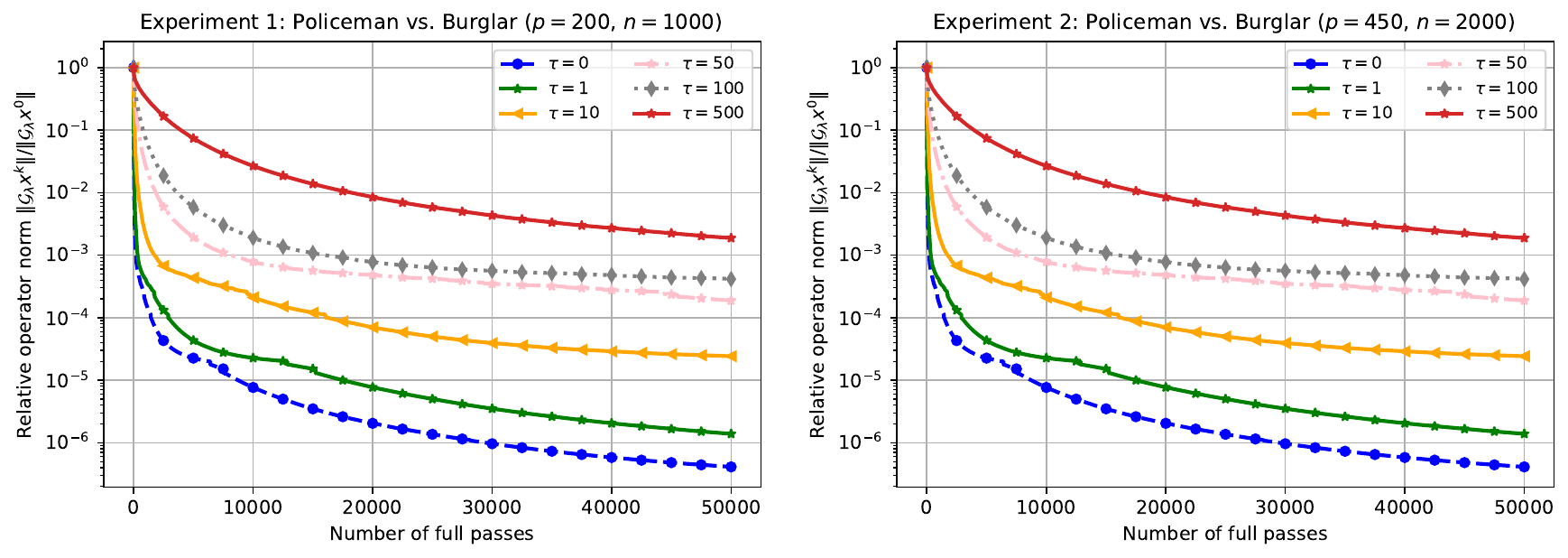}
	\vspace{-3ex}
	\caption{The performance (averaging over 5 problem instances) of \texttt{AFP-D} with different maximum delay values $\tau$ to solve \eqref{eq:bilinear_matrix_game} on 2 experiments.
	The numbers followed by the legends are the values of $\tau$ used in \texttt{AFP-D}.}
	\label{fig:matrix_games_tau_varies}
	\vspace{-3ex}
\end{figure}

\begin{figure}[hpt!]
	\vspace{-0ex}
	\centering
	\includegraphics[width=\textwidth]{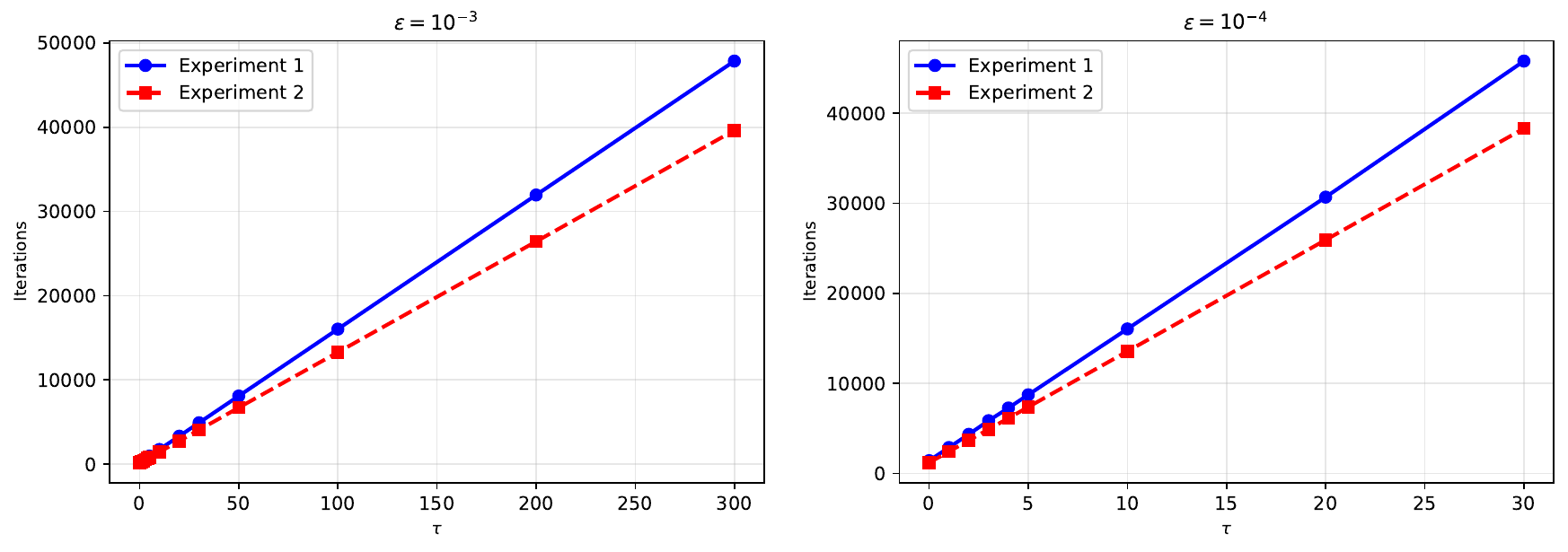}
	\vspace{-3ex}
	\caption{The number of iterations of \eqref{eq:iFKM4CE} (required to achieve $\norms{\Gc_{\lambda} x^k}/\norms{\Gc_{\lambda}x^0} \leq \epsilon$) vs. the maximum delay $\tau$.
	This plot verifies that the iteration-complexity depends \textbf{linearly} on $\tau$ as stated in Corollary~\ref{co:iFKM4CE_delayed_convergence_expectation}(iii).
	} 
	\label{fig:matrix_games_iter_vs_tau}
	\vspace{-2ex}
\end{figure}

We can see that the zero-latency variant ($\tau = 0$) provides a superior performance compared to all other variants, reaching the tolerance $10^{-6}$ after approximately $30{,}000$ iterations.
Moreover, as $\tau$ grows, performance degrades monotonically: larger $\tau$ yields slower decay of the relative operator norm.
This observation is consistent in both experiments and can be attributed to the dependence of the stepsize $\bar{\eta}$ as well as the upper bound of $\Expn{\norms{Gy^k}^2}$ on $\tau$ as proved in Theorem~\ref{th:iFKM4CE_convergence_expectation}.
Finally, the number of iterations that \eqref{eq:iFKM4CE} needs to achieve a given accuracy level $\epsilon$ is revealed in Figure~\ref{fig:matrix_games_iter_vs_tau}, which validates the theoretical finding about the \textit{\textbf{linear dependence}} of the \textit{iteration-complexity} on the maximum delay $\tau$.

Next, we test the effect of different strategies for constructing a delayed estimator $\widetilde{G}^k$ on the performance of \eqref{eq:iFKM4CE}.
In this test, we fix the maximum delay $\tau := 10$ for both \texttt{AFP-D} and \texttt{AFP-SD}.
Then, the numerical results are presented in Figure~\ref{fig:matrix_games_4variants}.
\begin{figure}[hpt!]
	\vspace{-0ex}
	\centering
	\includegraphics[width=\textwidth]{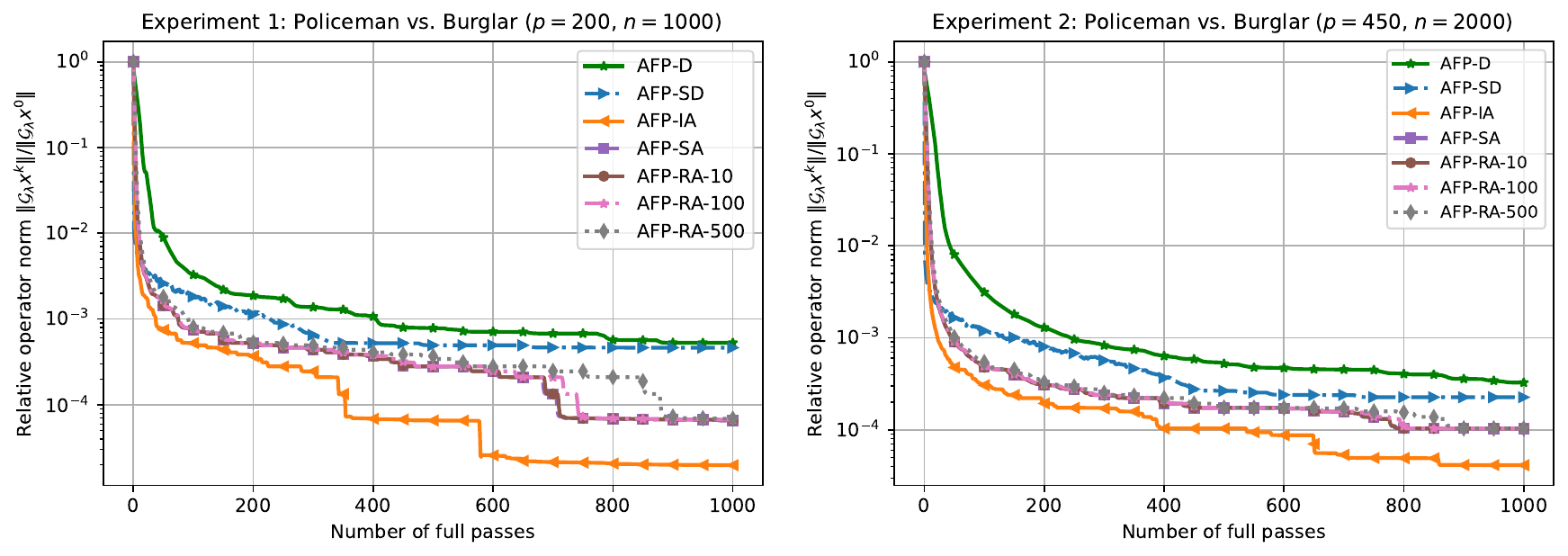}
	\vspace{-3ex}
	\caption{The performance (averaging over 5 problem instances) of 4 algorithms for solving \eqref{eq:bilinear_matrix_game} on 2 experiments.
	For \texttt{AFP-RA-$m$}, we run 3 cases with $m=10$, $m=100$, and $m=500$, respectively.
	}
	\label{fig:matrix_games_4variants}
	\vspace{-3ex}
\end{figure}

From Figure~\ref{fig:matrix_games_4variants}, we can see that \texttt{AFP-IA} outperforms the other variants when achieving the tolerance of magnitude $10^{-5}$ after nearly $400$ passes.
Among the remaining variants, \texttt{AFP-SA} and \text{AFP-RA-$m$} have a similar performance, as they all randomly evaluate some $G_i$'s to perform the iterate updates, while smaller values of $m$ yield slightly better performance due to their higher update frequency.
Finally, \texttt{AFP-SD} shows a slightly better performance than \texttt{AFP-D}, especially in the early phase when the mini-batch size is still below $n$.
In the later phase, when \texttt{AFP-SD} switches to use the full-batch operator, the two variants perform more similarly and ultimately reach the tolerance around $4\times10^{-4}$ after $1000$ passes.

\beforesubsec
\subsection{\textbf{Training neural networks via delayed inexact oracles}}\label{subsec:experiment2}
\aftersubsec
In this example, we evaluate different variants of our framework \eqref{eq:iFKM4CE} on a simple training task involving a shallow neural network. 
It is motivated by the following observations.
First, training large-scale machine learning models is typically carried out on parallel or distributed systems equipped with multiple CPU or GPU processors. 
In such environments, the use of delayed inexact oracles is unavoidable. 
Second, although the underlying optimization problem of this task is nonconvex, the network architecture contains only a few hidden layers, making its landscape relatively simple compared to deep learning models.
Third, several existing studies have shown that training neural networks using variants of Nesterov's momentum can improve performance, see, e.g., \cite{sutskever2013importance}.

\vspace{0.5ex}
\noindent\textbf{$\mathrm{(a)}$~\textit{Model and data.}} 
We train a shallow neural network to perform a well-known image classification on handwritten digits. 
The model is a fully connected network with two hidden layers containing $300$ and $100$ nodes, respectively, followed by a softmax cross-entropy loss layer. The network is implemented in PyTorch and trained on the well-known \texttt{MNIST} dataset with $N = 60{,}000$ samples \cite{MNIST}.
To simulate a privacy-sensitive environment and non-i.i.d. characteristics often encountered in practice, the data is partitioned so that each of the $10$ digit classes is stored exclusively on a single worker node (yielding $10$ workers in total). 
In this setting, the data on each worker remains confidential from both the central server and the other workers.

The first order optimality condition of this training problem can be written as $Gx^\star = 0$, a special case of \eqref{eq:CE}, where $Gx := \frac{1}{10}\sum_{i=1}^{10} G_i x$, $G_ix$ is the gradient of the loss function over the data of the $i$-th worker, and $x$ gathers all trainable parameters of the neural network.

\vspace{1ex}
\noindent\textbf{$\mathrm{(b)}$~\textit{Algorithms and parameters.}} 
We apply \eqref{eq:iFKM4CE} with inconsistent asynchronous oracle $\widetilde{G}^k$ defined in \eqref{eq:inconsistent_update}, denoted by \texttt{AFP-RA-$m$}, where $m \in \sets{1,2,5}$ is the number of active workers randomly sampled (without replacement) at each iteration.
We also implement the standard Gradient Descent method (\texttt{GD}) and the Incremental Aggregated Gradient method (\texttt{IAG}) for benchmarking.
The parameters of each algorithm are chosen as follows after tuning:
\begin{compactitem}[$\bullet$]
	\item For \texttt{GD} and \texttt{IAG}, the learning rates are set to $1.5\times10^{-2}$ and $1.25\times10^{-2}$, respectively.
	\item For \texttt{AFP-RA}: we choose $\gamma = 1.0$, $s = 1.05$, and $\eta = \frac{3}{1+\tau}$, where $\tau := 2\lceil \frac{N}{m} \rceil$.
\end{compactitem}

\vspace{1ex}
\noindent\textbf{$\mathrm{(c)}$~\textit{Numerical results.}} 
Figure \ref{fig:NN_mnist} compares the performance of different algorithmic variants on the training loss value and the test accuracy, each averaged over 10 runs.
\begin{figure}[hpt!]
\vspace{-3ex}
	\centering
	\includegraphics[width=\textwidth]{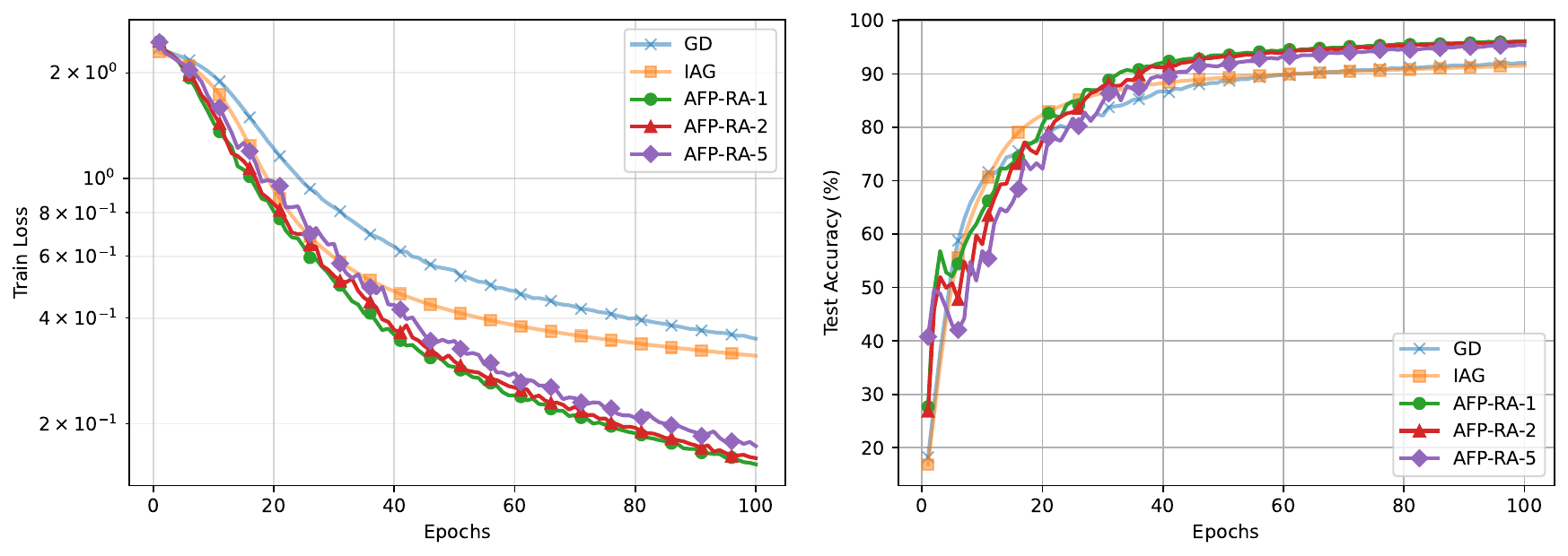}
	\vspace{-4ex}
	\caption{The behavior of the train loss and the test accuracy produced by 5 algorithms (three variants \texttt{AFP-RA} of \eqref{eq:iFKM4CE} and two competitors) for solving a neural network training problem using the \texttt{MNIST} dataset.}
	\label{fig:NN_mnist}
\vspace{-3ex}	
\end{figure}

We observe that the three variants of \eqref{eq:iFKM4CE} outperform the two competitors, while the ones with fewer active workers per iteration perform slightly better due to more updates per epoch.
This demonstrates the effectiveness of the accelerated methods over the non-accelerated counterparts, even when the updates are executed using delayed information.
Moreover, while \texttt{GD} typically performs well when the data is partitioned into i.i.d. subsets, it is not the case in such an extreme non-i.i.d. setting, which causes single-class-based updates, introducing significant bias to the network.
In contrast, the aggregated methods retain historical information from all classes -- despite the staleness -- thereby mitigating bias and yielding superior overall performance.

\appendix
\beforesec
\section{Appendix: Preliminary and Mathematical Tools}\label{apdx:useful_lemmas}
\aftersec
We first recall some technical results that will be used in our convergence analysis.
Then, we extend Lemma~3.5 from  \cite{feyzmahdavian2023asynchronous} to cover stochastic settings and also provide its proof.

%%% Lemma A.1.
\begin{lemma}[\cite{Bauschke2011}]\label{lem:tech_convergence1}
	Let $\sets{ a_k }$, $\sets{ b_k }$, and $\sets{ d_k }$ be sequences of real numbers. 
	Assume that $\sets{ a_k }$ is bounded from below, $\sets{b_k }$ and $\sets{ d_k }$ are nonnegative sequences such that $\sum_{k=0}^{\infty} d_k < \infty$, and
	\begin{equation*}
		a_{k+1} \leq a_k - b_k + d_k, \quad \forall k \geq 0.
	\end{equation*}
	Then, the following statements hold:
	\begin{compactitem}
		\item[$\mathrm{(i)}$] $\sets{b_k}$ is summable, namely $\sum_{k=0}^{\infty} b_k < \infty$.
		\item[$\mathrm{(ii)}$] $\sets{ a_k }$ is convergent.
	\end{compactitem}
\end{lemma}

%%% Lemma A.2.
\begin{lemma}[\cite{TranDinh2025a}]\label{lem:tech_convergence2}
	Let $\sets{ a_k }_{k\geq 0}$ be a sequence of nonnegative numbers and $\alpha > 0$ be a constant. 
	Assume that $\lim_{k \to \infty} k^{\alpha + 1} a_k$ exists and $\sum_{k=0}^\infty k^\alpha a_k < \infty$. Then, $\lim_{k\to \infty} k^{\alpha + 1} a_k = 0$.
\end{lemma}

%%% Lemma A.3.
\begin{lemma}[Supermartingale theorem, \cite{robbins1971convergence}]\label{lem:supermartingale}
	Let $\{X_k\}$, $\{ \alpha_k \}$, $\{V_k\}$, and $\{R_k\}$ be sequences of nonnegative integrable random variables on some arbitrary probability space and adapted to the filtration $\{\mathcal F_k\}$ with $\sum_{k=0}^\infty \alpha_k < +\infty$ and $\sum_{k=0}^\infty R_k < +\infty$ almost surely, and
	\begin{equation*}
		\Exp{X_{k+1} \mid \mathcal F_k} \leq (1 + \alpha_k) X_k - V_k + R_k, \quad \forall k \geq 0 \text{ almost surely.}
	\end{equation*}
	Then, $\{X_k\}$ almost surely converges and $\sum_{k=0}^\infty V_k < +\infty$ almost surely.
\end{lemma}

%%% Lemma A.4.
\begin{lemma}[Proposition 4.1 of \cite{davis2022variance}]\label{lem:Opial_NE}
	Suppose that $G$ in \eqref{eq:CE} is continuous.
	Let $\sets{x^k}$ be a sequence of random vectors such that for all $x^{\star} \in \zer{G}$, the sequence $\sets{\norms{x^k - x^{\star}}^2}$ almost surely converges to a $[0, \infty)$-valued random variable. 
	In addition, assume that $\sets{\norms{Gx^k}}$ also almost surely converges to zero.
	Then, $\sets{x^k}$ almost surely converges to a $\zer{G}$-valued random variable.
\end{lemma}

The following lemma is an extension of \cite[Lemma 3.5]{feyzmahdavian2023asynchronous} from a deterministic to a stochastic setting.

%%% Lemma A.2.
\begin{lemma}\label{le:Asyn_lemma_A2}
Let $\sets{V_k}$, $\sets{U_k}$, and $\sets{S_k}$ be three  sequences of nonnegative integrable random variables on some arbitrary probability space and adapted to the filtration $\{\mathcal F_k\}$ such that, almost surely
\begin{equation}\label{eq:recursive_relation2}
\Expn{ V_{k+1} \mid \Fc_k }  \leq V_k - U_k + \mu_k \sum_{l =  [k - \tau_k + 1]_{+}}^kS_l - \nu_kS_k + \omega_k, \quad \forall k \geq 0.
\end{equation}
where $\mu_k, \nu_k, \omega_k \geq 0$ and $0 \leq \tau_k \leq \tau$ for all $k\geq 0$.
Suppose further that $\sum_{l=0}^{\tau-1} \mu_{k+l} \leq \nu_k$.
\begin{compactitem}
\item[$\mathrm{(i)}$] Then, we have
\begin{equation}\label{eq:recursive_result2}
\sum_{k=0}^K\Expn{ U_k } \leq \Expn{ V_0 } + \sum_{k=0}^K\omega_k \quad \textrm{and} \quad \Expn{ V_{K+1} } \leq \Expn{ V_0 } + \sum_{k=0}^K\omega_k.
\end{equation}

\item[$\mathrm{(ii)}$] 
Let $D_0 := 0$ and, for $k \geq 1$, let
\begin{equation}\label{eq:Asyn_lemma_A2_p2}
	D_k := \sum_{l=0}^{k-1}\left[ \nu_l - \sum_{t=l}^{\min\sets{l+\tau-1, k-1}}\mu_t \right]S_l \qquad \text{and} \qquad \widetilde{V}_k := V_k + D_k.
\end{equation}
%%%
If $\sum_{k=0}^{\infty}\omega_k < +\infty$, then $\sets{\widetilde{V}_k}$ converges almost surely to a finite random variable $\widetilde{V}_{\infty}$ and $\sum_{k=0}^{\infty}U_k < +\infty$ almost surely.
Consequently, both $D_k$ and $V_k$ converge almost surely to finite random variables $D_{\infty}$ and $V_{\infty}$, respectively.
\end{compactitem}
\end{lemma}

%%% Proof of Lemma A.2.
\begin{proof}
(i)~Let $S_k = 0$ for $k\leq -1$.
Taking the full expectation on both sides of \eqref{eq:recursive_relation2}, we get
\begin{equation*} 
\Expn{ V_{k+1}  } + \Expn{ U_k }  \leq \Expn{ V_k }  + \sum_{l = [k - \tau_k + 1]_{+}}^k  \mu_k \Expn{ S_l } -  \nu_k \Expn{ S_k }  +  \omega_k. 
\end{equation*}
Summing up this inequality from $k=0$ to $K$, and noting that $b_0 = 1$,  $\Expn{S_k} = 0$ for $k \leq -1$, and $\tau_k \leq \tau$, we get
\begin{equation}\label{eq:lmA2_proof1} 
\hspace{-4ex}
\arraycolsep=0.2em
\begin{array}{lcl}
 \Expn{ V_{K+1} }  + \sum_{k=0}^K \Expn{ U_k }  &\leq & \Expn{ V_0 } + \sum_{k=0}^K \sum_{l = (k - \tau_k+1)_{+}}^k   \mu_k \Expn{ S_l }  - \sum_{k=0}^K \nu_k \Expn{ S_k } + \sum_{k=0}^K \omega_k \vspace{1ex}\\
& = & \Expn{ V_0 } + \sum_{k=0}^K \sum_{l = k - \tau_k +1}^k \mu_k \Expn{ S_l } - \sum_{k=0}^K  \nu_k \Expn{ S_k }  + \sum_{k=0}^K \omega_k \vspace{1ex}\\
& \leq & \Expn{ V_0 }  + \sum_{k=0}^K \sum_{l = k - \tau + 1}^k  \mu_k \Expn{ S_l } - \sum_{k=0}^K  \nu_k \Expn{ S_k }  + \sum_{k=0}^K \omega_k  \vspace{1ex}\\
& = & \Expn{ V_0 }  + \sum_{l=0}^{\tau - 1} \sum_{k=-l}^{K-l}  \mu_{k+l} \Expn{ S_k }  - \sum_{k=0}^K  \nu_k \Expn{ S_k }  + \sum_{k=0}^K \omega_k  \vspace{1ex}\\
& \leq & \Expn{ V_0 }  + \sum_{l=0}^{\tau-1} \sum_{k=0}^{K}  \mu_{k+l} \Expn{ S_k }  - \sum_{k=0}^K  \nu_k \Expn{ S_k } + \sum_{k=0}^K \omega_k  \vspace{1ex}\\
& = & \Expn{ V_0 }  - \sum_{k=0}^K \left( \nu_k - \sum_{l=0}^{\tau-1} \mu_{k+l}  \right) \Expn{ S_k } + \sum_{k=0}^K \omega_k,
\end{array}
\hspace{-8ex}
\end{equation}
where the last inequality follows from the fact that $S_k \geq 0$ for all $k\geq 0$ and $S_k = 0$ for $k < 0$.
We eventually get from \eqref{eq:lmA2_proof1} that
\begin{equation*} 
\arraycolsep=0.2em
\begin{array}{lcl}
\Expn{ V_{K+1} } + \sum_{k=0}^K \Expn{ U_k }  & \leq &  \Expn{ V_0 }  - \sum_{k=0}^K \left( \nu_k - \sum_{l=0}^{\tau-1} \mu_{k+l} \right) \Expn{ S_k } + \sum_{k=0}^K\omega_k.
\end{array}
\end{equation*}
Since $\sum_{l=0}^{\tau-1}\mu_{k+l} \leq \nu_k$, this estimate leads to $\Expn{ V_{K+1} } + \sum_{k=0}^K \Expn{ U_k } \leq \Expn{ V_0 } + \sum_{k=0}^K\omega_k$.
This inequality implies both estimates in \eqref{eq:recursive_result2}.

(ii)~Suppose we start with $D_0 = 0$ and aim to construct a sequence of nonnegative integrable random variables $\sets{D_k}$ such that
\begin{equation}\label{eq:lmA1_proof2}
	\arraycolsep=0.2em
	\begin{array}{lcl}
		D_{k+1} = D_k + \nu_k S_k - \mu_k \sum_{l =  [k - \tau + 1]_{+}}^k S_l.
	\end{array}
\end{equation}
By induction, this leads to 
\begin{equation*} 
\arraycolsep=0.2em
\begin{array}{lcl}
	D_k &:=& \sum_{r=0}^{k-1} \left[ \nu_r S_r - \mu_r \sum_{l=[r - \tau + 1]_+}^{r} S_l \right] \vspace{1ex}\\
	&=& \sum_{r=0}^{k-1} \nu_r S_r  - \sum_{r=0}^{k-1} \mu_r \sum_{l=[r - \tau + 1]_+}^{r} S_l  \vspace{1ex}\\
	&=& \sum_{l=0}^{k-1} \nu_l S_l  - \sum_{l=0}^{k-1} \sum_{r=l}^{\min\sets{l+\tau - 1, k-1}} \mu_r S_l  \vspace{1ex}\\
	&=& \sum_{l=0}^{k-1} \left[\nu_l - \sum_{r=l}^{\min\sets{l+\tau - 1, k-1}} \mu_r\right] S_l
\end{array}
\end{equation*}
as stated in \eqref{eq:Asyn_lemma_A2_p2}.
Using the condition $\sum_{l=0}^{\tau-1} \mu_{k+l} \leq \nu_k$, we can show that
\begin{equation*} 
	\arraycolsep=0.2em
	\begin{array}{lcl}
		\sum_{r=l}^{\min\sets{l+\tau-1, k-1}} \mu_r \leq \sum_{r=l}^{l+\tau-1} \mu_r = \sum_{t=0}^{\tau-1} \mu_{l+t} \leq \nu_l.
	\end{array}
\end{equation*}
Combining this fact and $S_l \geq 0$, we can show that $\sets{D_k}$ is nonnegative and monotonically increasing.

Now, taking the conditional expectation $\Expn{\cdot \mid \Fc_k}$ of \eqref{eq:lmA1_proof2} and then adding the result to \eqref{eq:recursive_relation2} yields
\begin{equation}\label{eq:lmA1_proof3}
\arraycolsep=0.2em
\begin{array}{lcl}
\mathbb{E}[\widetilde{V}_{k+1}\mid\mathcal{F}_k] &\leq & \widetilde{V}_k - U_k + \mu_k \sum_{l=[k-\tau_k+1]_+}^{k}S_l - \nu_k S_k + \omega_k  + \nu_k S_k - \mu_k \sum_{l =  [k - \tau_k+1]_{+}}^k S_l \vspace{1ex} \\
&\leq& \widetilde{V}_k - U_k + \mu_k \sum_{l=[k-\tau+1]_+}^{k}S_l + \omega_k - \mu_k \sum_{l =  [k - \tau_k+1]_{+}}^k S_l \vspace{1ex} \\
&=& \widetilde{V}_k - U_k + \omega_k.
\end{array}
\end{equation}
From \eqref{eq:lmA1_proof3} and $\sum_{k=0}^{\infty} \omega_k<\infty$, the classical Robbins-Siegmund theorem (i.e., Lemma~\ref{lem:supermartingale}) applies to $\widetilde{V}_k$, giving $ \widetilde{V}_k \to \widetilde{V}_\infty$ as $k\to\infty$ and $\sum_{k=0}^{\infty}U_k<\infty$ almost surely.

Finally, since $\lim_{k \to \infty}\widetilde{V}_k$ exists almost surely, there exists $M > 0$ such that $\widetilde{V}_k \leq M$ almost surely.
Moreover, since $V_k \geq 0$, we have
\begin{equation*} 
	\arraycolsep=0.2em
	\begin{array}{lcl}
		D_k = \widetilde{V}_k - V_k \leq \widetilde{V}_k \leq M \quad \text{almost surely}.
	\end{array}
\end{equation*}
Combining this fact and the monotonicity of $\sets{D_k}$, we conclude that $D_{\infty} := \lim_{k \to \infty} D_k$ exists almost surely.
Therefore, we conclude that
\begin{equation*} 
	\arraycolsep=0.2em
	\begin{array}{lcl}
		\lim_{k \to \infty} V_k = \lim_{k \to \infty} (\widetilde{V}_k - D_k) = \widetilde{V}_{\infty} - D_{\infty} := V_{\infty} \quad \text{almost surely}.
	\end{array}
\end{equation*}
This completes the proof.
\Eproof
\end{proof}
%%% End of Proof.

%%%%%%%%%%%%%%%%%%%%%%%%%%%%%%%%%%%%%%%%%%%%%%%%%%%%%%%%%%%%%%%%%%
%%% Appendix B.  The Proofs of Technical Results in Section 3.
%%%%%%%%%%%%%%%%%%%%%%%%%%%%%%%%%%%%%%%%%%%%%%%%%%%%%%%%%%%%%%%%%%
\beforesec
\section{Appendix: The Proofs of Technical Lemmas in Section~\ref{sec:iFKM_framework}}\label{sec:apdx:iFKM4CE_framework}
\aftersec
This appendix provides the missing proofs in Section~\ref{sec:iFKM_framework} of the main text.

\beforesubsec
\subsection{\textbf{The proof of Lemma~\ref{le:iFKM4CE_key_estimate2}: Lower bounding $\Lc_k - \Exps{k}{\Lc_{k+1}}$}}\label{subsec:apdx:le:iFKM4CE_key_estimate2}
\aftersubsec
To prove Lemma~\ref{le:iFKM4CE_key_estimate2}, we need the following technical lemma.

%%% Lemma 2.1.
\begin{lemma}\label{le:iFKM4CE_key_estimate1}
Suppose that Condition~\eqref{eq:G_cocoercivity} holds for \eqref{eq:CE}.
Let $\sets{(x^k, y^k, z^k)}$ be generated by \eqref{eq:iFKM4CE} such that $t_k > s$ and $t_k = t_{k-1} + 1$.
Then, for $\Lc_k$ defined by \eqref{eq:iFKM4CE_Lk_func}, we have
\begin{equation}\label{eq:iFKM4CE_key_est1}
\arraycolsep=0.2em
\begin{array}{lcl}
\Lc_k - \Lc_{k+1} & \geq & \frac{[2\beta - (1+c_1)\eta_k] t_k(t_k - s)}{2} \norms{Gy^{k+1} - Gy^k}^2  + \bar{\beta}t_k(t_k - s)\Dc(y^{k+1}, y^k)\vspace{1ex}\\
&&  + {~} \frac{a_k - \hat{a}_k}{2}\norms{Gy^k}^2  + s(s-1) \iprods{ Gy^k, y^k - x^{\star} } - \frac{\eta_kb_k}{2} \norms{e^k}^2,
\end{array}
\end{equation}
where $a_k$, $\hat{a}_k$, and $b_k$ are respectively given by
\begin{equation}\label{eq:iFKM4CE_coeffs}
\arraycolsep=0.2em
\left\{\begin{array}{lcl}
a_k & := &  \eta_{k-1}t_{k-1}(t_{k - 1} - s), \vspace{1ex}\\
\hat{a}_k &:= & \eta_k[(t_k - s)(t_k- 2s + 2\gamma_k) + c_2(s-\gamma_k) + (s-1)(1+\hat{c}_k)\gamma_k], \vspace{1ex}\\
b_k & := & \left(\frac{1}{c_1} + \frac{s-\gamma_k}{c_2}\right) t_k (t_k - s) + (s-1)\gamma_k \left(\frac{1+\hat{c}_k}{\hat{c}_k} +  \frac{\eta_{k-1}\gamma_{k-1}}{\eta_k\gamma_k - \eta_{k-1}\gamma_{k-1}}\right),
\end{array}\right.
\end{equation}
for any positive constants $c_1$, $c_2$, and $\hat{c}_k$.
\end{lemma}

%%% The Proof of Lemma 2.1.
\begin{proof}
First, recall $e^k := \widetilde{G}^{k} - Gy^k$ from \eqref{eq:oracle_error}.
Then, from the last and the first lines of \eqref{eq:iFKM4CE}, we have
\begin{equation*}
\arraycolsep=0.2em
\begin{array}{lcl}
	t_k y^{k+1} = s z^{k+1} + (t_k - s)x^{k+1} = sz^{k+1} + (t_k - s)y^k - \eta_k (t_k - s) (Gy^k + e^k).
\end{array}
\end{equation*}
Using $sz^{k+1} = sz^k - \gamma_k\eta_k(Gy^k + e^k)$ from the second line of \eqref{eq:iFKM4CE}, the last equation can be expressed into two different forms as
\begin{equation}\label{eq:key_lm1_proo1} 
\arraycolsep=0.2em
\begin{array}{lcl}
(t_k - s)t_k (y^{k+1} - y^k) &= & s(t_k - s)(z^{k+1} - y^k) - \eta_k (t_k - s)^2 (Gy^k + e^k) \vspace{1ex}\\
& = & s(t_k-s)(z^k - y^k) - \eta_k (t_k - s) (t_k-s + \gamma_k) (Gy^k + e^k), \vspace{1ex}\\
(t_k-s)t_k(y^{k+1} - y^k) &= & st_k(z^{k+1} - y^{k+1}) - \eta_k t_k(t_k - s) (Gy^k + e^k).
\end{array}
\end{equation}
Since $G$ satisfies Condition~\eqref{eq:G_cocoercivity} of Assumption~\ref{as:A1}, we can show that
\begin{equation*} 
\arraycolsep=0.2em
\begin{array}{lcl}
t_k(t_k-s) \iprods{Gy^{k+1}, y^{k+1} - y^k } - t_k(t_k-s)\iprods{Gy^k, y^{k+1} - y^k} &\geq &   \beta t_k(t_k-s) \norms{Gy^{k+1} - Gy^k}^2 \vspace{1ex}\\
&& + {~} \bar{\beta} t_k(t_k-s) \Dc(y^{k+1}, y^k).
\end{array}
\end{equation*}
Substituting the two expressions from \eqref{eq:key_lm1_proo1} into the last inequality, we can derive that
\begin{equation*} 
\arraycolsep=0.2em
\begin{array}{lcl}
\Tc_{[1]} &:= & st_k\iprods{Gy^{k+1}, z^{k+1} - y^{k+1}} - s(t_k-s)\iprods{Gy^k, z^k - y^k} \vspace{1ex}\\
& \geq & \beta t_k(t_k - s) \norms{Gy^{k+1} - Gy^k}^2 + \eta_k t_k(t_k-s)\iprods{Gy^{k+1}, Gy^k} + \bar{\beta} t_k(t_k-s)  \Dc(y^{k+1}, y^k)\vspace{1ex}\\
&& - {~} \eta_k (t_k - s)(t_k - s + \gamma_k ) \norms{Gy^k}^2 + \eta_k(t_k - s)\iprods{t_k Gy^{k+1} - (t_k - s + \gamma_k) Gy^k, e^k} \vspace{1ex}\\
& = & \frac{(2\beta - \eta_k)t_k(t_k - s)}{2}  \norms{Gy^{k+1} - Gy^k}^2 + \frac{\eta_kt_k(t_k - s)}{2}\norms{Gy^{k+1}}^2 +   \bar{\beta} t_k(t_k - s)  \Dc(y^{k+1}, y^k) \vspace{1ex}\\
&& - {~} \frac{\eta_k(t_k - s)(t_k- 2s + 2\gamma_k)}{2}\norms{Gy^k}^2 + \eta_k(t_k - s)\iprods{t_kGy^{k+1} - (t_k - s + \gamma_k)Gy^k, e^k}.
\end{array}
\end{equation*}
Here, we have used $2\iprod{Gy^{k+1}, Gy^k} = \norm{Gy^{k+1}}^2 + \norms{Gy^k}^2 - \norms{Gy^{k+1} - Gy^k}^2$ in the last equality.
By Young's inequality, for any $c_1 > 0$ and $c_2 > 0$, we can show that
\begin{equation*} 
\arraycolsep=0.2em
\begin{array}{lcl}
\Tc_{[2]} &:= & \eta_k(t_k - s)\iprods{t_kGy^{k+1} - (t_k - s + \gamma_k) Gy^k, e^k} \vspace{1ex}\\

&=& \eta_k t_k (t_k - s) \iprods{Gy^{k+1} - Gy^k, e^k} + (s - \gamma_k)\eta_k (t_k - s)\iprods{Gy^k, e^k}  \vspace{1ex}\\

&\geq& - \frac{c_1 \eta_k t_k (t_k - s)}{2} \norms{Gy^{k+1} - Gy^k}^2 - \frac{\eta_k t_k (t_k - s)}{2c_1} \norms{e^k}^2   \vspace{1ex}\\
&& - {~} \frac{c_2 (s - \gamma_k) \eta_k (t_k - s)}{2t_k} \norms{Gy^k}^2 - \frac{(s - \gamma_k)\eta_k t_k (t_k - s)}{2c_2}\norms{e^k}^2 \vspace{1ex}\\

&\geq& - \frac{c_1 \eta_k t_k (t_k - s)}{2} \norms{Gy^{k+1} - Gy^k}^2 - \frac{c_2(s - \gamma_k) \eta_k}{2} \norms{Gy^k}^2 - \left(\frac{1}{c_1} + \frac{s-\gamma_k}{c_2}\right)\frac{\eta_k t_k (t_k - s)}{2}\norms{e^k}^2,
\end{array}
\end{equation*}
where we have used $\frac{t_k - s}{t_k} \leq 1$ in the last inequality.
Substituting $\Tc_{[2]}$ into $\Tc_{[1]}$, and rearranging the result with the fact that $t_k = t_{k-1} + 1$, we get
\begin{equation*} 
\arraycolsep=0.2em
\begin{array}{lcl}
\Tc_{[3]} &:= & s t_{k-1}\iprods{Gy^k, y^k - z^k} - s t_k\iprods{Gy^{k+1}, y^{k+1} - z^{k+1}} \vspace{1ex}\\
& \geq & \frac{[2\beta - (1+c_1)\eta_k] t_k(t_k - s)}{2}  \norms{Gy^{k+1} - Gy^k}^2 + \frac{\eta_kt_k(t_k - s)}{2}\norms{Gy^{k+1}}^2 +   \bar{\beta} t_k(t_k - s)  \Dc(y^{k+1}, y^k) \vspace{1ex}\\
&& + {~} s(s-1)\iprods{Gy^k, y^k - z^k} - \frac{\eta_k[(t_k - s)(t_k- 2s + 2\gamma_k) + c_2(s-\gamma_k)]}{2}\norms{Gy^k}^2 - \frac{\eta_k t_k (t_k - s)}{2}\left(\frac{1}{c_1} + \frac{s-\gamma_k}{c_2}\right)\norms{e^k}^2.
\end{array}
\end{equation*}
Next, using again Young's inequality and $z^{k+1} - z^k = -\frac{\gamma_k\eta_k}{s}(Gy^k + e^k)$ from the second line of \eqref{eq:iFKM4CE}, for any $c_k > 0$ and $\hat{c}_k > 0$, we have 
\begin{equation*} 
\arraycolsep=0.2em
\begin{array}{lcl}
\Tc_{[4]} &:= & \frac{s^2(s-1)}{2\eta_k\gamma_k}\norms{z^k - x^{\star}}^2 - \frac{s^2(s-1)}{2\eta_k\gamma_k}\norms{z^{k+1} - x^{\star}}^2 \vspace{1ex}\\
& = & - \frac{s^2(s-1)}{\eta_k\gamma_k}\iprods{z^{k+1} - z^k, z^k - x^{\star}} - \frac{s^2(s-1)}{2\eta_k\gamma_k}\norms{z^{k+1} - z^k}^2 \vspace{1ex}\\
& = & s(s-1)\iprods{Gy^k + e^k, z^k - x^{\star}} -  \frac{(s-1)\eta_k\gamma_k}{2}\norms{Gy^k + e^k }^2 \vspace{1ex}\\
& \geq & s(s-1) \iprods{Gy^k, z^k - x^{\star}} - \frac{s^2(s-1)c_k}{2}\norms{z^k - x^{\star}}^2 - \frac{s-1}{2c_k}\norms{e^k}^2 \vspace{1ex}\\
&& - {~}  \frac{(s-1)(1+\hat{c}_k)\eta_k\gamma_k}{2}\norms{Gy^k}^2 -  \frac{(s-1)(1+\hat{c}_k)\eta_k\gamma_k}{2\hat{c}_k}\norms{e^k}^2.
\end{array}
\end{equation*}
Rearranging this inequality and choosing $c_k := \frac{\eta_k\gamma_k - \eta_{k-1}\gamma_{k-1}}{\gamma_{k-1}\gamma_k\eta_{k-1}\eta_k}$, we get
\begin{equation*} 
\arraycolsep=0.2em
\begin{array}{lcl}
\hat{\Tc}_{[4]} &:= & \frac{s^2(s-1)}{2\eta_{k-1}\gamma_{k-1}}\norms{z^k - x^{\star}}^2 - \frac{s^2(s-1)}{2\eta_k\gamma_k}\norms{z^{k+1} - x^{\star}}^2 \vspace{1ex}\\
& \geq & s(s-1) \iprods{Gy^k, z^k - x^{\star}}  -  \frac{(s-1)(1+\hat{c}_k)\eta_k\gamma_k}{2}\norms{Gy^k}^2 -  \frac{(s-1)\eta_k\gamma_k}{2} \big[ \frac{1+\hat{c}_k}{\hat{c}_k} +  \frac{\eta_{k-1}\gamma_{k-1}}{\eta_k\gamma_k - \eta_{k-1}\gamma_{k-1}} \big] \norms{e^k}^2.
\end{array}
\end{equation*}
Combining $\Tc_{[3]}$ and $\hat{\Tc}_{[4]}$, we obtain 
\begin{equation*} 
\arraycolsep=0.2em
\begin{array}{lcl}
\Tc_{[5]} &:= & s t_{k-1}\iprods{Gy^k, y^k - z^k} + \frac{s^2(s-1)}{2\eta_{k-1}\gamma_{k-1}}\norms{z^k - x^{\star}}^2 \vspace{1ex}\\
&& - {~} s t_k\iprods{Gy^{k+1}, y^{k+1} - z^{k+1}} - \frac{s^2(s-1)}{2\eta_k\gamma_k}\norms{z^{k+1} - x^{\star}}^2 \vspace{1ex}\\

& \geq & \frac{[2\beta - (1+c_1)\eta_k] t_k(t_k - s)}{2}  \norms{Gy^{k+1} - Gy^k}^2 + \frac{\eta_kt_k(t_k - s)}{2}\norms{Gy^{k+1}}^2 +   \bar{\beta} t_k(t_k - s)  \Dc(y^{k+1}, y^k) \vspace{1ex}\\
&& + {~} s(s-1)\iprods{Gy^k, y^k - x^\star} - \frac{\eta_k}{2}[(t_k - s)(t_k- 2s + 2\gamma_k) + c_2(s-\gamma_k) + (s-1)(1+\hat{c}_k)\gamma_k]\norms{Gy^k}^2 \vspace{1ex}\\
&& - {~} \frac{\eta_k}{2}\left[\left(\frac{1}{c_1} + \frac{s-\gamma_k}{c_2}\right) t_k (t_k - s) + (s-1)\gamma_k \left(\frac{1+\hat{c}_k}{\hat{c}_k} +  \frac{\eta_{k-1}\gamma_{k-1}}{\eta_k\gamma_k - \eta_{k-1}\gamma_{k-1}}\right)\right]\norms{e^k}^2.
\end{array}
\end{equation*}
Rearranging this inequality and using $a_k$, $\hat{a}_k$, and $b_k$ as in \eqref{eq:iFKM4CE_coeffs}, and the definition of $\Lc_k$ from \eqref{eq:iFKM4CE_Lk_func},  we finally get \eqref{eq:iFKM4CE_key_est1} from the last expression $\Tc_{[5]}$.
\Eproof
\end{proof}
%%% End of Proof.

Now, applying Lemma~\ref{le:iFKM4CE_key_estimate1} above, we are ready to prove Lemma~\ref{le:iFKM4CE_key_estimate2}.

%%% The proof of Lemma 2.2
\begin{proof}$($\textbf{The proof of Lemma~\ref{le:iFKM4CE_key_estimate2}}$)$
	We first choose $\hat{c}_k := 1$, $\gamma_k := \gamma \in (0,1]$, and $\eta_k = \frac{\eta t_k}{2(t_k - s)}$ for some $\eta > 0$. 
	Then, we have $\frac{\eta}{2} \leq \eta_k \leq \frac{\eta t_0}{2(t_0 - s)} \leq \eta$ due to $t_0 > 2s$.
	Next, we can show that
	\begin{equation*}
		\arraycolsep=0.2em
		\begin{array}{lcl}
			a_k - \hat{a}_k %& = & \eta_{k-1}t_{k-1}(t_{k - 1} - s) - \eta_k[(t_k - s)(t_k- 2s + 2\gamma_k) + c_2(s-\gamma_k) + (s-1)(1+\hat{c}_k)\gamma_k] \vspace{1ex}\\
%			&=& \frac{\eta t_{k-1}^2}{2} - \frac{\eta t_k}{2(t_k - s)}[(t_k - s)(t_k- 2s + 2\gamma) + c_2(s-\gamma) + 2(s-1)\gamma] \vspace{1ex}\\
%			&=& \frac{\eta (t_k - 1)^2}{2} - \frac{\eta t_k [t_k - 2(s-\gamma)]}{2} - \frac{\eta [c_2(s-\gamma) + 2(s-1)\gamma]t_k}{2(t_k - s)} \vspace{1ex}\\
			&=& \eta (s - \gamma - 1) t_k + \frac{\eta}{2} - \frac{\eta [c_2(s-\gamma) + 2(s-1)\gamma]t_k}{2(t_k - s)} \vspace{1ex}\\
			&\geq& \eta (s - \gamma - 1) t_k + \frac{\eta}{2} - \frac{\eta [c_2(s-\gamma) + 2(s-1)\gamma]t_0}{2(t_0 - s)} \vspace{1ex}\\
			&\geq& \frac{\eta}{2} (s - \gamma - 1) t_k > 0,
		\end{array}
	\end{equation*}
	provided that $s > 1 + \gamma$.
	Moreover, from \eqref{eq:iFKM4CE_coeffs}, we can also prove that 
	\begin{equation*} 
		\arraycolsep=0.2em
		\begin{array}{lcl}
			b_k & := & \left(\frac{1}{c_1} + \frac{s-\gamma_k}{c_2}\right) t_k (t_k - s) + (s-1)\gamma_k \left(\frac{1+\hat{c}_k}{\hat{c}_k} +  \frac{\eta_{k-1}\gamma_{k-1}}{\eta_k\gamma_k - \eta_{k-1}\gamma_{k-1}}\right) \vspace{1ex}\\
			
			&=& \left(\frac{1}{c_1} + \frac{s-\gamma}{c_2}\right) t_k (t_k - s) + (s-1)\gamma \left[2 - \frac{t_{k-1}(t_k - s)}{s}\right] \vspace{1ex}\\
			 
			&\leq& \left(\frac{1}{c_1} + \frac{s-\gamma}{c_2}\right) t_k(t_k-s),
		\end{array}
	\end{equation*}
	where the last inequality holds due to $t_{k-1} > 2s$.

	Finally, substituting the last two estimates of $a_k - \hat{a}_k$ and $b_k$ into \eqref{eq:iFKM4CE_key_est1} of Lemma~\ref{le:iFKM4CE_key_estimate1}, and noticing that $0 < \eta_k \leq \eta$, we obtain
	\begin{equation*}
		\arraycolsep=0.2em
		\begin{array}{lcl}
			\Lc_k - \Lc_{k+1} & \geq & \frac{[2\beta - (1+c_1)\eta] t_k(t_k - s)}{2} \norms{Gy^{k+1} - Gy^k}^2  + \bar{\beta}t_k(t_k - s)\Dc(y^{k+1}, y^k)\vspace{1ex}\\
			&&  + {~} \frac{\eta (s - \gamma - 1) t_k}{4}\norms{Gy^k}^2  + s(s-1) \iprods{ Gy^k, y^k - x^{\star} } - \frac{\eta}{2} \left(\frac{1}{c_1} + \frac{s-\gamma}{c_2}\right) t_k(t_k-s) \norms{e^k}^2,
		\end{array}
	\end{equation*}
	which is exactly \eqref{eq:iFKM4CE_key_est2}.
	\Eproof
\end{proof}
%%% End of Proof.

%%% The proof of Lemma 2.3
\beforesubsec
\subsection{\textbf{The proof of Lemma~\ref{le:iFKM4CE_lower_bound_of_Lk}: Lower bounding $\Lc_k$}}\label{subsec:apdx:le:iFKM4CE_lower_bound_of_Lk}
\aftersubsec
	Since $Gx^{\star} = 0$ for any $x^{\star} \in \zer{G}$, we obtain from Condition~\eqref{eq:G_cocoercivity} of Assumption~\ref{as:A1}(b) that $\iprods{Gy^k, y^k - x^{\star}} \geq \beta \norms{Gy^k}^2$.
	Using this relation, we can show from \eqref{eq:iFKM4CE_Lk_func} that 
	\begin{equation}\label{eq:iFKM4CE_iasyn_key_est2_proof1}
	\arraycolsep=0.2em
	\begin{array}{lcl}
		\Lc_k &=& \frac{a_k}{2} \norms{Gy^k}^2 + st_{k-1}\iprods{Gy^k, y^k - z^k} + \frac{s^2(s-1)}{2\eta_{k-1}\gamma_{k-1}}\norms{z^k - x^{\star}}^2 \vspace{1ex}\\
%		&=& \Big( \frac{a_k}{2} - \frac{\eta t_{k-1}^2}{8} \Big) \norms{Gy^k}^2 + st_{k-1} \iprods{Gy^k, y^k - x^{\star}} + \left[ \frac{s^2(s-1)}{2\eta_{k-1}\gamma_{k-1}} - \frac{2s^2}{\eta} \right]\norms{z^k - x^{\star}}^2 \vspace{1ex}\\
%		&& + {~} \frac{\eta t_{k-1}^2}{8}\norms{Gy^k}^2 - st_{k-1} \iprods{Gy^k, z^k - x^{\star}} + \frac{2s^2}{\eta}\norms{z^k - x^{\star}}^2 \vspace{1ex}\\
		&=& \Big( \frac{a_k}{2} - \frac{\eta t_{k-1}^2}{8} \Big) \norms{Gy^k}^2 + st_{k-1} \iprods{Gy^k, y^k - x^{\star}} + \left[ \frac{s^2(s-1)}{2\eta_{k-1}\gamma_{k-1}} - \frac{2s^2}{\eta} \right]\norms{z^k - x^{\star}}^2 \vspace{1ex}\\
		&& + {~} \frac{1}{8\eta} \norms{4s(z^k - x^{\star}) - \eta t_{k-1}Gy^k}^2 \vspace{1ex}\\
		&\geq& \frac{4a_k - \eta t_{k-1}^2  + 8s \beta t_{k-1}}{8} \norms{Gy^k}^2+ \left[ \frac{s^2(s-1)}{2\eta_{k-1}\gamma_{k-1}} - \frac{2s^2}{\eta} \right]\norms{z^k - x^{\star}}^2.
	\end{array}
	\end{equation}
	Using $a_k$ in \eqref{eq:iFKM4CE_key_est1}, $\eta_k$ in \eqref{eq:iFKM4CE_param1}, and $s > 1 + 3\gamma$, we can also show that
	\begin{equation*}
		\arraycolsep=0.2em
		\begin{array}{lcl}
			A_k &:=& 4a_k - \eta t_{k-1}^2  + 8\beta s t_{k-1} 
			= \eta t_{k-1}^2 + 8\beta s t_{k-1} 
			\geq \eta t_{k-1}^2, \vspace{2ex}\\
			
			B_k &:=& \frac{s^2(s-1)}{2\eta_{k-1}\gamma_{k-1}} - \frac{2s^2}{\eta} = \frac{s^2}{\eta} \left[ \frac{(s-1)(t_{k-1} - s)}{\gamma t_{k-1}} - 2 \right]
			\geq \frac{s^2}{\eta} \left[ \frac{3(t_{k-1} - s)}{t_{k-1}} - 2 \right]
			= \frac{s^2(t_{k-1} - 3s)}{t_{k-1}}.
		\end{array}
	\end{equation*}
	Substituting these expressions into \eqref{eq:iFKM4CE_iasyn_key_est2_proof1}, we obtain \eqref{eq:iFKM4CE_iasyn_key_est2}.
%%\end{proof}
\Eproof
%%% End of Proof.

%%% Proof of Lemma 3.4.
\beforesubsec
\subsection{\textbf{Proof of Lemma~\ref{le:iFKM4CE_descent_property}: Lower bounding $V_k - \Expsn{k}{V_{k+1}}$}}\label{subsec:apdx:le:iFKM4CE_descent_property}
\aftersubsec
Since $x^{\star} \in \zer{G}$, by Condition~\eqref{eq:G_cocoercivity} of Assumption~\ref{as:A1}(b), we have $\iprods{Gy^k, y^k - x^{\star}} \geq \beta\norms{Gy^k}^2$.
Substituting this inequality into \eqref{eq:iFKM4CE_key_est2}, then choosing $c_1 = c_2 := \tau > 0$ and taking the conditional expectation $\Expsn{k}{\cdot}$ on both sides of  the result, we obtain
\begin{equation}\label{eq:iFKM4CE_lm34_proof1}
	\hspace{-1ex}
	\arraycolsep=0.2em
	\begin{array}{lcl}
		\Lc_k - \Expsn{k}{\Lc_{k+1}} & \geq & \frac{[2\beta - (1+\tau)\eta] t_k(t_k - s)}{2} \Expsn{k}{  \norms{Gy^{k+1} - Gy^k}^2 }  + \bar{\beta}t_k(t_k - s)\Expsn{k}{ \Dc(y^{k+1}, y^k) }\vspace{1ex}\\
		&&  + {~} \frac{\eta (s - \gamma - 1) t_k + 4\beta s(s-1)}{4}\norms{Gy^k}^2 - \frac{\Lambda\eta}{2\tau} t_k(t_k-s) \norms{e^k}^2,
	\end{array}
	\hspace{-1ex}
\end{equation}
where $\Lambda := 1 + s - \gamma$.

Since we choose $t_k = k+3s+\tau$, it is obvious to show that $t_k(t_k-s) \leq \big( \alpha - \frac{s+1}{2s^2}\big) t_{k-1}(t_{k-1}-s)$, where $\alpha := \frac{(s+1)^2}{s^2}$.
%\footnote{\ncmt{
%This estimate still holds for $t_k = k + 2s + \tau$ with $\tau \geq 1$. 
%Indeed, we have
%\begin{equation*}
%\arraycolsep=0.2em
%\begin{array}{lcl}
%	\frac{t_k(t_k - s)}{t_{k-1}(t_{k-1} - s)} &=& \frac{k + 2s + \tau}{k + 2s + \tau - 1} \cdot \frac{k + s + \tau}{k + s + \tau - 1}
%	= \left( 1 + \frac{1}{k + 2s + \tau - 1}\right)\left( 1 + \frac{1}{k + s + \tau - 1}\right)
%	\stackrel{(k \geq 0)}{\leq} \left( 1 + \frac{1}{2s + \tau - 1}\right)\left( 1 + \frac{1}{s + \tau - 1}\right) \vspace{1ex}\\
%	&\stackrel{(\tau \geq 1)}{\leq}& \left( 1 + \frac{1}{2s}\right)\left( 1 + \frac{1}{s}\right) = \alpha - \frac{s+1}{2s^2}.
%\end{array}
%\end{equation*}
%}}
Using this inequality, we can show from the second line of \eqref{eq:iFKM_error_bound} that
\begin{equation*}
\arraycolsep=0.2em
\begin{array}{lcl}
t_k(t_k-s)\Expsn{k}{\Delta_k}  & \leq & \frac{\alpha(1-\kappa)}{1 - (1-\kappa)\alpha}\big[ t_{k-1}(t_{k-1}-s)\Delta_{k-1}  -  t_k(t_k-s)\Expsn{k}{\Delta_k} \big] \vspace{1ex}\\
&& {~} + \frac{ t_k(t_k-s)}{1 - (1-\kappa)\alpha} W_k + \frac{\delta_k}{1 - (1-\kappa)\alpha} - \frac{(1-\kappa)(s+1)t_{k-1}(t_{k-1}-s)}{2s^2\theta }\Delta_{k-1},
\end{array}
\end{equation*}
where $W_k := \sum_{l=[k-\tau_k + 1]_{+}}^k\big[\Theta\norms{Gy^l - Gy^{l-1}}^2 + \hat{\Theta}\Dc(y^l, y^{l-1})\big]$, provided that %$\frac{3s + 1}{(s+1)(2s+1)} < \kappa \leq 1$.
$\frac{2s + 1}{(s+1)^2}  < \kappa \leq 1$.
%\footnote{\ncmt{
%We need $1 - (1-\kappa)\alpha > 0$, where $\alpha := \frac{(s+1)^2}{s^2}$. This condition holds if $1 \geq \kappa > 1 - \frac{1}{\alpha} = 1 - \frac{s^2}{(s+1)^2} = \frac{2s + 1}{(s+1)^2}$.
%}}
Substituting $\Expsn{k}{\norms{e^k}^2} \leq \Expsn{k}{\Delta_k}$ from the first line of \ref{de:error_bound_cond} into \eqref{eq:iFKM4CE_lm34_proof1} and then using the last inequality, we can show that
\begin{equation}\label{eq:iFKM4CE_lm34_proof2}
	\hspace{-10ex}
	\arraycolsep=0.2em
	\begin{array}{lcl}
		\Lc_k - \Expsn{k}{\Lc_{k+1}} & \geq &  \frac{\eta (s - \gamma - 1) t_k + 4\beta s(s-1)}{4} \norms{Gy^k}^2  + \frac{\Lambda\eta(1 - \kappa)(s+1)t_{k-1}(t_{k-1}-s)}{4\tau s^2\theta }\Delta_{k-1} \vspace{1ex}\\  
		&&  + {~} \frac{[2\beta - (1+\tau)\eta] t_k(t_k - s)}{2} \Expsn{k}{  \norms{Gy^{k+1} - Gy^k}^2 } + \bar{\beta}t_k(t_k - s) \Expsn{k}{ \Dc(y^{k+1}, y^k) }  \vspace{1ex}\\
		&&  - {~} \frac{\Lambda\eta\alpha(1-\kappa)}{2\tau[1 - (1-\kappa)\alpha ]} \Big[ t_{k-1}(t_{k-1}-s)\Delta_{k-1}  -  t_k(t_k-s)\Expsn{k}{\Delta_k} \Big]  - \frac{\Lambda\eta\delta_k}{2\tau\theta } \vspace{1ex}\\
		&& - {~} \frac{ \Lambda \eta t_k(t_k-s)}{2\tau\theta  } \sum_{l=[k-\tau_k + 1]_{+}}^k\big[\Theta\norms{Gy^l - Gy^{l-1}}^2 + \hat{\Theta}\Dc(y^l, y^{l-1})\big] \vspace{1ex}\\
		&&  - {~} \frac{[2\beta - (1+\tau)\eta] t_{k-1}(t_{k-1} - s)}{2} \norms{Gy^k - Gy^{k-1}}^2 - \bar{\beta}t_{k-1}(t_{k-1} - s) \Dc(y^{k}, y^{k-1})  \vspace{1ex}\\
		&& + {~} \frac{\beta t_{k-1}(t_{k-1} - s)}{2} \norms{Gy^k - Gy^{k-1}}^2 + \frac{\bar{\beta}t_{k-1}(t_{k-1} - s)}{2} \Dc(y^{k}, y^{k-1})  \vspace{1ex}\\
		&& + {~} \frac{[\beta - (1+\tau)\eta] t_{k-1}(t_{k-1} - s)}{2} \norms{Gy^k - Gy^{k-1}}^2 + \frac{\bar{\beta}t_{k-1}(t_{k-1} - s)}{2} \Dc(y^{k}, y^{k-1}).
	\end{array}
	\hspace{-8ex}
\end{equation}
Now, if we define 
\begin{equation}\label{eq:iFKM4CE_lm34_proof3}
\left\{\arraycolsep=0.2em
\begin{array}{lcl}
	V_k &:=& \Lc_k + \frac{[2\beta - (1+\tau)\eta] t_{k-1}(t_{k-1} - s)}{2} \norms{Gy^k - Gy^{k-1}}^2 + \bar{\beta}t_{k-1}(t_{k-1} - s) \Dc(y^{k}, y^{k-1}) \vspace{1ex}\\
	&& + {~} \frac{\Lambda\eta\alpha(1-\kappa)}{2\tau[1 - (1-\kappa)\alpha ]} t_{k-1}(t_{k-1}-s)\Delta_{k-1}, \vspace{1.5ex}\\
	
	U_k &:=& \frac{\eta (s - \gamma - 1) t_k + 4\beta s(s-1)}{4} \norms{Gy^k}^2  + \frac{\Lambda\eta(1 - \kappa)(s+1)t_{k-1}(t_{k-1}-s)}{4\tau s^2\theta }\Delta_{k-1} \vspace{1ex}\\
	&& + {~} \frac{\beta t_{k-1}(t_{k-1}-s)}{2} \norms{Gy^k - Gy^{k-1}}^2 + \frac{\bar{\beta}t_{k-1} (t_{k-1}-s)}{2}\Dc(y^k, y^{k-1}), \vspace{1.5ex}\\
	
	S_k &:=& \Theta\norms{Gy^k - Gy^{k-1}}^2 + \hat{\Theta}\Dc(y^k, y^{k-1}), \vspace{1.5ex}\\
	
	\omega_k &:=& \frac{\Lambda\eta\delta_k}{2\tau\theta }, \quad 
	\mu_k := \frac{\Lambda\eta t_k(t_k-s)}{2\tau\theta  }, \quad 
	\nu_k := \frac{1}{2} t_{k-1} (t_{k-1} - s)\min\set{\frac{\beta  - (1+\tau)\eta}{\Theta}, \frac{\bar{\beta}}{\hat{\Theta}}},
\end{array}\right.\hspace{-4ex}
\end{equation}
then we can derive from \eqref{eq:iFKM4CE_lm34_proof2} that
\begin{equation*}
\arraycolsep=0.2em
\begin{array}{lcl}
\Expsn{k}{ V_{k+1} } &\leq & V_k - U_k + \mu_k\sum_{l=[k-\tau_k+1]_{+} }^kS_l - \nu_k  S_k  + \omega_k,
\end{array}
\end{equation*}
which is exactly \eqref{eq:iFKM4CE_le_key_est}.
Note that the quantities in \eqref{eq:iFKM4CE_lm34_proof3} are exactly \eqref{eq:iFKM4CE_le_Vk} and \eqref{eq:iFKM4CE_le_Uk}.
\Eproof
%%%\end{proof}
%%% End of Proof.

\beforesubsec
\subsection{\textbf{Supporting lemmas for convergence analysis}}\label{subsec:apdx:supporting_lemmas}
\aftersubsec
This appendix establishes a supporting lemma for our convergence analysis in the main text.

\begin{lemma}\label{le:appendix_convergence_series}
	%Under the same conditions and settings as in Theorem~\ref{th:iFKM4CE_convergence_expectation}, 
	Given a bounded sequence $\sets{\tau_k}$ with $0 \leq \tau_k \leq \tau$ for some fixed $\tau \geq 1$, $t_k = t_{k-1} + 1$, and a nonnegative random sequence $\sets{Z_k}$, if we have $\sum_{k=0}^\infty t_k (t_k - s) Z_k < +\infty$ almost surely, then the following bound also holds almost surely
	\begin{equation}
		\arraycolsep=0.2em
		\begin{array}{lcl}
			\sum_{k=0}^\infty t_k(t_k - s) \sum_{l = [k - \tau_k + 1]_+}^k Z_l < +\infty.
		\end{array}
	\end{equation}
\end{lemma}
\begin{proof}
	Since $\tau_k \leq \tau$, using Fubini-Tonelli's theorem and $\sum_{k=0}^\infty t_k(t_k - s) Z_k < +\infty$ almost surely, we have
	\begin{equation*}
		\arraycolsep=0.2em
		\begin{array}{lcl}
			A &:=& \sum_{k=0}^\infty t_k(t_k - s) \sum_{l = [k - \tau_k + 1]_+}^k Z_l \vspace{1ex}\\
			&\leq& \sum_{k=0}^\infty t_k(t_k - s) \sum_{l = [k - \tau + 1]_+}^k Z_l \vspace{1ex}\\
			&=& \sum_{l=0}^\infty Z_l \sum_{k=l}^{l+\tau-1} t_k(t_k - s) \vspace{1ex}\\
			&=& \sum_{l=0}^\infty Z_l \sum_{j=0}^{\tau-1} t_{j+l} (t_{j+l} - s) \vspace{1ex}\\ 
			&\leq& \tau \sum_{l=0}^\infty t_{l + \tau - 1} (t_{l + \tau - 1} - s) Z_l \vspace{1ex}\\ 
			&<& +\infty \quad \text{almost surely},
		\end{array}
	\end{equation*}	
	where the last inequality holds due to $\lim_{l\to\infty} \frac{t_{l + \tau - 1} (t_{l + \tau - 1} - s)}{t_l (t_l - s)} = 1$.
\Eproof
\end{proof}
%%%%%%%%%%%%%%%%%%%%%%%%%%%%%%%%%%%%%%%%%%%%%%%
%+ References.
%%%%%%%%%%%%%%%%%%%%%%%%%%%%%%%%%%%%%%%%%%%%%%%
\beforesec
\bibliographystyle{plain}
%\bibliography{/Users/quoc/Dropbox/E-Books/tran_bibtex_2025}
%\bibliography{/Users/quoc/Dropbox/E-Books/tran_bibtex_new}

\end{document}